\documentclass[10pt,a4paper]{article}
\PassOptionsToPackage{linktocpage}{hyperref}
\usepackage[latin1]{inputenc}
\usepackage[T1]{fontenc}
\usepackage{indentfirst}
\usepackage{amsmath,amssymb,amsthm,amsfonts}
\usepackage{calligra,mathrsfs}
\usepackage{graphicx,color}
\usepackage{xcolor}
\usepackage[all,cmtip]{xy}
\usepackage{bbm}
\usepackage{hyperref}
\usepackage{mathtools}
\usepackage{verbatim}
\usepackage{stmaryrd}
\usepackage{enumitem}
\usepackage{bm}
\usepackage{url}
\usepackage{comment}
\usepackage{slashed}
\usepackage{stmaryrd}
\usepackage{caption}
\usepackage{subcaption}
\usepackage{scalerel}

\usepackage{tikz-cd}
\usetikzlibrary{matrix}
\usetikzlibrary{positioning}
\usetikzlibrary{arrows}
\usetikzlibrary{backgrounds}

\usepackage{caption}
\hypersetup{colorlinks=true,linkcolor=black,urlcolor=lcolor,citecolor=black}

\definecolor{amethyst}{rgb}{0.6, 0.4, 0.8}
\definecolor{lcolor}{rgb}{0.6,0.3,0.3}
\definecolor{caribbeangreen}{rgb}{0.0, 0.8, 0.6}
\definecolor{flamingopink}{rgb}{0.99, 0.56, 0.67}
\definecolor{hollywoodcerise}{rgb}{0.96, 0.0, 0.63}
\definecolor{airforceblue}{rgb}{0.36, 0.54, 0.66}
\definecolor{ballblue}{rgb}{0.13, 0.67, 0.8}
\definecolor{brinkpink}{rgb}{0.98, 0.38, 0.5}
\definecolor{bluegray}{rgb}{0.4, 0.6, 0.8}
\definecolor{impx}{rgb}{0.96, 0.0, 0.63}
\definecolor{junglegreen}{rgb}{0.16, 0.67, 0.53}
\definecolor{bluebell}{rgb}{0.64, 0.64, 0.82}
\definecolor{cornflowerblue}{rgb}{0.39, 0.58, 0.93}
\definecolor{pastelred}{rgb}{1.0, 0.41, 0.38}
\definecolor{carminepink}{rgb}{0.92, 0.3, 0.26}
\definecolor{darkpastelpurple}{rgb}{0.59, 0.44, 0.84}
\definecolor{darkcerulean}{rgb}{0.03, 0.27, 0.49}
\definecolor{green(munsell)}{rgb}{0.0, 0.66, 0.47}
\definecolor{bleudefrance}{rgb}{0.19, 0.55, 0.91}
\definecolor{antiquefuchsia}{rgb}{0.57, 0.36, 0.51}
\definecolor{deeplilac}{rgb}{0.6, 0.33, 0.73}
\definecolor{ferngreen}{rgb}{0.31, 0.47, 0.26}

\definecolor{Icol}{rgb}{0.93, 0.53, 0.18}
\definecolor{Pcol}{rgb}{0.16, 0.32, 0.75}
\definecolor{Hcol}{rgb}{0.16, 0.67, 0.53}
\definecolor{Lcol}{rgb}{0.16, 0.67, 0.53}

\definecolor{qcolor}{rgb}{0.19,0.55,0.91}
\definecolor{hcolor}{rgb}{0.9,0.2,0.5}
\definecolor{scolor}{rgb}{0.9,0.5,0.2}

\definecolor{cccolor}{rgb}{0.55, 0.0, 0.55}
\definecolor{hxcolor}{rgb}{0., 0.5, 0.15}

\newcommand{\tc}[2]{#2}
\newcommand{\TC}[2]{#2}
\newcommand{\final}[1]{#1}
\newcommand{\newadd}[1]{\tc{bluegray}{#1}}
\newcommand{\step}{\vskip 3mm}

\newcommand{\p}{\partial}
\newcommand{\Z}{\mathbbm{Z}}

\newcommand{\Nzero}{\mathbbm{N}_0}
\newcommand{\R}{\mathbbm{R}}

\newcommand{\one}{\mathbbm{1}}
\newcommand{\eps}{\epsilon}
\DeclareMathOperator{\Hom}{Hom}
\DeclareMathOperator{\SO}{SO}

\DeclareMathOperator{\tr}{tr}

\renewcommand{\div}{\operatorname{div}}

\newcommand{\curl}{\operatorname{curl}}
\newcommand{\CURL}{\operatorname{CURL}}

\newcommand{\grad}{\operatorname{grad}}
\newcommand{\coker}{\textnormal{coker}}
\newcommand{\image}{\textnormal{image}}
\newcommand{\supp}{\textnormal{supp}}

\newcommand{\Harm}{\mathcal{H}}
\newcommand{\dist}{\textnormal{dist}}
\renewcommand{\subset}{\subseteq}
\newcommand{\ddR}{\tc{blue}{d}}

\newcommand{\I}{\mathcal{I}}

\newcommand{\new}[1]{\tc{deeplilac}{#1}}

\newcommand{\jap}[1]{\langle #1 \rangle}

\newcommand{\blf}[1]{\textnormal{\tc{red}{lf}}_{#1}}
\newcommand{\brf}[1]{\textnormal{\tc{red}{rf}}_{#1}}
\newcommand{\bff}[1]{\textnormal{\tc{red}{ff}}_{#1}}
\newcommand{\diag}{\textnormal{diag}}
\newcommand{\bdiag}{\textnormal{diag}_b}

\newcommand{\intermult}{\iota}

\newcommand{\gammaref}{\gamma}

\newcommand{\dens}{|\Omega|}

\newcommand{\distr}{\mathcal{D}}

\newtheoremstyle{mytheoremstyle}
{14pt}
{14pt}
{\itshape}
{20pt}
{\bfseries}
{.}
{.5em}
{}

\newtheoremstyle{myremarkstyle}
{10pt}
{10pt}
{}
{20pt}
{\itshape}
{.}
{.5em}
{}

\newtheoremstyle{myproofstyle}
{12pt}
{12pt}
{}
{20pt}
{\itshape}
{.}
{.5em}
{}

\newtheoremstyle{mycommentstyle}
{5pt}
{5pt}
{\footnotesize\sffamily\color{airforceblue}}
{0pt}
{}
{}
{.5em}
{}

\newtheoremstyle{myissuestyle}
{14pt}
{14pt}
{\itshape\color{magenta}}
{20pt}
{\bfseries}
{.}
{.5em}
{}

\newtheoremstyle{myquestionstyle}
{14pt}
{14pt}
{\color{airforceblue}}
{20pt}
{\bfseries}
{.}
{.5em}
{}

\theoremstyle{mytheoremstyle}
\newtheorem{theorem}{Theorem}
\newtheorem*{theorem*}{Theorem}
\newtheorem{definition}{Definition}
\newtheorem{lemma}{Lemma}

\newtheorem{corollary}[lemma]{Corollary}

\theoremstyle{myissuestyle}

\theoremstyle{myquestionstyle}

\theoremstyle{myremarkstyle}
\newtheorem{remark}{Remark}
\newtheorem{example}{Example}

\theoremstyle{myproofstyle}
\newtheorem*{PROOF}{Proof}

\newcommand{\func}{\mathcal{R}}
\newcommand{\vect}{\mathcal{V}}
\newcommand{\mat}{\mathcal{M}}
\newcommand{\matsym}{\mathcal{M}_{\textnormal{sym}}}
\DeclareMathOperator{\symtr}{\hat{s}}

\newcommand{\WdR}{W_{\textnormal{dR}}}
\newcommand{\WddR}{d_{\textnormal{dR}}}
\newcommand{\CC}{W}
\newcommand{\DD}{D}
\newcommand{\GDD}{\tc{orange}{G}_{\DD}}
\newcommand{\PiDD}{\tc{orange}{\Pi}_{\DD}}
\newcommand{\PDD}{\tc{orange}{P}_{\DD}}
\newcommand{\IDD}{\tc{orange}{I}_{\DD}}
\newcommand{\dx}{\,d\mu}
\newcommand{\Wten}{\widetilde{W}}
\newcommand{\dten}{\widetilde{D}}
\newcommand{\deltaten}{\tc{red}{q}}
\newcommand{\deltaHPL}{\tc{red}{q}}
\DeclareMathOperator{\tomat}{\tilde\varepsilon}
\DeclareMathOperator{\tovec}{\varepsilon}
\newcommand{\matst}{\widehat{\mat}}
\DeclareMathOperator{\gradst}{\widehat{\grad}}
\DeclareMathOperator{\curlst}{\widehat{\curl}}
\DeclareMathOperator{\divst}{\widehat{\div}}
\newcommand{\HHdr}{\widetilde{G}}
\newcommand{\DIVtodr}{\widetilde{I}}
\newcommand{\drtoDIV}{\widetilde{P}}
\DeclareMathOperator{\Mincl}{\tc{orange}{\hat\iota}}

\newcommand{\Gten}{\tc{lcolor}{\slashed{G}}}
\newcommand{\Pten}{\tc{lcolor}{\slashed{P}}}
\newcommand{\Iten}{\tc{lcolor}{\slashed{I}}}

\newcommand{\PGRg}{A}
\newcommand{\PGRk}{B}
\newcommand{\GGRg}{G_A}
\newcommand{\GGRk}{G_B}
\newcommand{\bump}{\psi}
\usetikzlibrary{decorations.markings}

\newcommand{\homotopypicture}{
\begin{tikzpicture}[scale=0.95]
\def\x{2};
\def\y{1.5};
\def\z{7};
\def\sh{0.5};
\def\l{4};
\begin{scope}[decoration={
    markings,
    mark=at position 0.65 with {\arrow{Latex[length=1.5mm,width=1.5mm]}}}]
\draw[->] (0,-0.2) -- (0,2.5) node[anchor=east] {\footnotesize $t$};
\draw (0,0) node[anchor=east] {\footnotesize $0$};
\draw (0,2) node[anchor=east] {\footnotesize $1$};
\draw[very thick] (-0.1,2) -- (\l,2) node[anchor=west] {\footnotesize $\{1\}\times\Mclosed$};
\draw[very thick] (-0.1,0) -- (\l,0) node[anchor=west] {\footnotesize $\{0\}\times\Mclosed$};
\draw[very thick,scolor,postaction={decorate}] (\x,1.5) -- (\x,0) node[midway,anchor=east,scolor] {\footnotesize $h_0(s,(t,y))$};
\draw[very thick, cornflowerblue,postaction={decorate}] (\x,1.5) -- (\x+\sh,0)
node[midway,anchor=west,cornflowerblue] 
{\footnotesize $h_{a\neq0}(s,(t,y))$};
\draw[fill=black] (\x,1.5) circle(2pt) 
node[anchor=south] {\footnotesize $(t,y)$};
\draw[fill=scolor,scolor] (\x,0) circle(2pt) 
node[anchor=north,xshift=-2mm] {\footnotesize $\phantom{\phi_a}y$};
\draw[fill=cornflowerblue,cornflowerblue] (\x+\sh,0) circle(2pt) 
node[anchor=north,xshift=2mm] {\footnotesize $\phi_a(y)$};
\end{scope}
\begin{scope}[xshift=0.55\linewidth,
decoration={
    markings,
    mark=at position 0.75 with {\arrow{Latex[length=1.5mm,width=1.5mm]}}}]
\draw[->] (0,-0.2) -- (0,2.5) node[anchor=east] {\footnotesize $t$};
\draw (0,0) node[anchor=east] {\footnotesize $0$};
\draw (0,2) node[anchor=east] {\footnotesize $1$};
\draw[very thick] (-0.1,2) -- (3,2) node[anchor=west] {\footnotesize $\{1\}\times\Mclosed$};
\draw[very thick] (-0.1,0) -- (3,0) node[anchor=west] {\footnotesize $\{0\}\times\Mclosed$};
\draw[very thick,scolor,postaction={decorate}] (\y,.5) -- (\y,0);
\draw[very thick, cornflowerblue,postaction={decorate}] (\y,.5) -- (\y+\sh,0);
\draw[fill=black] (\y,.5) circle(2pt) 
node[anchor=south] {\footnotesize $(t',y)$};
\draw[fill=scolor,scolor] (\y,0) circle(2pt) 
node[anchor=north,xshift=-2mm] {\footnotesize $\phantom{\phi_a}y$};
\draw[fill=cornflowerblue,cornflowerblue] (\y+\sh,0) circle(2pt) 
node[anchor=north,xshift=2mm] {\footnotesize $\phi_a(y)$};
\end{scope}
\end{tikzpicture}}

\title{%
A support preserving homotopy for the de Rham complex 
with boundary decay estimates}
\author{Andrea N\"utzi}

\usetikzlibrary{shapes}

\newcommand{\Mfd}{\tc{lcolor}{M}}
\newcommand{\MfdBC}{\tc{lcolor}{Y}}
\newcommand{\XMfd}{\tc{lcolor}{X}}
\newcommand{\Mdim}{\tc{lcolor}{n}}
\newcommand{\nBC}{\tc{hollywoodcerise}{N}}
\newcommand{\bdf}{\tc{lcolor}{\rho}}

\newcommand{\Gfull}{\tc{lcolor}{G}}
\newcommand{\Ifull}{\tc{lcolor}{I}}
\newcommand{\Pfull}{\tc{lcolor}{P}}
\newcommand{\Pifull}{\tc{lcolor}{\Pi}}

\newcommand{\hodge}{\tc{lcolor}{\star}}
\newcommand{\MMfd}{\tc{lcolor}{M}^{\tc{lcolor}{2}}}
\newcommand{\bMfd}{\tc{lcolor}{M}^{\tc{lcolor}{2}}_{\tc{lcolor}{b}}}
\newcommand{\bconormal}{\tc{lcolor}{\I}_{\textnormal{cl}}}
\newcommand{\Cvan}{\dot{C}}
\newcommand{\Symb}{S_{\text{cl}}}

\newcommand{\A}{\mathcal{A}}

\newcommand{\ind}[1]{\mathcal{E}_{#1}}
\newcommand{\indd}[1]{\mathcal{E}_{#1}'}
\newcommand{\inddd}[1]{\mathcal{E}_{#1}''}
\newcommand{\KB}{\mathcal{H}}
\newcommand{\Hsob}{\tc{ferngreen}{H}_{\tc{ferngreen}{b}}}

\newcommand{\blowdown}{\tc{cccolor}{\beta}_b}
\newcommand{\blowdownDistr}{\tc{cccolor}{\beta}_b}
\newcommand{\bPsdo}{\Psi_{\smash{b}}}

\newcommand{\bVF}{\tc{impx}{\mathcal{V}}_{\tc{impx}{b}}}
\newcommand{\Psdo}{\tc{cccolor}{\Psi}}
\newcommand{\piR}{\pi_R}
\newcommand{\piL}{\pi_L}

\newcommand{\Mclosed}{\TC{cccolor}{Y}}
\newcommand{\Mcollar}{\TC{cccolor}{U}}
\newcommand{\McollarBC}{\tc{cccolor}{Y}}
\newcommand{\Mcollarb}{\TC{blue}{U_b^2}}
\newcommand{\McollarbOpen}{\TC{cccolor}{X}}
\newcommand{\Gcollar}{\tc{scolor}{G}}
\newcommand{\tcoord}{t}

\newcommand{\VBrel}[3]{\tc{ferngreen}{\wedge^{#1} T^*(#2,#3)}}
\newcommand{\VBrelpt}[4]{\tc{ferngreen}{\wedge^{#1} T^*_{#4}(#2,#3)}}

\newcommand{\CdimNEW}{q}

\newcommand{\Kbundle}{W}
\newcommand{\Lmap}{\TC{orange}{L}}

\newcommand{\Amap}{\tc{caribbeangreen}{A}}

\newcommand{\charact}{\tc{orange}{\chi}}
\newcommand{\Rrad}[1]{\overline{\R^{#1}}}

\newcommand{\Lb}{\tc{ferngreen}{L}^{\tc{ferngreen}{2}}}

\newcommand{\bVFspan}{\tc{ferngreen}{V}}

\newcommand{\Lx}{L^2}
\newcommand{\Hx}{\tc{deeplilac}{H}_{b}}

\newcommand{\heaviside}{\tc{orange}{H}}

\newcommand{\bdfdiag}{\TC{cccolor}{\rho}}
\newcommand{\blowdowndiag}{\tc{blue}{\beta}_{\tc{blue}{\textnormal{d}}}}
\newcommand{\cutoff}{\tc{antiquefuchsia}{\varphi}}

\newcommand{\nVF}{\TC{red}{\ell}}

\newcommand{\order}[1]{\tc{junglegreen}{#1}}
\newcommand{\dd}{\TC{orange}{\mathrm{d}}}

\newcommand{\Gtay}{\tc{lcolor}{G'}_{\tc{lcolor}{H}}}
\newcommand{\Ptay}{\tc{lcolor}{P'}_{\tc{lcolor}{H}}}
\newcommand{\Ghodge}{\tc{lcolor}{G}_{\tc{lcolor}{H}}}
\newcommand{\phodge}{\tc{lcolor}{P}_{\tc{lcolor}{H}}}
\newcommand{\ihodge}{\tc{lcolor}{I}_{\tc{lcolor}{H}}}

\newcommand{\Gfullx}{\Gfull}
\newcommand{\Pifullx}{\Pifull}
\newcommand{\GdR}{\tc{lcolor}{G}}
\newcommand{\PidR}{\tc{lcolor}{\Pi}}
\newcommand{\PdR}{\tc{lcolor}{P}}
\newcommand{\IdR}{\tc{lcolor}{I}}
\newcommand{\Lie}{\mathcal{L}}
\newcommand{\HWdr}{H_{\textnormal{dR}}}

\let\oldabstract\abstract
\let\oldendabstract\endabstract
\makeatletter
\renewenvironment{abstract}
{%
               {\list{}{\setlength{\leftmargin}{6.5mm} 
                        \listparindent 1.5em%
                        \itemindent    \listparindent%
                        \rightmargin   \leftmargin%
                        \parsep        \z@ \@plus\p@}%
                \item\relax}%
               {\endlist}%
\oldabstract}
{\oldendabstract}
\makeatother

\date{}

\begin{document}
\maketitle

\begin{abstract}
We study the de Rham complex of relative differential forms on compact manifolds with boundary. Chain homotopies for this complex are highly non-unique, and different homotopies can have different analytic properties, particularly near the boundary. We construct a chain homotopy that has desirable support propagation properties, and that satisfies estimates relative to weighted Sobolev norms, where the weights measure decay at the boundary. The estimates are optimal given the homogeneity properties of the de Rham differential under boundary dilation, and are obtained by showing that the homotopy is a b-pseudodifferential operator. As a corollary we obtain a right inverse of the divergence operator on Euclidean space that preserves support on large balls around the origin, and satisfies estimates that measure decay at infinity. Such a support preserving right inverse was constructed before by Bogovskii, but its mapping properties are not optimal with respect to decay. As a further corollary, in three dimensions we obtain a right inverse of the divergence operator on symmetric traceless matrices, and therefore of the linearized constraint operator of general relativity about flat space.
\end{abstract}

\tableofcontents

\section{Introduction}
Consider a smooth compact manifold $\Mfd$ with boundary $\p\Mfd$.
%
The de Rham complex on $\Mfd$ is given by the smooth 
differential forms $\Omega(\Mfd)$ together with the de Rham differential $\ddR$.
We consider its subcomplex of relative forms,
given by all forms whose 
tangential parts vanish at the boundary,
\begin{equation*}
\Omega(\Mfd,\p\Mfd) = \{ \omega\in\Omega(\Mfd)\mid i^*\omega=0 \}
\end{equation*}
where $i:\p\Mfd\hookrightarrow\Mfd$ is the natural inclusion.
This is indeed a subcomplex because $i^*$ and $\ddR$ commute.
The relative homology
\[ 
H(\Mfd,\p\Mfd)
=
\frac{\{\omega\in\Omega(\Mfd,\p\Mfd)\mid\ddR\omega=0\}}{\{\ddR\omega\mid \omega\in\Omega(\Mfd,\p\Mfd)\}}
\]
is finite-dimensional,
and isomorphic to the homology of the complex
of compactly supported forms in the interior of $\Mfd$.

A chain homotopy for the complex of relative forms on $\Mfd$
is, for our purpose, a linear map $\Gfull$ from relative
$k$-forms to relative $(k-1)$-forms such that
\begin{equation*}
\one - \Pifull \;=\; \ddR\Gfull + \Gfull\ddR
\end{equation*}
where $\Pifull$ is a projection 
onto a complement of the image of $\ddR$
in the kernel of $\ddR$,
in other words, it is a projection onto relative homology.
This equation implies in particular that 
the restriction of $\Gfull$ to closed forms
is a right inverse of $\ddR$, 
up to finitely many integrability conditions.
Such a chain homotopy is highly non-unique.
We construct a homotopy that has desirable properties
with regard to support and estimates, in the following informal sense: 
\newcommand{\cutoffin}{\varphi}
\newcommand{\cutoffout}{\varphi}
\begin{enumerate}[label=(\roman*)]
\item\label{item:introSupport}
It is support preserving near the boundary:
For all sufficiently small $\eps>0$, 
if the input vanishes in an $\eps$-neighborhood of a boundary component,
then also the output vanishes there.
The $\eps$-neighborhood is defined
with respect to a boundary defining function.

\item\label{item:introLocalpsdo}
It is a pseudodifferential operator in the interior:
For every smooth function $\cutoffin$
with compact support in the interior of $\Mfd$, the map
$\cutoffout\Gfull\cutoffin$ is a pseudodifferential operator of order $-1$.
This implies mapping properties, 
for example $\cutoffout\Gfull\cutoffin$ 
gains one derivative with respect to Sobolev norms.

\item\label{item:introbdarypsdo}
It is a pseudodifferential operator up to the boundary:
The map $\Gfull$ satisfies estimates for weighted Sobolev norms,
where the weights measure decay at the boundary.
The estimates 
are optimal with respect to the homogeneity properties,
under boundary dilation,
of the de Rham differential acting on relative forms.
They are obtained by showing that
$\Gfull$ is a b-pseudodifferential operator
in the sense of Melrose \cite{MelroseGreenBook},
see also \cite{Mazzeo91},
a natural class of pseudodifferential operators
on manifolds with boundary that have controlled homogeneity,
with a Schwartz kernel that vanishes in a 
neighborhood of the so-called left face.
This is consistent with the fact that the de Rham differential
acting on relative forms is b-differential.
\end{enumerate}

Previous constructions yield chain homotopies which have 
some but not all of these properties. 
A well-known construction uses the Hodge decomposition.
Given a smooth Riemannian metric on $\Mfd$, 
then \cite[Section 5.9]{taylor1} the Greens function of the
Hodge Laplacian with relative boundary conditions
post-composed with the codifferential is a chain homotopy 
that satisfies \ref{item:introLocalpsdo}
but does not satisfy \ref{item:introSupport}
and has a different structure at the boundary 
than required in \ref{item:introbdarypsdo}.
If one uses a metric with cylindrical ends rather than a smooth metric, 
then 
\cite[Section 6.4]{MelroseGreenBook}
the Hodge decomposition yields a homotopy
for a related complex that does not
satisfy \ref{item:introSupport} but satisfies 
\ref{item:introLocalpsdo} and part of \ref{item:introbdarypsdo}, 
namely it is b-pseudodifferential,
but the kernel does not vanish near the left face.

Suppose that the manifold $\Mfd$ is the radial compactification of $\R^{\Mdim}$. 
Then the restriction of the chain homotopy $\Gfull$
to a map from $\Mdim$-forms to $(\Mdim-1)$-forms 
is a right inverse of 
the divergence operator on $\R^{\Mdim}$ for the Euclidean metric,
up to the usual integrability condition.
We compare with a construction of Bogovskii \cite{Bogovskii2,Bogovskii1}, who
introduced an
explicit right inverse of the divergence designed to satisfy  \ref{item:introSupport}. Concretely, 
for $f$ a smooth compactly supported function on $\R^{\Mdim}$ set
\[ 
B(f)(x) = -x \int_0^1 \frac{1}{s^{\Mdim+1}} f\big(\frac{x}{s}\big)\,ds
\;,\quad x\in\R^{\Mdim}\setminus0
\]
This vector field is smooth away from the origin,
and has an integrable singularity at the origin.
It is a right inverse in the sense of distributions.
Then the right inverse introduced by Bogovskii 
is given by the average
\begin{equation}\label{eq:translaverage}
\int_{\R^{\Mdim}} \zeta(a) (T_a\circ B \circ T_{-a})(f)\, da
\end{equation}
where $T_a$ denotes translation by $a$,
and $\zeta$ is a bump function around the origin with integral 
equal to one.
This inverse satisfies \ref{item:introLocalpsdo} \cite{OhTataru1}. 
It does not satisfy \ref{item:introbdarypsdo}, essentially because
the average over translations degenerates at infinity
(see Corollary \ref{cor:EuclEst} and the discussion thereafter).
It was recently used
in the context of Euler flows \cite{IsettOh}
and the Maxwell-Klein-Gordon equation \cite{OhTataru1};
and to prove various 
gluing results for the constraint equations of general relativity \cite{MaoTao,MaoOhTao}.

In \cite{D,HintzUnderdetermined},
the inhomogeneous problem for more general
underdetermined elliptic linear operators is studied.
In \cite{D}, see also \cite{CD}, sufficient conditions for the existence of smooth compactly supported solutions are given, assuming the inhomogeneity is compactly supported.
For underdetermined elliptic b-differential 
operators on compact manifolds with boundary,
such as the divergence on asymptotically Euclidean manifolds
upon compactification,
\cite{HintzUnderdetermined} constructs
solutions with optimal asymptotic expansion at infinity.
\step
Though our main theorem concerns the de Rham differential,
it may be used to study other important operators.
In particular, in Appendix \ref{ap:div2} we use the 
chain homotopy $\Gfull$ to construct a 
right inverse of the divergence operator on symmetric traceless matrices 
in $\R^3$, with analytic properties like $\Gfull$.

The divergence on symmetric traceless matrices 
appears in the constraint equations of general relativity,
a key motivation for this paper.
Specifically we are interested in solutions of the constraints that 
are Schwartz class up to a Kerr-tail at infinity\footnote{%
An interesting open problem is whether solutions
of the Einstein equations with such initial data
admit a smooth conformal compactification at null infinity,
c.f.~\cite{mythesis}.}.
Constructing such solutions via an iteration scheme
requires a right inverse of the linearized constraint
operator with precise control over the asymptotics at infinity,
which this paper provides, see Remark \ref{rem:linconstraints}.
\step
We now state the main theorem, which yields existence of 
a chain homotopy for the complex of relative forms on $\Mfd$ that 
has the properties \ref{item:introSupport}, \ref{item:introLocalpsdo}, \ref{item:introbdarypsdo}.
Recall that a boundary defining function for a 
boundary component $\MfdBC\subset\p\Mfd$ is a
smooth function $\bdf: \Mfd\to [0,\infty)$ that satisfies 
$\MfdBC = \bdf^{-1}(0)$ and for which $d\bdf|_{\MfdBC}$ is nowhere zero.
Given such a boundary defining function there
exists $\eps_0>0$ such that $\bdf^{-1}([0,\eps_0])$
is a collar neighborhood of $\MfdBC$, 
meaning that there is a diffeomorphism 
$\bdf^{-1}([0,\eps_0])\to [0,\eps_0]\times \MfdBC$
that restricts to the natural identification $\MfdBC \to \{0\}\times \MfdBC$,
and whose first component 
$\bdf^{-1}([0,\eps_0])\to [0,\eps_0]$ is given by $\bdf$.
%

\begin{theorem}\label{thm:main}
Let $\Mfd$ be a smooth connected compact manifold with boundary that has
$\nBC\ge1$ boundary components, $\p\Mfd=\MfdBC_1\sqcup\dots\sqcup\MfdBC_{\nBC}$.
For each $i=1\dots\nBC$ fix a boundary defining function $\bdf_i$ for $\MfdBC_i$.
Fix $\eps_0>0$ such that for each $i$ the set 
$\bdf_i^{-1}([0,\eps_0])$ is a collar neighborhood of $\MfdBC_i$
and such that the collar neighborhoods are pairwise disjoint.
Then there exists an $\R$-linear map
\begin{equation}\label{eq:Gfullthm}
\Gfull:\; \Omega^k(\Mfd,\p\Mfd)\to \Omega^{k-1}(\Mfd,\p\Mfd)
\end{equation}
defined for every integer $k$, 
with the following properties:
\begin{enumerate}[label=(a\arabic*)]
\item \label{item:AlgFull}
Algebraic control: 
The map 
$\Pifull:\Omega^k(\Mfd,\p\Mfd)\to \Omega^{k}(\Mfd,\p\Mfd)$ defined by 
\[ 
\one-\Pifull \;=\; \ddR \Gfull + \Gfull \ddR 
\]
satisfies $\Pifull^2=\Pifull$ and $\Pifull\ddR =0$ and $\ddR\Pifull=0$.
In particular, $\Pifull$ is a projection onto 
a complement of the image of $\ddR$ in the kernel of $\ddR$. 
\item \label{item:SupportFull}
Support control: 
For all $0<\eps\le\frac{\eps_0}{2}$,
$i=1\dots\nBC$, 
and $\omega\in \Omega(\Mfd,\p\Mfd)$,
\begin{align}\label{eq:SupportFull}
\omega|_{\bdf_i\le\eps}=0 
\qquad&\Rightarrow\qquad
(\Gfull\omega)|_{\bdf_i\le\eps}=0
\end{align}
%
Moreover for all 
$i=1\dots\nBC$
and $\omega\in \Omega(\Mfd,\p\Mfd)$
one has $(\Pifull\omega)|_{\bdf_i\le\frac{\eps_0}{2}}=0$.
\item \label{item:AnalysisFull}
Analytic control: The maps $\Gfull$ and $\Pifull$ are b-pseudodifferential
operators,
\begin{align*}
\Gfull&\;\in\; \bPsdo^{-1,\order{(0,\emptyset,0)}}(\Mfd;\VBrel{}{\Mfd}{\p\Mfd})\\
\Pifull&\;\in\;\Psdo^{-\infty,\order{(\emptyset,0)}}(\Mfd;\VBrel{}{\Mfd}{\p\Mfd})
\end{align*}
See Definition \ref{def:bpsdo} for explanation of this notation.
\final{Here $\VBrel{}{\Mfd}{\p\Mfd}$ is defined
as the vector bundle whose space of smooth sections is 
$\Omega(\Mfd,\p\Mfd)$, as in Swan's theorem.} 
\end{enumerate}
\end{theorem}
The proof appears in Section \ref{sec:proof}.
\step
We explain the mapping properties of $\Gfull$ implied by \ref{item:AnalysisFull},
see Lemma \ref{lem:Mappingprop} in Section \ref{sec:bpsdo} for details.
To state them we need the following definitions.
Let $\bVF$ be the space of b-vector fields,
given by all smooth vector fields on $\Mfd$
that are tangential to the boundary $\p\Mfd$.
In local coordinates $(\bdf,y^1,\dots,y^{\Mdim-1})$ 
around a boundary point,
where $\bdf$ is a boundary defining function 
and $y^1,\dots,y^{\Mdim-1}$ 
with $\Mdim=\dim\Mfd$ restrict to coordinates on $\p\Mfd$,
a local $C^\infty$-basis of $\bVF$ is given by
\begin{equation}\label{eq:bvfbasis}
\bdf\p_{\bdf},\ \p_{y^1},\dots,\p_{y^{\Mdim-1}}
\end{equation}
Note that these vector fields are invariant under scaling $\bdf$ by a positive number.
For $s\in\Nzero$ the b-Sobolev norm for smooth functions on $\Mfd$ is defined by 
\begin{equation}\label{eq:Hsobnorm}
\|f\|_{\Hsob^{s}(\Mfd)}^2
\;=\;
\sum_{j=0}^{s} 
\sum_{V_1,\dots,V_j\in\bVFspan} 
\|V_1\cdots V_{j}f\|_{\Lb(\Mfd)}^2
\end{equation}
\final{where $\bVFspan$ is a finite $C^\infty$-generating set of $\bVF$},
and where the $\Lb$-norm is defined using a smooth positive
measure on $\Mfd$.
One generalizes to $s\in\R$ by interpolation and duality.
Then for $s,\delta\in\R$ the weighted b-Sobolev norms are
$\|f\|_{\bdf^{\delta}\Hsob^s(\Mfd)} = \|\bdf^{-\delta} f\|_{\Hsob^s(\Mfd)}$
with $\bdf$ a boundary defining function for $\p\Mfd$.
\begin{corollary}
Let $\Gfull$ be as in Theorem \ref{thm:main}.
For all $s,\delta\in\R$ with $\delta>-\frac12$ there exists
a constant $C>0$ such that for all $\omega\in\Omega(\Mfd,\p\Mfd)$,
\begin{equation}\label{eq:IntroestimatesG}
\|\Gfull\omega\|_{\bdf^{\delta}\Hsob^{s+1}(\Mfd)}
\;\le\;
C\|\omega\|_{\bdf^{\delta}\Hsob^{s}(\Mfd)}
\end{equation}
where the norms are taken componentwise using 
local $C^\infty$-bases of $\Omega(\Mfd,\p\Mfd)$.
\end{corollary}
Note that near the boundary a basis for the relative forms is
\begin{equation}\label{eq:relativebasis}
d\bdf \wedge dy^{i_1}\wedge\dots\wedge dy^{i_k},\;
\bdf\,dy^{i_1}\wedge\dots\wedge dy^{i_k} 
\end{equation}
where 
$1\le i_1<\dots< i_k \le \Mdim$ and $k$ runs over the integers.
Each basis element has the same
homogeneity under scaling $\bdf$ by a positive number.

\final{The estimates \eqref{eq:IntroestimatesG} are natural
estimates for a chain homotopy for the de Rham complex of relative forms.}
Namely $\ddR:\Omega(\Mfd,\p\Mfd)\to\Omega(\Mfd,\p\Mfd)$
is a first order b-differential operator,
which means that in the basis \eqref{eq:relativebasis} it only
differentiates along the b-vector fields \eqref{eq:bvfbasis},
see Lemma \ref{lem:ddRNEW}.
Hence it satisfies
\[ 
\|\ddR\omega\|_{\bdf^{\delta}\Hsob^{s}(\Mfd)}
\;\le\; 
C\|\omega\|_{\bdf^{\delta}\Hsob^{s+1}(\Mfd)}
\]
which is precisely the converse of \eqref{eq:IntroestimatesG}.

\step
{\bf Chain homotopy for de Rham on $\R^{\Mdim}$.}
As a special case one can apply Theorem \ref{thm:main}
to the radial compactification of $\R^{\Mdim}$, which
is informally given by attaching to $\R^{\Mdim}$ a sphere at infinity.
In this case the chain homotopy \eqref{eq:Gfullthm} restricts to,
in particular, a chain homotopy for the complex 
of compactly supported forms on $\R^{\Mdim}$,
and also, a chain homotopy for the complex
of forms whose components are Schwartz functions.
This homotopy preserves support on large balls around the origin.
The estimates \eqref{eq:IntroestimatesG} are equivalent to 
estimates for standard weighted 
b-Sobolev norms on $\R^{\Mdim}$.
The b-Sobolev norm for $s\in\Nzero$ is\footnote{%
Beware that
this is not the same norm as \eqref{eq:Hsobnorm}
with $\Mfd$ the radial compactification of $\R^{\Mdim}$,
which uses an $L^2$-norm relative to a
smooth measure on the compactification,
see Lemma \ref{lem:ComparabilityOfNorms}.}
\begin{equation}\label{eq:HxnormNEW}
\|f\|_{\Hx^s(\R^{\Mdim})}^2
=
\sum_{\substack{\alpha\in\Nzero^{\Mdim}\\|\alpha|\le s}}
\|(\jap{x}\p_{x^1})^{\alpha_1}\cdots  (\jap{x}\p_{x^{\Mdim}})^{\alpha_{\Mdim}} f\|_{\Lx(\R^{\Mdim})}^2
\end{equation}
where $x$ are the standard Cartesian coordinates on $\R^{\Mdim}$,
where $\jap{x}=\sqrt{1+\|x\|^2}$,
and the $\Lx$-norm is defined using the 
standard Lebesgue measure $dx^1\cdots dx^{\Mdim}$.
One generalizes to $s\in\R$ by interpolation and duality.
Then for $s,\delta\in\R$ the weighted b-Sobolev norms are
$\|f\|_{\jap{x}^{-\delta}\Hx^s(\R^{\Mdim})} = \|\jap{x}^{\delta} f\|_{\Hx^s(\R^{\Mdim})}$.
\begin{corollary}\label{cor:EuclEst}
Let $\Gfull$ be as in Theorem \ref{thm:main}
with $\Mfd$ the radial compactification of $\R^{\Mdim}$.
For all $k=0\dots\Mdim$ and
$s,\delta\in\R$ with $\delta>k-\frac{\Mdim}{2}$ 
there exists a constant $C>0$ such that for all 
compactly supported forms $\omega\in\Omega^k_c(\R^{\Mdim})$,
\begin{equation}
\|\Gfullx\omega\|_{\jap{x}^{-\delta+1}\Hx^{s+1}(\R^{\Mdim})}
\;\le\;
C
\|\omega\|_{\jap{x}^{-\delta}\Hx^{s}(\R^{\Mdim})}
\label{eq:estGxNEW}
\end{equation}
where the norms are taken componentwise 
using the basis given by all wedge products of zero or more of the
differentials $dx^1,\dots,dx^{\Mdim}$.
\end{corollary}
See Theorem \ref{thm:Mainx} in Section \ref{sec:Rdcor}
for details, which one can read without reading
Sections \ref{sec:Preliminaries}, \ref{sec:collarnbhd}, \ref{sec:prfthm1}.
Specializing further to top degree forms, this chain homotopy
is a right inverse of 
the divergence operator on $\R^{\Mdim}$ for the Euclidean metric,
up to the usual integrability condition.
The mapping properties \eqref{eq:estGxNEW} of this inverse
are better than those of the 
Bogovskii inverse \eqref{eq:translaverage},
for which \eqref{eq:estGxNEW} fails, 
see Remark \ref{rem:CompareBog}.
In Appendix \ref{ap:div2},
this inverse is used to construct a right inverse of the divergence
operator on symmetric traceless matrices in three dimensions.
\step
{\bf Proof strategy.}
The construction of the chain homotopy $\Gfull$ in Theorem \ref{thm:main}
has essentially two parts. 
Near the boundary, 
it uses fiber integration familiar from 
the proof of the Poincar\'e lemma,
and averages it over suitable topological homotopies;
this average does not degenerate at the boundary, 
giving the optimal weighted estimates
(Section \ref{sec:collarnbhd} and \ref{sec:HE1}).
In the interior, the construction uses standard Hodge theory
(Section \ref{sec:HodgeNEW}).
The two parts are patched together 
by formulating and composing them as homotopy equivalences
(Section \ref{sec:proof}).
\step
{\bf Acknowledgements.}
I thank Rafe Mazzeo for many discussions,
and for very useful comments on drafts of this article.
I was supported by the Swiss National Science Foundation,
project number P500PT-214470.

\section{Preliminaries}\label{sec:Preliminaries}

\subsection{Homotopy equivalences}\label{sec:HomotopyEquivalence}
In this purely algebraic section we recall homotopy equivalences.
\step
A complex $(C,\dd)$ is a $\Z$-graded real vector space $C=\oplus_{k\in\Z}C^k$ together with a differential $\dd$, 
which is a linear map $C\to C$ of degree one with $\dd^2=0$,
where degree one means that it maps $C^k\to C^{k+1}$ for every $k$.
A homotopy equivalence between two complexes 
$(C,\dd)$ and $(C',\dd')$ 
is a four-tuple of linear maps 
\begin{equation}\label{eq:HE}
\big(\,r:C\to C',\,\ell:C'\to C,\,u:C\to C,\,v:C'\to C'\,\big)
\end{equation}
where $r,\ell$ have degree zero
and $u,v$ have degree negative one,
such that:
\begin{align}\label{eq:HEcond}
\begin{aligned}
&\text{$\ell$ and $r$ are chain maps}\\
&\one-\ell r = \dd u+u\dd\\
&\one-r\ell = \dd'v+v\dd'
\end{aligned}
\end{align}
Pictorially, one has a non-commutative diagram
\begin{equation*}
\begin{tikzpicture}[baseline=(current  bounding  box.center)]
  \matrix (m) [matrix of math nodes, column sep = 28mm, minimum width = 4mm ]
  {
    (C,\dd) & (C',\dd') \\
  };
  \path[-stealth]
    (m-1-1) edge [transform canvas={yshift=1.5mm}] node [above] {$r$ \footnotesize chain map} (m-1-2)
            edge [out=180+25,in=180-25,min distance=8mm] node
                 [left,xshift=-1mm] {
                 $u$ \footnotesize homotopy 
                 } (m-1-1)
    (m-1-2) edge [transform canvas={yshift=-1.5mm}] node [below] {$\ell$ \footnotesize chain map} (m-1-1)
            edge [out=-25,in=25,min distance=8mm] node
                 [right,xshift=1mm
                 ] {
                 $v$ \footnotesize homotopy
                 } (m-1-2);
\end{tikzpicture}
\end{equation*}
\newadd{Note that $\ell$ and $r$ induce isomorphisms
in homology that are inverse to each other. In particular
$(C,\dd)$ and $(C',\dd')$ have isomorphic homologies.}

A contraction is a special instance of a homotopy equivalence.
Namely, a contraction from $(C,\dd)$ to $(C',\dd')$ is a triple of maps
\begin{equation}\label{eq:contraction}
\big(\,r,\,\ell,\,u\,\big)
\end{equation}
such that $(r,\ell,u,0)$ is a homotopy equivalence
between $(C,\dd)$ and $(C',\dd')$.
Note that in this case $\ell r$ is a projection, $(\ell r)^2=\ell r$,
\newadd{and $\dd' = r \dd \ell$}.
\begin{lemma}[Composing homotopy equivalences]\label{lem:compHE}
Consider three complexes
$$(C,\dd)\qquad(C',\dd')\qquad(C'',\dd'')$$
Let $(r_1,\ell_1,u_1,v_1)$
be a homotopy equivalence between $(C,\dd)$ and $(C',\dd')$,
and let $(r_2,\ell_2,u_2,v_2)$
be a homotopy equivalence between $(C',\dd')$ and $(C'',\dd'')$.
Then 
\begin{align*}
\big(\ 
r_3=r_2r_1\,,\ 
\ell_3=\ell_1\ell_2\,,\ 
u_3=u_1 + \ell_1 u_2 r_1\,,\ 
v_3=v_2 + r_2 v_1 \ell_2\ 
\big)
\end{align*}
is a homotopy equivalence between $(C,\dd)$ and $(C'',\dd'')$.
Furthermore, if $\dd''=0$ then 
$(r_3,\ell_3,u_3)$ is a contraction from $(C,\dd)$ to $(C'',0)$.
\end{lemma}
\begin{proof}
The properties \eqref{eq:HEcond} hold
by direct computation.\qed
\end{proof}
\final{Note that composition of homotopy equivalences is associative.}
\step
\newcommand{\newu}{\slashed{u}}
\newcommand{\newr}{\slashed{r}}
\newcommand{\newell}{\slashed{\ell}}
We recall the homological perturbation lemma,
which will be used only in Appendix \ref{ap:div2}.
A contraction $(r,\ell,u)$ from $(C,\dd)$ to $(C',\dd')$
that also satisfies 
\begin{align}\label{eq:sidecond}
ru=0 && u\ell=0 && u^2=0
\end{align}
is called a contraction with side conditions
(or special deformation retract \cite{crainic}).
\begin{lemma}[Homological perturbation lemma]\label{lem:HPL}
Let $(r,\ell,u)$ be a contraction with side conditions
from $(C,\dd)$ to $(C',\dd')$. Let $\deltaHPL:C\to C$ be a linear map of degree one
such that $(\dd+\deltaHPL)^2=0$, and such that $\deltaHPL u$ is nilpotent.
Then 
$$\big(\,
\newr = r(\one + \deltaHPL u)^{-1},\ 
\newell = (\one + u\deltaHPL)^{-1} \ell,\ 
\newu = u(\one + \deltaHPL u)^{-1}\,
\big)$$
is a contraction with side conditions from $(C,\dd+\deltaHPL)$ to $(C',\newr(\dd+\deltaHPL)\newell)$.
\end{lemma}
\begin{proof}
See \cite[Section 3.2]{crainic}.\qed
\end{proof}
\begin{remark}\label{rem:sideconditions}
Every contraction can be transformed into a contraction with
side conditions \cite[Remark 2.3(i)]{crainic}. Namely, 
let $(r,\ell,u)$ be a contraction from $(C,\dd)$ to $(C',\dd')$.
Set $u_1= (1-\ell r)u(1-\ell r)$ and $u_2 = u_1 \dd u_1$.
Then $(r,\ell,u_2)$ is a contraction with side conditions from $(C,\dd)$ to $(C,\dd')$.
%
%
%
%
%
%
\end{remark}

\subsection{b-pseudodifferential operators}\label{sec:bpsdo}
We recall basic notions of the calculus of b-pseudodifferential operators
of Melrose, see \cite{MelroseGreenBook,Mazzeo91,Hoermander3},
where we follow the conventions in \cite{Mazzeo91}.
We will only need a subspace of the full calculus,
given by Schwartz kernels that are smooth 
at the boundary faces.
All manifolds are assumed to be smooth and connected.
\step
\textbf{Distributions.}
Let $\XMfd$ be a compact manifold with corners.
We denote by $\Cvan^\infty(\XMfd)$ the space of smooth
real-valued
functions that vanish to infinite order at the boundary.
For $s\in\R$ let $\dens^s\to\XMfd$ be the $s$-density bundle\footnote{%
We use the standard convention that $1$-densities can be integrated.}.
The space of (extendible) distributional $s$-densities $\distr(\XMfd,\dens^s)$
is given by all continuous linear maps 
$\Cvan^\infty(\XMfd,\dens^{1-s}) \to \R$,
where the inclusion $C^\infty(\XMfd,\dens^s)\hookrightarrow\distr(\XMfd,\dens^s)$
is given by integration.
We abbreviate $\distr(\XMfd,\dens^0)=\distr(\XMfd)$.
Similarly, for $E\to\XMfd$ any smooth vector bundle,
the space of distributional sections $\distr(\XMfd,E)$
is given by all continuous linear maps 
$\Cvan^\infty(\XMfd,E^* \otimes \dens^{1}) \to \R$,
with $E^*$ the dual bundle.
\step
{\bf b-double space.}
Let $\Mfd$ be a compact manifold with boundary 
that has $\nBC\ge1$ boundary components 
$\p\Mfd=\MfdBC_1\sqcup\dots\sqcup\MfdBC_{\nBC}$.
The product $\MMfd=\Mfd\times\Mfd$ is a manifold with corners.
Its boundary hypersurfaces are $\MfdBC_i \times \Mfd$
and $\Mfd \times \MfdBC_i$, and its corners are $\MfdBC_i\times\MfdBC_j$
for $i,j=1\dots\nBC$.
The diagonal is 
$\diag(\Mfd) = \{ (p,p) \in \Mfd^2 \}$
and intersects the boundary in the corners
\begin{equation*}
B = \cup_{i=1}^{\nBC} (\MfdBC_i\times\MfdBC_i)
\end{equation*}
This intersection is not transversal, 
which motivates introducing the b-double space $\bMfd$
given by the 
blowup of $B$ in $\MMfd$,
\begin{equation}
\label{eq:defbMfd}
\bMfd=[\MMfd,B] \qquad\qquad \blowdown:\bMfd\to\MMfd
\end{equation} 
where $\blowdown$ is the blowdown map,
see Figure \ref{fig:M2boundary}.
\begin{figure}%
\centering
\begin{tikzpicture}
\begin{scope}
\draw[very thick] (1,0) -- (2,0) node[midway,anchor=north] 
{\footnotesize $\brf{i}$};
\draw[very thick] (0,1) -- (0,2) node[midway,anchor=east] 
{\footnotesize $\blf{i}$};
\draw[thick,densely dashed] (0.7,0.7) -- (1.7,1.7) node[midway,anchor=north west,xshift=-1mm,yshift=1mm] 
{\footnotesize $\bdiag$};
\draw[-,very thick,blue] (1,0) arc (0:90:1) node[black,midway,anchor=north east] {\footnotesize $\bff{i}$};
\end{scope}
\begin{scope}[xshift=3cm]
\draw[very thick,-{Latex[length=1.8mm,width=1.8mm]}] (1,1) -- (2,1) node[midway,anchor=south east,,xshift=2.5mm] {\footnotesize $\blowdown$};
\end{scope}
\begin{scope}[xshift=7.5cm]
\draw[very thick] (0,0) -- (2,0) node[xshift=-9mm,yshift=-3mm] 
{\footnotesize $\Mfd\times\MfdBC_i$};
\draw[very thick] (0,0) -- (0,2) node[xshift=-7mm,yshift=-8mm] 
{\footnotesize $\MfdBC_i\times\Mfd$};
\draw[thick,densely dashed] (0,0) -- (1.7,1.7) node[midway,anchor=north west,xshift=2.5mm,yshift=4.5mm] 
{\footnotesize $\diag(\Mfd)$};
\draw[color=blue,fill=blue] (0,0) circle (0.7mm) node[anchor=north east,black] {\footnotesize $\MfdBC_i^2$};
\end{scope}
\end{tikzpicture}
\captionsetup{width=115mm}
\caption{%
On the right is a neighborhood of the corner
$\MfdBC_i^2 \subset\Mfd^2$,
which is locally given by $[0,1)^2\times \MfdBC_i^2$,
the factor $\MfdBC_i^2$ is suppressed in the figure.
On the left is a neighborhood of $\bff{i}\subset\bMfd$,
which is locally given by the blowup of $(0,0)$ in $[0,1)^2$, times $\MfdBC_i^2$.}
\label{fig:M2boundary}
\end{figure}%
This is again a manifold with corners,
its boundary hypersurfaces are:
\begin{align}\label{eq:facesdef}
\begin{aligned}
\bff{i} &= \overline{\blowdown^{-1}(\MfdBC_i \times \MfdBC_i)}\\
\blf{i} &= \overline{\blowdown^{-1}(\MfdBC_i \times \Mfd)\setminus \bff{i}}\\
\brf{i} &= \overline{\blowdown^{-1}(\Mfd \times \MfdBC_i)\setminus \bff{i}}
\end{aligned}
\end{align}
We denote 
$\bff{}=\cup_{i=1}^{\nBC}\bff{i}$,
$\blf{}=\cup_{i=1}^{\nBC}\blf{i}$,
$\brf{}=\cup_{i=1}^{\nBC}\brf{i}$
and refer to them as the 
front face, left face, right face respectively.
The lift of the diagonal is the b-diagonal
\begin{align*}
\bdiag &= 
\overline{\blowdown^{-1}(\diag(\Mfd))\setminus \bff{}}
\end{align*}
and intersects the boundary transversally in the front face. 
In particular it is an embedded closed $\dim\Mfd$-dimensional
submanifold of $\bMfd$.
\begin{remark}[Local model for $\bMfd$] \label{rem:blowup}
After choosing a collar neighborhood $[0,1)\times\MfdBC_i$
of $\MfdBC_i$, the b-double space $\bMfd$ is, locally around $\bff{i}$, given by
\[ 
[[0,1)^2,(0,0)]\times \MfdBC_i^2
\]
The first factor is the blowup of $(0,0)$ in $[0,1)^2$, given by
$\{(r,\theta)\in [0,\infty)\times[0,\tfrac{\pi}{2}]\mid 
r< 1/{\max}\{\sin\theta,\cos\theta\} \}$
with blowdown map 
$(r,\theta) \mapsto (r\cos\theta,r\sin\theta)$.
Then 
$\bff{i}$, $\blf{i}$, $\brf{i}$
are locally given by
$r=0$, $\theta=\tfrac\pi2$, $\theta=0$
respectively.
It is often useful to work with projective coordinates.
Away from $\blf{i}$ these are $(t,u)$ with $u=t'/t$,
where $(t,t')$ are the coordinates on $[0,1)^2$;
then $t$ and $u$ are boundary defining functions for 
$\bff{i}$ respectively $\brf{i}$.
Away from $\brf{i}$ these are $(v,t')$ with $v=t/t'$.
\end{remark}

\textbf{Classical conormal distributions.}
For $m\in\R\cup\{-\infty\}$ let
\begin{equation}\label{eq:conormaldef}
\bconormal^m(\bMfd,\bdiag)
\subset \distr(\bMfd)
\end{equation}
be the space of classical (or one step) conormal distributions 
of order $m$ along $\bdiag$.
This is defined invariantly in \cite[Chapter 18.2]{Hoermander3}.
Here we give a characterization in terms of local coordinates.
A distribution $\mu\in \distr(\bMfd)$
is in \eqref{eq:conormaldef} if and only if
$\mu$ is smooth on $\bMfd\setminus\bdiag$
and for every $p\in\bdiag$ there exist:
\begin{itemize}
\item 
An open neighborhood $U$ of $p$
with smooth coordinates $(x,z)$ that range over $B\cap (H\times\R^{\Mdim})$ where 
$B=\{(x,z)\in\R^{\Mdim}\times\R^{\Mdim}\mid \|x\|,\|z\|<1 \}$ and
\begin{itemize}
\item 
$H = \R^{\Mdim}$ when $p\notin\bff{}$.
\item 
$H = \{x\in\R^{\Mdim}\mid x^1\ge0\}$ when $p\in\bff{}$.
\end{itemize}
and such that $p$ corresponds to $x=z=0$
and $U\cap\bdiag=\{z=0\}$.
\item 
A function $\cutoff\in C^\infty(\bMfd)$ with support in $U$
and with $\cutoff=1$ around $p$.
\item 
A classical symbol $a\in \Symb^{m}(\R^{\Mdim};\R^{\Mdim})$,
see \cite[Definition 3.11]{hintzmicro}.
\end{itemize}
such that, as compactly supported distributions on $H\times\R^{\Mdim}$,  
\begin{equation*}
\cutoff\mu =
\textstyle \int_{\R^{\Mdim}} a(x,\xi)e^{i z\cdot \xi} \,d\xi
\end{equation*}
where on the left hand side $\cutoff\mu$ is extended by zero to $H\times\R^{\Mdim}$,
and on the right hand side we use the 
distributional inverse Fourier transform in $\xi$.

For $E\to\bMfd$ a vector bundle the space $\bconormal^m(\bMfd,\bdiag;E)
\subset \distr(\bMfd,E)$ is characterized analogously, using local trivializations of $E$.
\step

We recall a simple sufficient condition for a distribution to
be classical conormal, which will be enough for our purpose.
Consider the blowup
\begin{equation}\label{eq:bMfddiagBU}
[\bMfd,\bdiag] \qquad \blowdowndiag:[\bMfd,\bdiag]\to\bMfd
\end{equation}
where $\blowdowndiag$ is the blowdown map.
Note that 
\begin{equation}\label{eq:liftdiag}
\blowdowndiag^{-1}(\bdiag)\subset [\bMfd,\bdiag]
\end{equation}
is an embedded closed codimension one submanifold.
\begin{lemma}\label{lem:KernelDiagonal}
Let $\bdfdiag$ be a boundary defining function for \eqref{eq:liftdiag}.
For all $\mu \in \distr(\bMfd)$ 
and $m\in\Z$ with $m\le \dim\Mfd-1$
one has
\[ 
\bdfdiag^{m}\blowdowndiag^*\mu \in C^\infty([\bMfd,\bdiag])
\quad\Rightarrow\quad
\mu\in \bconormal^{m-\dim\Mfd}(\bMfd,\bdiag)
\]
An analogous statement holds for distributional sections of
bundles on $\bMfd$.
\end{lemma}
\begin{proof}
Upon choosing local coordinates
this follows from elementary properties of the Fourier transform,
see for example \cite[Corollary 2.26]{Hintz2302}.\qed
\end{proof}
\textbf{b-pseudodifferential operators.}
Let $\piL,\piR:\Mfd^2\to\Mfd$ be the projections
onto the left respectively right factor.
Let $E,F\to \Mfd$ be smooth vector bundles on $\Mfd$.
Recall that the Schwartz kernel of a continuous linear map 
$C^\infty(\Mfd,E) \to C^\infty(\Mfd,F)$ is a distribution in\footnote{
Explicitly, the Schwartz kernel $K_A$ of a map $A$ 
is the unique distribution that satisfies 
$\langle A(e), f\rangle 
=
\langle
K_A, 
\pi_{R}^*e \otimes \pi_{L}^* f 
\rangle$ for all 
$e\in C^\infty(\Mfd,E)$ and 
$f\in \Cvan^\infty(\Mfd,F^*\otimes \dens^1)$.} 
\begin{equation}\label{eq:kerspace1}
\distr\big(\Mfd^2, \Hom(E,F)\otimes \pi_R^*\dens^1\big)
\end{equation}
where the bundle $\Hom(E,F)$ has fiber $\Hom_{\R}(E_{p'},F_{p})$ at $(p,p')\in\Mfd^2$, 
and where the bundle $\pi_R^*\dens^1$ 
is the pullback of the one-density bundle on $\Mfd$
along $\pi_R$.
Using $\piL^*\Omega^{\frac12}\otimes\piR^*\Omega^{\frac12}\simeq\Omega^{\frac12}$ on $\Mfd^2$,
this space is canonically isomorphic to 
\begin{equation}\label{eq:kerspace2}
\distr\big(\Mfd^2, \KB(E,F) \otimes \dens^{\frac12}\big)
\end{equation}
where we abbreviate $\KB(E,F) = \Hom(E\otimes\dens^{-\frac12},F\otimes\dens^{-\frac12})$.
The pullback along the blowdown map $\blowdown:\bMfd\to\Mfd^2$ defines a map 
from \eqref{eq:kerspace2} to 
\begin{equation}\label{eq:blowupbundle}
\distr\big(\bMfd, 
\blowdown^*(\KB(E,F))
\otimes \dens^{\frac12}\big)
\end{equation}
that we also denote by $\blowdownDistr^*$.
Beware that this map is defined using two distinct pullbacks,
the pullback of sections of $\KB(E,F)$
to sections of $\blowdown^*(\KB(E,F))$,
and the pullback of densities on $\Mfd^2$ to densities on $\bMfd$.

For a triple $\ind{}=(\ind{\bff{}},\ind{\blf{}},\ind{\brf{}})$ with 
$\ind{\bff{}},\ind{\blf{}},\ind{\brf{}}\in \Nzero\cup\{\emptyset\}$
define the subspace\footnote{%
Beware that some references use different orderings of the indices.
In particular, in \cite{Mazzeo91} the indices are ordered
as in $(\ind{\blf{}},\ind{\brf{}},\ind{\bff{}})$,
the convention used here is more standard now.}
\begin{equation}\label{eq:Alfrfffdef}
\A^{\ind{}}(\bMfd)
\;\subset\;
C^\infty(\bMfd)
\end{equation}
given by all smooth functions on $\bMfd$ that 
vanish to order $\ind{\bff{}}$, $\ind{\blf{}}$, $\ind{\brf{}}$
at the front face, left face, right face  respectively,
where vanishing to order $\emptyset$ means vanishing to infinite order.
In other words, if $\ind{\bff{}},\ind{\blf{}},\ind{\brf{}}\in \Nzero$
then 
\begin{equation}\label{eq:AlfrfffdefIntegers}
\A^{\ind{}}(\bMfd)
\;=\;
\bdf_{\bff{}}^{\ind{\bff{}}}
\bdf_{\blf{}}^{\ind{\blf{}}}
\bdf_{\brf{}}^{\ind{\brf{}}}
 C^\infty(\bMfd)
\end{equation}
where $\bdf_{\bff{}},\bdf_{\blf{}},\bdf_{\brf{}}$
are boundary defining functions for 
$\bff{}$, $\blf{}$, $\brf{}$ respectively.
If one or more of $\ind{\bff{}},\ind{\blf{}},\ind{\brf{}}$ are
equal to $\emptyset$ then \eqref{eq:Alfrfffdef} is given by
taking the intersection of \eqref{eq:AlfrfffdefIntegers} over 
all integers, as in, for example,
\[ 
\A^{(\emptyset,\ind{\blf{}},\ind{\brf{}})}(\bMfd) 
\;=\; \bigcap_{\ell\in\Nzero}
\A^{(\ell,\ind{\blf{}},\ind{\brf{}})}(\bMfd)
\]
\begin{definition}\label{def:bpsdo}
The space of b-pseudodifferential operators
of order $m\in\R\cup\{-\infty\}$ and index set 
$\ind{}=(\ind{\bff{}},\ind{\blf{}},\ind{\brf{}})
\in (\Nzero\cup\{\emptyset\})^3$,
denoted by
\begin{align}\label{eq:bPsdodef}
\bPsdo^{m,\ind{}}(\Mfd;E,F) 
\end{align}
is given by all continuous linear maps 
$A: C^\infty(\Mfd,E) \to C^\infty(\Mfd,F)$
such that
\begin{equation}\label{eq:bpsdodef}
\blowdown^* K_A \;\in\; 
\tfrac{1}{\sqrt{\bdf_{\bff{}}}}\,\A^{\ind{}}(\bMfd)\,
\bconormal^m\Big(\bMfd,\bdiag;\,\blowdown^*(\KB(E,F))\otimes \dens^{\frac12}\Big)
\end{equation}
where $K_A$ is the Schwartz kernel of $A$,
and $\blowdown^* K_A$ its pullback to \eqref{eq:blowupbundle},
and where $\bdf_{\bff{}}$ is a boundary defining function for the front face.
Furthermore define 
\begin{equation}\label{eq:respsdo}
\Psdo^{m,(\ind{\blf{}},\ind{\brf{}})}(\Mfd;E,F)
\;=\;
\bPsdo^{m,(\emptyset,\ind{\blf{}},\ind{\brf{}})}(\Mfd;E,F)
\end{equation}
\end{definition}
b-pseudodifferential operators 
are also called totally characteristic operators.
The operators in Definition \ref{def:bpsdo} are only a subspace of 
the full calculus of Melrose in \cite{Mazzeo91,MelroseGreenBook},
which also contains operators with more general index sets
that allow certain singular behavior at the boundary faces.
Further, in \eqref{eq:bpsdodef} we only allow
conormal distributions that are classical,
which suffices for our purpose.
{The Schwartz kernels of the operators \eqref{eq:respsdo} 
vanish to infinite order at the front face, and hence they
have a simple behavior already on the blowdown space $\Mfd^2$.}
We abbreviate
$\bPsdo^{m,\ind{}}(\Mfd;E)=\bPsdo^{m,\ind{}}(\Mfd;E,E)$,
and 
$\bPsdo^{m,\ind{}}(\Mfd)=\bPsdo^{m,\ind{}}(\Mfd;\Mfd\times\R)$
where $\Mfd\times\R$ is the trivial bundle of rank one.
\step
We give basic examples of b-pseudodifferential operators
\cite[Eq.~3.2-3.3]{Mazzeo91}.
Multiplication by a smooth function is in 
$\bPsdo^{0,\order{(0,\emptyset,\emptyset)}}(\Mfd)$.
%
Differentiation with respect to the b-vector fields
$\bVF$ (see  \eqref{eq:bvfbasis})
is in $\bPsdo^{1,\order{(0,\emptyset,\emptyset)}}(\Mfd)$.
More generally, define 
a first order b-differential operator $A:C^\infty(\Mfd,E)\to C^\infty(\Mfd,F)$ to be any first order differential operator 
that satisfies
\begin{equation}\label{eq:firstorderb}
A(\bdf C^\infty(\Mfd,E))\subset \bdf C^\infty(\Mfd,F)
\end{equation}
with $\bdf$ a boundary defining function for $\p\Mfd$.
This is equivalent to the more standard definition
that requires that, in a local trivialization, the principal part of 
$A$ be given by b-vector fields%
\footnote{%
Recall that $\bVF$ is given by all vector fields that are
tangential to $\p\Mfd$. The claim follows from the
fact that $\bVF$ is equivalently given by 
all vector fields $V$ with
$V(\bdf C^\infty(\Mfd))\subset \bdf C^\infty(\Mfd)$.}.
For such $A$ one has
\begin{equation}\label{eq:bvpsdo}
A \in \bPsdo^{1,\order{(0,\emptyset,\emptyset)}}(\Mfd;E,F)
\end{equation}
\begin{lemma}\label{lem:ddRNEW}
The de Rham differential 
$\ddR:\Omega(\Mfd,\p\Mfd)\to\Omega(\Mfd,\p\Mfd)$
is a first order b-differential operator, in particular
$\ddR 
\in
\bPsdo^{1,\order{(0,\emptyset,\emptyset)}}(\Mfd;\VBrel{}{\Mfd}{\p\Mfd})$.
\end{lemma}
This fails if the de Rham differential is viewed as a map 
$\ddR:\Omega(\Mfd)\to\Omega(\Mfd)$.
\begin{proof}
Using $\Omega(\Mfd,\p\Mfd)= \ddR\bdf\wedge\Omega(\Mfd)+\bdf\Omega(\Mfd)$
one gets $\ddR(\bdf\Omega(\Mfd,\p\Mfd)) \subset \bdf\Omega(\Mfd,\p\Mfd)$
as required. 
Alternatively, one can make a local calculation using the
trivialization \eqref{eq:relativebasis}: 
For smooth functions $f,g$ and every multiindex $I$ one has
$
\ddR((f d\bdf + g \bdf )\wedge dy^{I})
=
((\p_{y^i}f)dy^i + (\bdf\p_{\bdf}g) + g) \wedge d\bdf\wedge dy^I
+
(\p_{y^i}g) \bdf dy^i\wedge dy^I
$, 
so in this trivialization,
the principal part of $d$ is given by the 
b-vector fields \eqref{eq:bvfbasis}.\qed
\end{proof}

{\bf Composition. Commutators. Mapping properties.}
We state properties
of b-pseudodifferential operators that will be used later on.

Unlike the full calculus, 
the space of b-pseudodifferential operators in Definition \ref{def:bpsdo} 
is not closed under composition. 
For example, even if one composes operators whose Schwartz kernels 
are smooth on $\bMfd$, 
their composition can have logarithmic behavior at the boundary faces.
In this article we will not encounter this effect, that is, 
all compositions that we need 
stay in the more restrictive space in Definition \ref{def:bpsdo}. 
These compositions are summarized in the next lemma.
\begin{lemma}\label{lem:composition}
Let $m,m'\in\R\cup\{-\infty\}$ and let 
$\ind{}$,
$\indd{}$,
$\inddd{}$
be one of the triples
\begin{align*}
\arraycolsep=11pt
\def\arraystretch{1.4}
\begin{array}{ccc}
\ind{}&\indd{}&\inddd{}\\
\hline
(\ind{\bff{}},\ind{\blf{}},\ind{\brf{}})
&
(\indd{\bff{}},\emptyset,\emptyset) 
&
(\ind{\bff{}}+\indd{\bff{}},\ind{\blf{}},\ind{\brf{}}+\indd{\bff{}})\\
(\ind{\bff{}},\emptyset,\ind{\brf{}})
&
(\emptyset,\indd{\blf{}},\indd{\brf{}})
&
(\emptyset,\ind{\bff{}}+\indd{\blf{}},\indd{\brf{}})\\
(\ind{\bff{}},\emptyset,\emptyset )
&
(\indd{\bff{}},\indd{\blf{}},\indd{\brf{}})
&
(\ind{\bff{}}+\indd{\bff{}},\ind{\bff{}}+\indd{\blf{}},\indd{\brf{}})
\end{array}
\end{align*}
where $\ind{\bff{}},\dots,\indd{\brf{}}\in\Nzero\cup\{\emptyset\}$,
with the understanding that $\ell+\emptyset=\emptyset$
for all $\ell\in\Nzero\cup\{\emptyset\}$. 
If 
$A\in\smash{\bPsdo^{m,\ind{}}(\Mfd)}$ and 
$B\in\smash{\bPsdo^{m',\indd{}}(\Mfd)}$
then 
$A\circ B \in \bPsdo^{m+m',\inddd{}}(\Mfd)$.
\end{lemma}
\begin{proof}
By \cite[Theorem 3.15]{Mazzeo91}.\qed
\end{proof}
We recall a commutation result.
\begin{lemma}\label{lem:commutator}
Let $A\in\bPsdo^{m,\ind{}}(\Mfd)$ with $m\in\R\cup\{-\infty\}$ and 
$\ind{}\in(\Nzero\cup\{\emptyset\})^3$, let $V\in\bVF$, 
and let $f\in \bdf^{\ell}C^\infty(\Mfd)$ where $\ell\in\Nzero$ and 
$\bdf$ is a boundary defining function for $\p\Mfd$.
Then
\begin{subequations}
\begin{align}
[A,f] &\in\bPsdo^{m-1,\order{(\ind{\bff{}}+\ell,\ind{\blf{}},\ind{\brf{}})}}(\Mfd)
\label{eq:commutatorf}\\
[A,fV] &\in\bPsdo^{m,\order{(\ind{\bff{}}+\ell,\ind{\blf{}},\ind{\brf{}})}}(\Mfd)
\label{eq:commutatorV}
\end{align}
\end{subequations}
where $[\,\cdot\,,\,\cdot\,]$ is the commutator.
\end{lemma}
\begin{proof}
Let $K_A$ be the Schwartz kernel of $A$.
Then the kernel of $[A,f]$ is
\[ 
(\piR^*f - \piL^*f) K_A
\]
The pullback of $\piR^*f - \piL^*f$ along the blowdown map
vanishes along $\bdiag\subset\bMfd$, and vanishes to order $\ell$ at $\bff{}$,
which implies \eqref{eq:commutatorf}.
We show that both terms in 
\[ 
[A,fV] = f[A,V] + [A,f]V
\]
are in \eqref{eq:commutatorV}.
First term:
We have $f\in\bPsdo^{0,\order{(\ell,\emptyset,\emptyset)}}(\Mfd)$
and
$[A,V]\in\bPsdo^{m,\ind{}}(\Mfd)$
by \cite[Proposition 3.30]{Mazzeo91}.
Then Lemma \ref{lem:composition} implies that this term is in 
\eqref{eq:commutatorV}.
Second term:
This follows from \eqref{eq:commutatorf} and
\eqref{eq:bvpsdo} and Lemma \ref{lem:composition}.\qed
\end{proof}

Recall the weighted b-Sobolev norms 
$\|\cdot\|_{\bdf^\delta\Hsob^{s}(\Mfd)}$ in \eqref{eq:Hsobnorm}.
The functions $\Cvan^\infty(\Mfd)$ are
dense in the corresponding weighted b-Sobolev spaces $\bdf^\delta\Hsob^{s}(\Mfd)$.
\begin{lemma}\label{lem:Mappingprop}
Let $s,m,\delta,\delta'\in\R$ and 
$\ind{}=(\ind{\bff{}},\ind{\blf{}},\ind{\brf{}})\in(\Nzero\cup\{\emptyset\})^3$ such that 
\begin{align*}
\ind{\bff{}} &\ge \delta'-\delta &
\ind{\blf{}}-\delta' &> -\tfrac12 &
\ind{\brf{}}+\delta &>-\tfrac12 
\end{align*}
with the understanding that each inequality that involves $\emptyset$
holds trivially.
If 
$A\in\bPsdo^{m,\ind{}}(\Mfd)$
then there exists $C>0$ 
such that for all $f\in \Cvan^\infty(\Mfd)$
one has
\begin{equation}\label{eq:estimategeneral}
\|Af\|_{\bdf^{\delta'}\Hsob^{s-m}(\Mfd)}
\le 
C \|f\|_{\bdf^{\delta}\Hsob^s(\Mfd)}
\end{equation}
In particular, $A$ extends to a bounded linear map
$\bdf^{\delta}\Hsob^s(\Mfd) \to \bdf^{\delta'}\Hsob^{s-m}(\Mfd)$.
\end{lemma}
\begin{proof}
By \cite[Corollary 3.23 and Theorem 3.25]{Mazzeo91}.\qed
\end{proof}
\begin{remark}\label{rem:normsVB}
\final{Lemma \ref{lem:composition}, \ref{lem:commutator}, \ref{lem:Mappingprop} hold analogously for maps between sections of vector bundles,
where in \eqref{eq:estimategeneral} 
the norms are taken componentwise relative to 
local trivializations of the bundles.
It is understood that the trivializations 
are regular up to the boundary of $\Mfd$.}
%
%
\end{remark}

\section{Collar neighborhood of boundary}\label{sec:collarnbhd}
We construct the chain homotopy \eqref{eq:Gfullthm}
in a collar neighborhood of the boundary.
Let $\Mclosed$ be a smooth connected closed manifold.
Define $\Mcollar = [0,1]\times \Mclosed$,
where we denote by $\tcoord$ the coordinate on the first factor $[0,1]$,
and denote $\McollarBC_0 = \{0\}\times\Mclosed$.
The relative forms 
$\Omega\left(\Mcollar,\McollarBC_0\right)$
are given by all forms on $\Mcollar$ whose pullback along the inclusion 
$\McollarBC_0 \hookrightarrow \Mcollar$ vanishes.
This complex is exact,
which follows from the fact that
the complex of forms on $[0,1]$ whose
pullback to $0$ vanishes,
is exact.
\begin{lemma}\label{lem:HomotopyCollar}
Let $\Mcollar=[0,1]\times\Mclosed$ be as above.
There exists an $\R$-linear map 
\[ 
\Gcollar:\; \Omega^k(\Mcollar,\McollarBC_0)
\to\Omega^{k-1}(\Mcollar,\McollarBC_0)
\]
defined for every integer $k$,
with the following properties:
\begin{enumerate}[label=(a\arabic*)]
\item \label{item:CollarHomotopy}
Algebraic control:
It satisfies $\one = \ddR\Gcollar+\Gcollar\ddR$.
\item \label{item:CollarSupport}
Support control: 
For all $0<\eps\le1$ and all 
$\omega\in\Omega\left(\Mcollar,\McollarBC_0\right)$,
\[ 
\omega|_{\tcoord\le\eps}=0
\quad\Rightarrow\quad
(\Gcollar\omega)|_{\tcoord\le\eps}=0
\]
\item \label{item:CollarAnalytic}
Analytic control:
If $\cutoff\,{\in}\, C^\infty([0,1])$ vanishes 
in a neighborhood of $t\,{=}\,1$ then
\begin{equation}\label{eq:cutgcur}
\cutoff\Gcollar\cutoff\;\in\;
\bPsdo^{-1,\order{(0,\emptyset,0)}}(\Mcollar;
\VBrel{}{\Mcollar}{\McollarBC_0})
\end{equation}
where $\VBrel{}{\Mcollar}{\McollarBC_0}$ is the bundle
whose space of smooth sections is $\Omega(\Mcollar,\McollarBC_0)$.
\end{enumerate}
\end{lemma}
We prove this lemma in the remainder of this section.
The proof is constructive. 
In Section \ref{sec:DefOfOp} we define an operator that
satisfies \ref{item:CollarHomotopy} and \ref{item:CollarSupport}.
In Sections \ref{sec:AuxGeom} and \ref{sec:SchwartzKernelGCollar} we show that it also satisfies \ref{item:CollarAnalytic}.

\subsection{Definition of operator}\label{sec:DefOfOp}
We start with some preliminary definitions.
Fix $\nVF\ge\dim\Mclosed$ smooth vector fields 
$V_1,\dots,V_{\nVF}$ on $\Mclosed$ that 
pointwise span the tangent space.
Since $\Mclosed$ is closed, the vector fields are complete.
For $i=1\dots \nVF$ let $\phi^{V_i}:\R\times \Mclosed\to \Mclosed$ be the flow of $V_i$.
For every $a\in \R^{\nVF}$ define $\phi_a:\Mclosed\to \Mclosed$ by 
$\phi_a = \phi_{a^1}^{V_1} \circ \cdots \circ \phi_{a^{\nVF}}^{V_{\nVF}}$
and define 
\begin{align}\label{eq:hadef}
h_a:\; [0,1]\times \Mcollar \to \Mcollar \qquad (s,(\tcoord,y)) \mapsto (s\tcoord,\phi_{(1-s)a}(y))
\end{align}
Here $s$ is the coordinate on $[0,1]$,
and $(\tcoord,y)$ is in $\Mcollar=[0,1]\times\Mclosed$.
See Figure \ref{fig:homotopypicture}.
\begin{figure}%
\centering
\homotopypicture
\captionsetup{width=115mm}
\caption{%
Left: 
The figure shows the interval $[0,1]$ times a patch on $\Mclosed$.
The two lines indicate the path $[0,1]\ni s\mapsto h_a(s,(t,y))$
for $a=0$ respectively for a general $a\neq0$,
where $s=1$ corresponds to the starting point $(t,y)$.
Depending on the value of $a$, the paths approach different
boundary points. 
Right: 
For a starting point $(t',y)$ with $t'<t$,
the angle between the two paths increases, 
and they approach the same boundary points for all 
values of $t'$, including when $t'=0$.}
\label{fig:homotopypicture}
\end{figure}%
\begin{remark}
For each $a$ the map $h_a$ is a topological homotopy between the identity 
$h_a(1,\cdot)$ and the map $h_a(0,\cdot)$. 
The image of $h_a(0,\cdot)$ 
is contained in the boundary component $\McollarBC_0=\{0\}\times\Mclosed$,
but its restriction to $\McollarBC_0$
is not the identity map for general $a\neq0$,
see Figure \ref{fig:homotopypicture}.
This property will be important to conclude 
that the chain homotopy \eqref{eq:Gcollarzera} 
below satisfies \ref{item:CollarAnalytic} in Lemma \ref{lem:HomotopyCollar}.
Geometrically it means that the average in \eqref{eq:Gcollarzera}
does not degenerate at the boundary,
in the sense that the Schwartz kernel is 
smoothed out including at the boundary.
By contrast, the average over translations in the
Bogovskii operator \eqref{eq:translaverage} 
can be seen to degenerate,
leading to non-optimal mapping properties.
\end{remark}
\newcommand{\Ps}{j}
Fiber integration is the map 
$\pi_*:\Omega([0,1]\times\Mcollar)\to \Omega(\Mcollar)$
given by 
\[
\pi_{*} \omega = \int_0^1 \Ps_s^* (\intermult_{\p_{s}}(\omega))\,ds
\]
Here $s$ is the coordinate on $[0,1]$,
$\intermult$ denotes interior multiplication,
and $\Ps_s^*$ is the pullback along the map 
$\Ps_s:\Mcollar\to [0,1]\times\Mcollar$ given by $\Ps_s(t,y) = (s,(t,y))$.
\begin{definition}\label{def:Gcollardef}
For every $a\in\R^{\nVF}$ define the $\R$-linear map 
\begin{align}
\Gcollar_a = \pi_* h_a^* \;:\;\;
\Omega^k(\Mcollar,\McollarBC_0)\to \Omega^{k-1}(\Mcollar,\McollarBC_0)
\nonumber
\end{align}
where $k$ is any integer.
For every smooth function 
$\zeta:\R^{\nVF}\to [0,\infty)$ with compact support 
and $\int_{\R^{\nVF}} \zeta(a)\,da= 1$ define
\begin{align}
\Gcollar_{\zeta} = \int_{\R^{\nVF}} \zeta(a) \Gcollar_a \,da
\;:\;\;\Omega^k(\Mcollar,\McollarBC_0)\to \Omega^{k-1}(\Mcollar,\McollarBC_0)
\label{eq:Gcollarzera}
\end{align}
The dependence on the fixed vector fields
$V_1,\dots,V_{\nVF}$ is suppressed in the notation.
\end{definition}
The map $\Gcollar_a$ indeed maps relative forms 
to relative forms, using the fact that 
$h_a\circ \Ps_s\circ i$
maps to $\McollarBC_0$, where $i:\McollarBC_0 \hookrightarrow \Mcollar$ is the natural inclusion.%
\begin{lemma}\label{lem:GaGzetaSatisfiesP1P2}
For all $\zeta$ the map $\Gcollar_\zeta$ satisfies \ref{item:CollarHomotopy}
and \ref{item:CollarSupport} in Lemma \ref{lem:HomotopyCollar}.
\end{lemma}
The property \ref{item:CollarHomotopy} would fail if
$\Gcollar_\zeta$ was defined on all forms
rather than just relative forms.
We also note that there is ample freedom in the 
choice of $V_1,\dots,V_{\nVF}$ and $\zeta$,
which may be used to refine the support property \ref{item:CollarSupport}.

\begin{proof}
It suffices to check the lemma for $\Gcollar_a$.
\ref{item:CollarHomotopy}: 
Following the proof of the Poincar\'e lemma one obtains
$\one-h_{a}(0,\cdot)^* = \ddR \Gcollar_a+\Gcollar_a \ddR $,
where $h_{a}(0,\cdot)^*$ vanishes on relative forms.
\ref{item:CollarSupport}:
The value of $\Gcollar_a(\omega)$ at $(\tcoord,y)\in\Mcollar$ is
\[ 
(\Gcollar_a(\omega))(\tcoord,y) =\textstyle \int_0^1  
\Ps_s^* \intermult_{\p_s} (h_a^*\omega)(s,(\tcoord,y))\,ds
\]
where by \eqref{eq:hadef} the form $\omega$
is only evaluated at points $(t',y')$ with $t'\le t$.
\qed
\end{proof}
In the next two sections we show that for all $\zeta$
with support in a sufficiently small ball around the origin,
$\Gcollar_{\zeta}$ satisfies 
\ref{item:CollarAnalytic} in Lemma \ref{lem:HomotopyCollar}.
This uses the average over $a$.
To illustrate this we first consider a toy 
case where $\Mclosed$ is parallelizable.
\begin{example}\label{ex:Torus}
Let $\Mclosed=(\R/\Z)^{\CdimNEW}$ be the 
$\CdimNEW$-dimensional torus with standard coordinates $y$. 
Choose $\nVF=\CdimNEW$ and $V_{i}=\p_{y^i}$, which yields $\phi_a(y) = y+a$.
To show that $\Gcollar_\zeta$
satisfies \ref{item:CollarAnalytic} in Lemma \ref{lem:HomotopyCollar}
one must show that each component of $\cutoff\Gcollar_\zeta\cutoff$
with respect to a $C^\infty(\Mcollar)$-basis of
$\Omega(\Mcollar,\McollarBC_0)$
is in 
\smash{$\bPsdo^{-1,\order{(0,\emptyset,0)}}(\Mcollar)$}.
For concreteness we only consider one such component: 
For every $f\in C^\infty(\Mcollar)$ consider 
in $G_{\zeta}(fdt\wedge dy^1\wedge\dots\wedge dy^{\CdimNEW})$
the coefficient function of $tdy^1\wedge\cdots\wedge dy^{\CdimNEW}$,
and denote it by $g_\zeta(f)$.
This is a map $g_\zeta:C^\infty(\Mcollar)\to C^\infty(\Mcollar)$.
Explicitly, 
\begin{align*} 
g_{\zeta}(f)(t,y)
&=  \int_{\R^{\CdimNEW}}\zeta(a) \Big(\int_0^1 f(st,y+(1-s)a)\,ds\Big)\,da
\intertext{%
Assume that $\supp(\zeta)$ is contained
in a ball with small radius $\eps>0$ around the origin.
Then we can substitute $(s,a)$ with $(t',y')=(st,y+(1-s)a)$,
obtaining}
g_{\zeta}(f)(t,y)
&= 
\int_{B_\eps(y)}  \Big(\frac{1}{t}\int_0^t 
\frac{1}{(1-u)^{\CdimNEW}}
\zeta\Big(\frac{y'-y}{1-u}\Big) f(t',y')\,dt'\Big)\,dy'
\intertext{%
where we set $u=t'/t$,
and where $B_\eps(y)$ is the ball of radius $\eps$ around $y$.
From this formula one can easily read off the Schwartz kernel $K$ of $g_{\zeta}$, which is}
K(t,y,t',y') &= 
\chi_{B_\eps(y)}(y') \heaviside(1-u)
\frac{1}{(1-u)^{\CdimNEW}}
\zeta\Big(\frac{y'-y}{1-u}\Big)\,
\frac{1}{t} |dtdydt'dy'|^{\frac12}
\end{align*}
as a section of \eqref{eq:kerspace2} with $E$ and $F$
the trivial bundle, and for the purpose of
this example we have identified $\KB(E,F)$
with the trivial bundle using $|dtdy|$ and $|dt'dy'|$
as reference densities.
Here $\heaviside$ is the Heaviside function,
in particular the kernel vanishes for $u>1$.
To check $\cutoff g_{\zeta}\cutoff\in \smash{\bPsdo^{-1,\order{(0,\emptyset,0)}}(\Mcollar)}$
we pullback $K$ to the b-double space (see Figure \ref{fig:M2boundary})
where due to the support of $\cutoff$ we are 
away from $t=1$ and away from $t'=1$.
This pullback has support away from the left face,
where the projective coordinates $(t,y,u,y')$
are regular coordinates,
\final{and $t$ is a boundary defining function for the 
front face}, see Remark \ref{rem:blowup}.
Note that 
$t^{-1} |dtdydt'dy'|^{\frac12}
= t^{-\frac12} |dtdydudy'|^{\frac12}$.
From Lemma \ref{lem:KernelDiagonal} 
with $m=\CdimNEW$
it is then immediate that this pullback 
\final{is classical conormal of order $-1$}
along the b-diagonal $u=1$, $y=y'$,
up to the singular factor $t^{-\frac12}$,
consistent with \eqref{eq:bpsdodef}.
\end{example}

\subsection{Auxiliary geometric constructions}\label{sec:AuxGeom}
The constructions in this section will be used 
to make a substitution similar to that in Example \ref{ex:Torus}.
Fix an auxiliary Riemannian metric on $\Mclosed$.
Let ${\dist}:\Mclosed\times\Mclosed\to \R$ be the associated distance function.
For $\eps>0$ we denote 
\[ 
\diag_{\eps}(\Mclosed) = \{ (y,y')\in\Mclosed\times\Mclosed \mid \dist(y,y')<\eps \}
\]
For $y\in\Mclosed$ we denote by 
$B_{\eps}(y,\Mclosed)=\{y'\in\Mclosed\mid \dist(y,y')<\eps\}$
the open metric ball of radius $\eps$ around $y$,
and by $\exp_y: T_y\Mclosed \to \Mclosed$ the exponential map.
\step
Define the vector bundle map 
$\Lmap : \Mclosed\times \R^{\nVF} \to T\Mclosed$ 
that is fiberwise given by
\begin{equation*}
\Lmap_y:\R^{\nVF} \to T_y\Mclosed
\qquad
a \mapsto \sum_{i=1}^{\nVF} a^i V_i(y)
\end{equation*}
The assumption that the vector fields $V_i$ pointwise span the tangent space
implies that $\Lmap$ is fiberwise surjective.
Thus it defines a direct sum decomposition
$$\Mclosed\times\R^{\nVF} = \Kbundle \oplus \Kbundle^\bot$$ 
where $\Kbundle$ and $\Kbundle^\bot$ are 
the bundles whose fibers are
$\Kbundle_y = \ker(\Lmap_y)$
respectively the orthogonal complement
$\Kbundle_y^\bot = \ker(\Lmap_y)^\bot$,
using the standard inner product.
We denote by $B_r(\Kbundle_y)$, $B_r(\Kbundle_y^\bot)$
the open balls of radius $r>0$ around the origin.

\begin{lemma}\label{lem:DefFNEWSHORT}
There exist $\eps,r,r'>0$ such that the following holds. 
\begin{itemize}
\item 
For all
$(y,y')\in\diag_{\eps}(\Mclosed)$ and $a_0\in B_r(\Kbundle_y)$
there exists a unique $a_1\in B_{r'}(\Kbundle^{\bot}_y)$ such that
$\phi_{a_0\oplus a_1}(y) = y'$.
Denoting this $a_1$ by $\Amap(y,y',a_0)$ yields
\begin{equation}\label{eq:AdefSHORT}
\phi_{a_0\oplus \Amap(y,y',a_0)}(y) = y'
\end{equation}
One has $\Amap(y,y,0)=0$.
The section $\Amap(y,y',a_0)$ depends smoothly on $y,y',a_0$,
and extends smoothly to 
the closures of $\diag_{\eps}(\Mclosed)$ and of $B_r(\Kbundle_y)$.
\item 
For all $y\in \Mclosed$ and all $a_0\in B_r(\Kbundle_y)$ the map 
\begin{equation}\label{eq:AlocdiffeoSHORT}
\Amap(y,\,\cdot\,,a_0):\; B_\eps(y,\Mclosed) \to \Kbundle^\bot_y
\end{equation}
is a diffeomorphism onto its image.
Moreover for all $c\in[0,1]$ and all $a_0\in B_{cr}(\Kbundle_y)$ one has 
$B_{cr}(\Kbundle^\bot_y) \subset \Amap(y,B_{c\eps}(y,\Mclosed),a_0)$.
\end{itemize}
\end{lemma}
If the lemma holds for one triple of parameters $\eps,r,r'$,
then for every smaller $\eps$ there exists a smaller $r$
such that the lemma still holds.
\begin{proof}
This is a routine application of the implicit function theorem.\qed
\end{proof}
For every $y\in \Mclosed$ the Lebesgue measure $da$ on $\R^{\nVF}$
induces measures $da_0$ and $da_1$
on $\Kbundle_y$ respectively $\Kbundle_y^\bot$,
using the standard inner product.
Let $\eps,r,r'>0$ be as in Lemma \ref{lem:DefFNEWSHORT}.
Then for every $y\in \Mclosed$ and $a_0\in B_r(\Kbundle_y)$ define the one-density 
\begin{align}\label{eq:y'measuredefSHORT}
\mu_{y,a_0} = \Amap(y,\,\cdot\,,a_0)^* da_1 \in C^\infty(B_{\eps}(y,\Mclosed),\dens^1)
\end{align}
where we take the pullback
along the diffeomorphism \eqref{eq:AlocdiffeoSHORT}.
The density $\mu_{y,a_0}(y')$ depends smoothly on $y,y',a_0$
and it is nowhere zero.
\newcommand{\CA}{C_{\Amap}}
\step
In local trivializations $\Amap(y,y',a_0)$ is approximately $y'-y$,
which is made precise in the next lemma.
\begin{lemma}\label{lem:FfirstorderSHORT}
Let $\eps,r,r'>0$ be constants for which Lemma \ref{lem:DefFNEWSHORT}
holds and such that $\eps$ is strictly smaller 
than the injectivity radius\footnote{That is, 
for every $y\in\Mclosed$ the map
$\exp_y$ is a diffeomorphism $B_\eps(T_y\Mclosed)\to B_\eps(y,\Mclosed)$.}
of $\Mclosed$.
Then there exists $\CA>0$ such that for all
$(y,y')\in\diag_{\eps}(\Mclosed)$ and $a_0\in B_r(\Kbundle_y)$ one has
\begin{align}\label{eq:FirstorderSHORT}
\|\Amap(y,y',a_0) - (\Lmap_y|_{\Kbundle_y^{\bot}})^{-1}\exp_y^{-1}(y')\| 
\;\le \;
\CA(\dist(y,y')^2 + \|a_0\|^2)
\end{align}
\end{lemma}
\begin{proof}
By Taylor's theorem it suffices to show 
that the difference on the left hand side of \eqref{eq:FirstorderSHORT}
and its derivatives with respect to $y,y',a_0$
vanish when $y=y'$ and $a_0=0$.
It vanishes there because 
$\Amap(y,y,0)=0$ by Lemma \ref{lem:DefFNEWSHORT} and $\exp_y^{-1}(y)=0$.
The derivatives of $\Amap$ 
can be computed by differentiating \eqref{eq:AdefSHORT}.
For example, the Jacobian of $\Amap$ with respect to $y$
equals \smash{$-(\Lmap_y|_{\Kbundle_y^{\bot}})^{-1}$},
hence the Jacobian of the difference in \eqref{eq:FirstorderSHORT} vanishes.
Analogously for $y',a_0$.\qed
\end{proof}
\final{The quotient $(y'-y)/(1-u)$ in Example \ref{ex:Torus}
will be replaced by $\Amap(y,y',(1-u)a_0)/(1-u)$ in the general analysis,
which we control using the next lemma.}
\begin{lemma}\label{lem:asspupperboundFSHORT}
Set $C_{\Lmap} = \sup_{y\in\Mclosed} \|\Lmap_y|_{\Kbundle_y^\bot}\|$
using the $\ell^2$ matrix norm.
Let $\eps,r,r'>0$, $\CA>0$ be constants for which 
Lemma \ref{lem:DefFNEWSHORT} and \ref{lem:FfirstorderSHORT} hold and for which 
\begin{equation}\label{eq:smepsrSHORT}
\CA(2C_{\Lmap}\eps + r)\le 1
\end{equation}
Then for all $(y,y')\in\diag_{\eps}(\Mclosed)$, all $a_0\in B_r(\Kbundle_y)$ 
and all $s\in[0,1]$ one has
\begin{align}\label{eq:asspupperboundFSHORT}
\dist(y,y') \ge (2 C_{\Lmap}) r s
\;\;\Rightarrow\;\;
\left\|\Amap(y,y',sa_0)\right\| \ge rs
\end{align}
\end{lemma}
\begin{proof}
\newcommand{\PP}{f}
\newcommand{\QQ}{g}
By the reverse triangle inequality 
$\|\Amap(y,y',sa_0)\|\ge |\PP-\QQ|$ where
\begin{align*}
\PP 
&= \|(\Lmap_y|_{\Kbundle_y^\bot})^{-1}\exp_y^{-1}(y')\| 
\ge \tfrac{1}{C_{\Lmap}}\dist(y,y')\\
\QQ 
&= \|\Amap(y,y',sa_0)-(\Lmap_y|_{\Kbundle_y^\bot})^{-1}\exp_y^{-1}(y')\|\\
&\le
\CA(\eps \dist(y,y') + s r^2)
\end{align*}
using $\|\exp_y^{-1}(y')\| = \dist(y,y')$ and 
Lemma \ref{lem:FfirstorderSHORT}.
Therefore 
\begin{align*}
\PP-\QQ
&\ge 
(\tfrac{1}{C_{\Lmap}}- \CA\eps )\dist(y,y') - \CA s r^2
\intertext{
By \eqref{eq:smepsrSHORT} the first term on the right hand side is nonnegative.
Hence we can use the lower bound for $\dist(y,y')$, which yields
}
\PP-\QQ&\ge
rs \left(2- \CA(2C_{\Lmap} \eps + r) \right)
\ge
rs
\end{align*}
where we again use \eqref{eq:smepsrSHORT}.
Since this is nonnegative, it implies the claim.\qed
\end{proof}

\subsection{Analysis of Schwartz kernel}\label{sec:SchwartzKernelGCollar}
\newcommand{\Pimap}{Q}
Here we derive a formula for the Schwartz kernel of
$\Gcollar_{\zeta}$, similar to Example \ref{ex:Torus}. 
We then lift the kernel to the b-double space discussed in 
Section \ref{sec:bpsdo}, which will identify 
$\Gcollar_{\zeta}$ as a b-pseudodifferential operator
and prove \ref{item:CollarAnalytic} in Lemma \ref{lem:HomotopyCollar}.
\step
For the remainder of this section we fix
constants $\eps,r,r'>0$ for which Lemmas \ref{lem:DefFNEWSHORT},
\ref{lem:FfirstorderSHORT}, \ref{lem:asspupperboundFSHORT} hold.
We also fix a function $\zeta$ as in Definition \ref{def:Gcollardef}
that satisfies
\begin{equation}\label{eq:supportzeraassp}
\supp(\zeta)\subset B_{\frac{r}{2}}(\R^{\nVF})
\end{equation}

\newcommand{\pt}{p}
For every $a\in\R^{\nVF}$ and $s\in[0,1]$ define the auxiliary 
$\R$-linear map
\[ 
\Pimap(a,s) = \Ps_s^* \intermult_{\p_s} h_a^* 
:\;\;\Omega(\Mcollar,\McollarBC_0)
\to \Omega(\Mcollar,\McollarBC_0)
\]
using notation from Section \ref{sec:DefOfOp};
note that $\Gcollar_a(\omega)=\int_0^1 \Pimap(a,s)(\omega) \,ds$.
This map depends smoothly on $a,s$.
Moreover, for every $\pt\in\Mcollar$ there exists a unique 
$\Pimap_{\pt}(a,s)\in 
\Hom_{\R}(
\VBrelpt{}{\Mcollar}{\McollarBC_0}{h_a(s,\pt)},
\VBrelpt{}{\Mcollar}{\McollarBC_0}{\pt})$
such that for all $\omega\in \Omega(\Mcollar,\McollarBC_0)$,
\begin{equation}
\label{eq:px}
(\Pimap(a,s)(\omega))(\pt)
=
\Pimap_{\pt}(a,s)\big(\omega(h_a(s,\pt))\big)
\end{equation}
where $\VBrelpt{}{\Mcollar}{\McollarBC_0}{\pt}$
is the fiber of $\VBrel{}{\Mcollar}{\McollarBC_0}$
at $\pt$.
The map $\Pimap_{\pt}(a,s)$ depends smoothly on $\pt,a,s$.
\begin{lemma}\label{lem:Gformula}
For all $\omega\in\Omega(\Mcollar,\McollarBC_0)$ and $(t,y)\in\Mcollar$ with $t\neq0$
one has
\begin{align}\label{eq:Gformula}
\begin{aligned}
(\Gcollar_{\zeta}(\omega))(t,y)
&=
\int_{\Kbundle_y}\Big( 
\frac{1}{t}\int_0^t 
\int_{B_{\eps}(y,\Mclosed)} 
\frac{1}{(1-u)^{\dim\Mclosed}}\zeta(a_0\oplus a_1)
\\
&
\qquad\qquad
\Pimap_{t,y}(a_0\oplus a_1,u)(\omega(t',y'))
\,d\mu_{y,(1-u)a_0}(y')\,dt'
\Big)\,da_0
\end{aligned}
\end{align}
where we set $u=\frac{t'}{t}$ and 
$a_1=\frac{\Amap(y,y',(1-u)a_0)}{1-u}$, in particular, with \eqref{eq:AdefSHORT},
$$\Pimap_{t,y}(a_0\oplus a_1,u)\in 
\Hom_{\R}\left(
\VBrelpt{}{\Mcollar}{\McollarBC_0}{t',y'},
\VBrelpt{}{\Mcollar}{\McollarBC_0}{t,y}\right)
$$
The integral converges absolutely.
The integrand vanishes when $\dist(y,y')\ge\frac{\eps}{2}$.
\end{lemma}
\begin{proof}
By Fubini, \eqref{eq:supportzeraassp}, \eqref{eq:px} 
and \eqref{eq:hadef} we have
\begin{align}\label{eq:formulaGint}
\begin{aligned}
(\Gcollar_{\zeta}(\omega))(t,y)
=
\int_{{B_{\frac{r}{2}}(\Kbundle_y)}} 
\int_0^1
\int_{{B_{\frac{r}{2}}(\Kbundle_y^\bot)}}
\qquad\qquad\qquad\qquad\qquad\\
\qquad\qquad  \zeta(a)\Pimap_{t,y}(a,s)\big(\omega(st,\phi_{(1-s)a}(y))\big)
\,da_1 ds\, da_0
\end{aligned}
\end{align}
where $a=a_0\oplus a_1$.
For $t\neq0$ consider the map 
\begin{align}
\begin{aligned}
(0,t) \times B_{\eps}(y,\Mclosed)
&\to 
(0,1)\times \Kbundle_{y}^\bot
\\
(t',y') &\mapsto (\tfrac{t'}{t},\tfrac{\Amap(y,y',(1-t'/t)a_0)}{1-t'/t})
\end{aligned}
\label{eq:substitutemap}
\end{align}
By Lemma \ref{lem:DefFNEWSHORT} this is a diffeomorphism onto its image,
and the image contains $(0,1)\times B_{\frac r2}(\Kbundle_y^\bot)$.
Hence we can substitute, which yields \eqref{eq:Gformula},
using \eqref{eq:AdefSHORT} and the definition of $\mu$
in \eqref{eq:y'measuredefSHORT}.
The image of the restriction of \eqref{eq:substitutemap}
to $(0,t) \times B_{\frac\eps2}(y,\Mclosed)$
also contains $(0,1)\times B_{\frac r2}(\Kbundle_y^\bot)$,
hence the integrand of \eqref{eq:Gformula} vanishes
when $y'\notin B_{\frac\eps2}(y,\Mclosed)$.
Clearly the integral \eqref{eq:formulaGint} converges absolutely.
Since it agrees with \eqref{eq:Gformula} up to substitution,
also \eqref{eq:Gformula} converges absolutely.\qed
\end{proof}
Recall the discussion of the Schwartz kernels in \eqref{eq:kerspace1}.
\begin{lemma}\label{lem:KernelDef}
The map $\Gcollar_{\zeta}$ has a Schwartz kernel 
\begin{align*}
K_{\Gcollar_{\zeta}} &\in 
\distr\left(\Mcollar^2,\Hom(
\VBrel{}{\Mcollar}{\McollarBC_0},\VBrel{}{\Mcollar}{\McollarBC_0})\otimes\piR^*\dens^1\right)
\end{align*}
It is absolutely integrable on $\Mcollar^2$. 
As a function of $(t,y,t',y')\in\Mcollar^2$ it is given,
for $0<t\neq t'$, by
\begin{align}\label{eq:kernelformula}
\begin{aligned}
K_{\Gcollar_{\zeta}}(t,y,t',y')
&=
\int_{\Kbundle_y}\Big(
\frac{1}{t} \heaviside(t-t') 
\charact_{\diag_{\eps}(\Mclosed)}(y,y')
\frac{1}{(1-u)^{\dim\Mclosed}}\\
&\qquad
\zeta(a_0\oplus a_1) \Pimap_{t,y}(a_0\oplus a_1,u) \otimes 
|dt'|\mu_{y,(1-u)a_0}(y')\Big) \,da_0
\end{aligned}
\end{align}
where $u=\frac{t'}{t}$ and 
$a_1=\frac{\Amap(y,y',(1-u)a_0)}{1-u}$.
Here $\heaviside$ is the Heaviside function
and $\charact_{\diag_{\eps}(\Mclosed)}$ the characteristic
function. 
The kernel vanishes when $\dist(y,y')\ge\frac{\eps}{2}$.
\end{lemma}
\begin{proof}
One checks that \eqref{eq:kernelformula}
defines an absolutely integrable distribution
by using the substitution \eqref{eq:substitutemap}.
Then \eqref{eq:kernelformula} is the
Schwartz kernel of $\Gcollar_{\zeta}$
by Lemma \ref{lem:Gformula} and Fubini.
It vanishes for $\dist(y,y')\ge\frac{\eps}{2}$ because
the integrand of \eqref{eq:Gformula} vanishes there.\qed
\end{proof}
\begin{lemma}\label{lem:Gcollarpsdo}
The map $\Gcollar_{\zeta}$ satisfies \ref{item:CollarAnalytic} in Lemma \ref{lem:HomotopyCollar}.
\end{lemma}
\newcommand{\FF}{f}
\begin{proof}
Abbreviate $E=\VBrel{}{\Mcollar}{\McollarBC_0}$
and $\CdimNEW = \dim\Mclosed$.
The Schwartz kernel of $\cutoff\Gcollar_\zeta\cutoff$ is
\begin{equation*}
K 
= 
(\piL^*\cutoff)(\piR^*\cutoff)K_{\Gcollar_{\zeta}}
\in
\distr\big(\Mcollar^2, \Hom(E,E) \otimes \piR^*\dens^1\big)
\end{equation*}
with $K_{\Gcollar_{\zeta}}$ as in Lemma \ref{lem:KernelDef}.
As explained in \eqref{eq:kerspace2}, this is equivalently a section 
\begin{equation} \label{eq:Ksec}
K \in \distr\big(\Mcollar^2, \KB(E,E) \otimes \dens^{\frac12}\big)
\end{equation}
Fix a nowhere vanishing reference density 
$\gammaref\in C^\infty(\Mcollar,\dens^{\frac12})$.
For $(t,y,t',y')\in \Mcollar^2$ with $(y,y')\in\diag_\eps(\Mclosed)$
and for $a_0\in B_r(\Kbundle_y)$ define
$$
\Gamma(a_0,t,y,t',y')
=
\tfrac{|dt'|\mu_{y,a_0}(y')}{\gammaref^{2}(t',y')}
$$
where $\gammaref^2$ is the second tensor power of 
$\gammaref$, which is a one-density.
Then $\Gamma$ is a smooth nowhere vanishing function
by \eqref{eq:y'measuredefSHORT}.
We factor $K$ in \eqref{eq:Ksec} as 
\begin{equation}\label{eq:Kfactor} 
K = \tilde{K} \otimes \tfrac{1}{t}(\piL^*\gammaref)(\piR^*\gammaref) 
\end{equation}
where
$\tilde{K}\in \distr(\Mcollar^2, \KB(E,E))$. 
By Lemma \ref{lem:KernelDef} this is given,
for $0<t\neq t'$, by
\begin{align}\label{eq:Ktildeformula}
\begin{aligned}
\tilde{K}(t,y,t',y')
&=
\int_{\Kbundle_y}\cutoff(t)\cutoff(t')\heaviside(t-t') 
\charact_{\diag_{\eps}(\Mclosed)}(y,y')
\frac{1}{(1-u)^{\CdimNEW}}\\
&\qquad
\zeta(a_0\oplus a_1) 
\Gamma((1-u)a_0,t,y,t',y')
\tilde{\Pimap}_{t,y}(a_0\oplus a_1,u)\, da_0
\end{aligned}
\end{align}
with $u=\frac{t'}{t}$ and $a_1=\frac{\Amap(y,y',(1-u)a_0)}{1-u}$,
and $\tilde{\Pimap}_{t,y}(a_0\oplus a_1,u)$ given by
pre- and post- composing $\Pimap_{t,y}(a_0\oplus a_1,u)$
with multiplication by $\gammaref(t',y')$
respectively $\frac{1}{\gammaref(t,y)}$.

\begin{figure}
\centering
\begin{tikzpicture}[scale=0.9]
\def\s{2.7};
\draw[white,fill=gray!40] (0,0) -- (\s-0.4,0) -- (\s-0.4,\s-0.4);
\draw[white] (1,0) arc (0:90:1) node[midway,anchor=north east] {\footnotesize $\bff{0}$};
\filldraw[white] (0,0) circle (1);
\draw[very thick,-{Latex[length=1.5mm,width=1.5mm]}] (1,0) -- (\s+0.3,0) node[midway,anchor=north] {\footnotesize $\brf{0}=\{u=0\}$}
node[anchor=west,xshift=-1mm] {\footnotesize $t$};
\draw[very thick,-{Latex[length=1.5mm,width=1.5mm]}] (0,1) -- (0,\s+0.3) node[midway,anchor=east] {\footnotesize $\blf{0}$}
node[anchor=south,yshift=-1mm] {\footnotesize $t'$};
\draw[very thick,dashed] (\s,0) -- (\s,\s);
\draw[very thick,dashed] (0,\s) -- (\s,\s);
\draw[thick] (0.7,0.7) -- (\s-0.1,\s-0.1) node[midway,anchor=south,xshift=-4mm,yshift=-1mm] {\footnotesize $u=1$};
\draw[-,very thick] (1,0) arc (0:90:1) node[midway,anchor=north east] {\footnotesize $\bff{0}=\{t=0\}$};
\end{tikzpicture}
\captionsetup{width=115mm}%
\caption{The figure illustrates the subspace
$\McollarbOpen\simeq[[0,1)^2,(0,0)]\times\Mclosed^2$ of $\Mcollarb$,
where the factor $\Mclosed^2$ is suppressed.
The support of $\blowdown^*K$ is contained in the gray region.}
\label{fig:KernelPicture}
\end{figure}

Define the b-double space $\blowdown:\Mcollarb\to\Mcollar^2$ 
as in \eqref{eq:defbMfd},
and for $i=0,1$ define $\bff{i},\blf{i},\brf{i}$
as in \eqref{eq:facesdef} with $\Mfd=\Mcollar$ and
$\MfdBC_i=\{i\}\times\Mclosed$.
The kernel $K$ vanishes in a neighborhood of $t=1$ and of $t'=1$,
hence its pullback to $\Mcollarb$ has compact support in 
the subset $\McollarbOpen\subset\Mcollarb$ given by 
$$\McollarbOpen
=
[([0,1)\times\Mclosed)^2,\McollarBC_0^2]
\simeq
[[0,1)^2,(0,0)]\times\Mclosed^2
$$
See Figure \ref{fig:KernelPicture}.
In particular it vanishes around $\bff{1},\blf{1},\brf{1}$.
Furthermore the kernel $K$ vanishes for $t<t'$, hence
its pullback to $\Mcollarb$ vanishes around $\blf{0}$.
This explains in particular 
the empty index at the left face in \eqref{eq:cutgcur}.
On $\McollarbOpen\setminus\blf{0}$ the projective coordinates
$(t,y,u,y')$ 
with \smash{$u=\frac{t'}{t}$} are regular coordinates.
In these coordinates:
\begin{itemize}
\item 
$t$ and $u$ are boundary defining functions for 
$\bff{0}$ respectively $\brf{0}$.
\item 
The b-diagonal is given by $u=1$, $y=y'$.
\item 
The blowdown map is given by
$\blowdown(t,y,u,y')=(t,y,tu,y')$.
\end{itemize}

Therefore for \eqref{eq:cutgcur} it remains to show that 
\begin{equation}\label{eq:Kgoal}
\sqrt{t}\blowdown^*K 
\in 
\bconormal^{-1}\big(\Mcollarb,\bdiag;\blowdown^*(\KB(E,E))\otimes \dens^{\frac12}\big)
\end{equation}
Using \eqref{eq:Kfactor} we have 
$$\sqrt{t}\blowdown^*K 
=
\blowdown^*\tilde{K} \otimes \tfrac{1}{\sqrt{t}}\blowdown^*((\piL^*\gammaref)(\piR^*\gammaref))$$
where the second factor 
is a smooth nowhere vanishing half-density on $\McollarbOpen\setminus\blf{0}$.
To see this use $\blowdown^*|dtdt'| = t|dtdu|$.
Hence for \eqref{eq:Kgoal} it remains to show that
\begin{equation}\label{eq:Ktildeconormal}
\blowdown^*\tilde{K} \in 
\bconormal^{-1}\big(\Mcollarb,\bdiag;
\blowdown^*\KB(E,E)\big)
\end{equation}
Using \eqref{eq:Ktildeformula}, for
$(t,y,u,y')\in\McollarbOpen\setminus\blf{0}$ with 
$u\neq1$ we have 
\begin{align}\label{eq:formulabetak}
\begin{aligned}
(\blowdown^*\tilde{K})(t,y,u,y')
&=
\int_{\Kbundle_y}\cutoff(t)\cutoff(tu) \heaviside(1-u) 
\charact_{\diag_{\eps}(\Mclosed)}(y,y')
\frac{1}{(1-u)^{\CdimNEW}} \\
&\quad\zeta(a_0\oplus a_1)
\Gamma((1-u)a_0,t,y,tu,y')
\tilde{\Pimap}_{t,y}(a_0\oplus a_1,u)\,da_0
\end{aligned}
\end{align}
where \smash{$a_1=\frac{\Amap(y,y',(1-u)a_0)}{1-u}$}.
Then by Lemma \ref{lem:asspupperboundFSHORT} and \eqref{eq:supportzeraassp} 
and by Lemma \ref{lem:KernelDef},
\begin{align}\label{eq:supp}
\begin{aligned}
\supp(\blowdown^*\tilde{K})
\subset
\big\{
(t,y,u,y')\in \McollarbOpen\setminus\blf{0} \mid 
\ &0\le u\le 1\\
  &\dist(y,y')\le \tfrac{\eps}{2}\\
  &\dist(y,y')\le 2C_{\Lmap}r(1-u)
\big\}
\end{aligned}
\end{align}
where $C_{\Lmap}>0$ is a constant defined in Lemma \ref{lem:asspupperboundFSHORT}.
We claim that 
\begin{equation}\label{eq:K'Udiagsmooth}
\text{$\blowdown^*\tilde{K}$ is smooth on 
$\Mcollarb\setminus\bdiag$}
\end{equation}
To check this let $p=(t,y,u,y')\in\supp(\blowdown^*\tilde{K})\setminus\bdiag$.
Then \eqref{eq:supp} and $p\notin\bdiag$ imply $u<1$, 
hence $\blowdown^*\tilde{K}$ is smooth around $p$ because 
$\cutoff$, $\zeta$, $\Gamma$, $\tilde\Pimap$, $\Amap$ are smooth,
using Lemma \ref{lem:DefFNEWSHORT}.

We now analyze $\blowdown^*\tilde{K}$ near the b-diagonal,
where we use the blowup 
$\blowdowndiag:[\Mcollarb,\bdiag]\to\Mcollarb$ in \eqref{eq:bMfddiagBU}.
Let 
$\bdfdiag$ be a boundary defining function for 
$\blowdowndiag^{-1}(\bdiag)$.
Then by Lemma \ref{lem:KernelDiagonal}
with $m=\CdimNEW$, for \eqref{eq:Ktildeconormal} it suffices to show that
\begin{equation*}
\bdfdiag^{\CdimNEW}(\blowdowndiag^*\blowdown^*\tilde{K})
\in 
C^\infty\big([\Mcollarb,\bdiag],\blowdowndiag^*\blowdown^*\KB(E,E)\big)
\end{equation*}
By \eqref{eq:K'Udiagsmooth} this section is smooth 
away from $\blowdowndiag^{-1}(\bdiag)$.
We now show that for all $p\in\bdiag$ one has 
\begin{align}\label{eq:smoothatp}
\text{$\bdfdiag^{\CdimNEW}(\blowdowndiag^*\blowdown^*\tilde{K})$
is smooth around $\blowdowndiag^{-1}(p)$}
\end{align}
By \eqref{eq:supp} we can assume that $p\in\McollarbOpen\setminus\blf{0}$.
We denote $\blowdown(p) = (t_0,y_0,t_0,y_0)$.

To show this we fix local coordinates around $y_0\in\Mclosed$
such that $y_0$ corresponds to the coordinate origin.
By abuse of notation we will not distinguish between 
points and their coordinates, so $y_0=0$.
The coordinates around 
$y_0\in\Mclosed$ induce coordinates around $(y_0,y_0)\in\Mclosed^2$.
We also fix local trivializations of the bundles $\Kbundle$ and
$\Kbundle^\bot$ around $y_0$ that respect the inner product on the fibers,
and a local trivialization of $\blowdown^*\KB(E,E)$ around $p$.

Relative to these trivializations, $\blowdown^*\tilde{K}$ 
is a matrix whose entries are distributions around $p$.
Let $\kappa$ be one such entry. 
Then, for $Z\subset\McollarbOpen\setminus\blf{0}$ a sufficiently small open neighborhood 
of $p$, there exists a smooth function 
\[ 
\FF:\; \R^{\nVF-\CdimNEW}\times\R^{\CdimNEW}\times Z \to \R
\]
whose support in the first and second arguments
is contained in the ball of radius $r/2$ centered at the origin,
and such that on $Z$ one has 
\[ 
\kappa(t,y,u,y')
= 
\heaviside(1-u)\frac{1}{(1-u)^{\CdimNEW}}\int_{\R^{\nVF-\CdimNEW}}
\FF(a_0,\frac{\Amap(y,y',(1-u)a_0)}{1-u},t,y,u,y')\,da_0
\]
The function $\FF$ is defined by the formula \eqref{eq:formulabetak},
using the trivializations. It is smooth
because $\cutoff$, $\zeta$, $\Gamma$, $\tilde\Pimap$ are smooth,
and it has the given support by \eqref{eq:supportzeraassp}.

Fix an auxiliary smooth function $\psi:\R^{2}\setminus0\to [0,1]$
that satisfies $\psi(\lambda w)=\psi(w)$ for all $\lambda>0$ and $w\in\R^{2}\setminus0$, 
that is equal to $1$ when $|w_1|\le w_2$, and that vanishes
identically when $|w_1|\ge 2 w_2$.
By \eqref{eq:supp} there exists $c>0$ such that
$$
\text{$\psi(\|y-y'\|,c(1-u))=1$ on the support of $\kappa$}
$$ 
where, as explained above, we identify points with their coordinates.
Then for \eqref{eq:smoothatp} it suffices to show that each of%
\begin{subequations}\label{eq:smoothblowup123}
\begin{align}
&\blowdowndiag^*\big(\psi(\|y-y'\|,c(1-u)) \heaviside(1-u)\big)\label{eq:smoothblowup1} \\
\bdfdiag&\blowdowndiag^*\big(\psi(\|y-y'\|,c(1-u)) \tfrac{1}{1-u}\big)
\label{eq:smoothblowup2}\\
&\blowdowndiag^* 
\FF(a_0,\tfrac{\Amap(y,y',(1-u)a_0)}{1-u},t,y,u,y')
\label{eq:smoothblowup3}
\end{align}
\end{subequations}
is smooth around $\blowdowndiag^{-1}(p)$,
where in \eqref{eq:smoothblowup3} we require joint smoothness
also in $a_0\in \R^{\nVF-\CdimNEW}$.
To check this we introduce coordinates $(x,z)$ around $p$
that are adapted to the b-diagonal. Concretely we set 
\begin{align*}
x = (t-t_0,y+y')
&&
z = (1-u,y-y')
\end{align*}
They range over a neighborhood of the origin in 
$H\times\R^{\Mdim}$ where $\Mdim=\CdimNEW+1$ and 
\begin{itemize}
\item 
$H = \R^{\Mdim}$ when $p\notin\bff{0}$.
\item 
$H = \{x\in\R^{\Mdim}\mid x^1\ge0\}$ when $p\in\bff{0}$.
\end{itemize}
The point $p$ corresponds to $x=z=0$,
and $\bdiag$ is locally given by $z=0$.

Relative to these coordinates,
an open neighborhood of $\blowdowndiag^{-1}(p)$ in 
$[\Mcollarb,\bdiag]$ is given by an open neighborhood of 
$0\times (0\times S^{\Mdim-1})$ in 
\begin{equation}\label{eq:Hblowup}
H \times [\R^{\Mdim},0]
\end{equation}
Here $[\R^{\Mdim},0]=[0,\infty)\times S^{\Mdim-1}$
using polar coordinates $z=r\omega$ with $r\in[0,\infty)$
and $\omega\in S^{\Mdim-1}$.
We set $\bdfdiag=r$ locally.
We can now show that \eqref{eq:smoothblowup123} are smooth,
when viewed as functions on \eqref{eq:Hblowup},
defined locally around $x=0$, $r=0$:
\begin{itemize}
\item 
For $r>0$ the functions \eqref{eq:smoothblowup1} and \eqref{eq:smoothblowup2} are
equal to, respectively,
$$\psi(\|\vec\omega\|,c\omega^1)\heaviside(\omega^1)
\qquad\qquad
\psi(\|\vec\omega\|,c\omega^1)\tfrac{1}{\omega^1}
$$
where $\vec{\omega}=(\omega^2,\dots,\omega^{\Mdim})$.
The function $\psi(\|\vec\omega\|,c\omega^1)$ vanishes around $\omega^1=0$,
hence these are smooth functions of $(r,\omega)\in [0,\infty)\times S^{\Mdim-1}$.
\item 
Consider \eqref{eq:smoothblowup3}.
By the support of $\FF$ we can assume $\|a_0\|\le r/2$.
Let $(J_{z}\Amap)(x)\in \R^{\CdimNEW\times\Mdim}$ 
be the Jacobian matrix of $\Amap(y,y',(1-u)a_0)$ with respect to 
$z$ and evaluated at $(x,0)\in H\times\R^{\Mdim}$.
By Lemma \ref{lem:FfirstorderSHORT} this is independent of $a_0$,
and for each $x$ 
we have $(J_{z^1}\Amap)(x)=0$ and the
square matrix $(J_{\vec{z}}\Amap)(x)$ is invertible,
where $\vec{z}=(z^2,\dots,z^{\Mdim})$.
By Taylor's theorem with remainder
and using polar coordinates
$z=r\omega$, for $r>0$ we obtain 
\[ 
\frac{\Amap(y,y',(1-u)a_0)}{1-u}
=
\frac{
(J_{\vec{z}}\Amap)(x)\vec{\omega} + r \sum_{i,j=1}^{\Mdim}\omega^i\omega^jR_{ij}(x,r\omega,a_0)}{\omega^1}
\]
where $R_{ij}$ are smooth functions with values in $\R^{\CdimNEW}$,
defined around $(x,z)=(0,0)$ and for $\|a_0\|\le r/2$.
Since $(J_{\vec{z}}\Amap)(x)$ is invertible,
the numerator on the right hand side is nonzero around $r=0$, $\omega^1=0$.
Since $\FF$ has compact support in the second argument,
this implies that \eqref{eq:smoothblowup3} vanishes 
around $r=0$, $\omega^1=0$, thus it is smooth there.
Away from $\omega^1=0$ the denominator is nonzero, 
hence \eqref{eq:smoothblowup3} is smooth because $\FF$ is smooth.
Hence \eqref{eq:smoothblowup3} is 
a smooth function of $(x,r,\omega,a_0)$ around $x=0$, $r=0$,
as claimed.
\end{itemize}
This concludes the proof of \eqref{eq:smoothatp}, 
and hence the proof of the lemma.\qed
\end{proof}
\begin{proof}[of Lemma \ref{lem:HomotopyCollar}]
Set $\Gcollar = \Gcollar_{\zeta}$ and use 
Lemma \ref{lem:GaGzetaSatisfiesP1P2} and 
Lemma \ref{lem:Gcollarpsdo}.\qed
\end{proof}
\section{Proof of Theorem \ref{thm:main}}\label{sec:prfthm1}
\newcommand{\Omegacc}{\tc{junglegreen}{\Omega}_{\tc{junglegreen}{\text{c}}}}
\newcommand{\GHE}{\tc{lcolor}{G}_{\tc{lcolor}{\p}}}
\newcommand{\GHEcc}{\tc{lcolor}{G}_{\tc{lcolor}{\text{c}}}}
\newcommand{\pHE}{\tc{lcolor}{P}_{\tc{lcolor}{\p}}}
\newcommand{\iHE}{\tc{lcolor}{I}_{\tc{lcolor}{\p}}}
The strategy to prove Theorem \ref{thm:main} is as follows.
We will first construct 
a homotopy equivalence and a contraction
(see Section \ref{sec:HomotopyEquivalence}):
\begin{subequations}\label{eq:HEsSummary}
\begin{equation}\label{eq:HEsSummary1}
\begin{tikzpicture}[baseline=(current  bounding  box.center)]
  \matrix (m) [matrix of math nodes, column sep = 23mm, minimum width = 4mm ]
  {
    \Omega(\Mfd,\p\Mfd) & \Omegacc(\Mfd) \\
  };
  \path[-stealth]
    (m-1-1) edge [transform canvas={yshift=1.5mm}] node [above] {\footnotesize $\pHE$} (m-1-2)
            edge [out=180+13,in=180-13,min distance=5mm] node
                 [left,xshift=-1mm] {
                 \footnotesize $\GHE$  homotopy 
                 } (m-1-1)
    (m-1-2) edge [left hook-stealth,transform canvas={yshift=-1mm}] node [below] {\footnotesize $\iHE$} (m-1-1)
            edge [out=-17,in=+17,min distance=5mm] node
                 [right,xshift=1.5mm
                 ] {
                 \footnotesize $\GHEcc$  homotopy
                 } (m-1-2);
\end{tikzpicture}
\end{equation}
\begin{equation}\label{eq:HEsSummary2}
\begin{tikzpicture}[baseline=(current  bounding  box.center)]
  \matrix (m) [matrix of math nodes, column sep = 23mm, minimum width = 4mm ]
  {
    \Omega(\Mfd,\p\Mfd) & \Harm_g(\Mfd) \\
  };
  \path[-stealth]
    (m-1-1) edge [transform canvas={yshift=1.5mm}] node [above] {\footnotesize $\phodge$} (m-1-2)
            edge [out=180+13,in=180-13,min distance=5mm] node
                 [left,xshift=-1mm] {
                 \footnotesize $\Ghodge$  homotopy 
                 } (m-1-1)
    (m-1-2) edge [left hook-stealth,transform canvas={yshift=-1.5mm}] node [below] {\footnotesize $\ihodge$} (m-1-1)
            edge [out=-17,in=+17,min distance=5mm] node
                 [right,xshift=1mm
                 ] {
                 \footnotesize $0$  homotopy
                 } (m-1-2);
\end{tikzpicture}
\end{equation}
\end{subequations}
Here $\Omegacc(\Mfd)$ are all
forms that vanish in an $\frac{\eps_0}{2}$-neighborhood of the boundary;
and $\Harm_g(\Mfd)\simeq H(\Mfd,\p\Mfd)$ are the 
relative harmonic forms for a smooth Riemannian metric $g$.
The homotopy equivalence \eqref{eq:HEsSummary1} is obtained 
using Lemma \ref{lem:HomotopyCollar} on collar neighborhoods,
with a homotopy $\GHE$ that has good analytic properties,
it preserves support and is b-pseudodifferential.
The contraction \eqref{eq:HEsSummary2} is obtained using the
Hodge decomposition,
with a homotopy $\Ghodge$ that does not have good analytic properties.
Homotopy equivalences can be composed, see Lemma \ref{lem:compHE}.
We will compose \eqref{eq:HEsSummary} as follows,
using \eqref{eq:HEsSummary1} twice:
\begin{equation}\label{eq:HEcompNEWNEW}
\Omega(\Mfd,\p\Mfd) 
\ \xleftrightarrow{\eqref{eq:HEsSummary1}\,}\ \Omegacc(\Mfd)
\ \xleftrightarrow{\eqref{eq:HEsSummary1}\,}\ \Omega(\Mfd,\p\Mfd)
\ \xleftrightarrow{\eqref{eq:HEsSummary2}\,}\ \Harm_g(\Mfd)
\end{equation}
This yields a new contraction from
$\Omega(\Mfd,\p\Mfd)$ to $\Harm_g(\Mfd)$.
The resulting homotopy on $\Omega(\Mfd,\p\Mfd)$
will be the map $\Gfull$ in Theorem \ref{thm:main},
explicitly,
\begin{equation}
\label{eq:gformula} 
\Gfull = \GHE + \iHE\GHEcc\pHE + \iHE\pHE \Ghodge \iHE\pHE
\end{equation}
Here $\Ghodge$ only appears
sandwiched between $\iHE\pHE$,
which maps to forms that vanish in an 
$\frac{\eps_0}{2}$-neighborhood of the boundary.
This will imply that $\Gfull$ 
inherits the good analytic properties of $\GHE$,
suitable for Theorem \ref{thm:main}.
The algebraic properties will follow 
from the defining properties of a contraction.

The homotopy $\Gfull$ can be used to solve
the inhomogeneous equation $\ddR\omega = \eta$,
where $\eta$ is a given exact relative form,
namely $\ddR(\Gfull\eta)=\eta$.
The three terms in \eqref{eq:gformula} can be viewed
as solving this equation in three steps:
\begin{align*}
\ddR (\GHE \eta) &=\eta -\iHE\pHE\eta\\
\ddR (\iHE\GHEcc\pHE \eta) &=\phantom{\eta-\ }\iHE\pHE\eta -(\iHE\pHE)^2\eta\\
\ddR (\iHE\pHE \Ghodge \iHE\pHE\eta)&=
\phantom{\eta -\iHE\pHE\eta-\;}
(\iHE\pHE)^2\eta
\end{align*}
In the first step the equation is solved near the boundary,
the error $\iHE\pHE\eta$ vanishes near the boundary. 
In the second step the error is modified to $(\iHE\pHE)^2\eta$,
which is then conveniently solved using $\Ghodge$ in the third step.

\final{%
We note that the specific composition \eqref{eq:HEcompNEWNEW}
is chosen partly for convenience to simplify the analysis.
We have not pursued alternative constructions composing fewer than three homotopy equivalences, since we saw no particular benefit.}

\subsection{%
Homotopy equivalence between\\
relative and compactly supported forms}\label{sec:HE1}
We construct the homotopy equivalence \eqref{eq:HEsSummary1}.

\begin{lemma}\label{lem:contractiontocompact}
Let $\Mfd$, $\bdf_i$ with $i=1\dots\nBC$, 
and $\eps_0$ be as in Theorem \ref{thm:main}.
Define
\[ 
\Omegacc(\Mfd) = 
\big\{ \omega\in\Omega(\Mfd)\mid 
\forall i=1\dots\nBC:\;
\omega|_{\bdf_{i}\le\frac{\eps_0}{2}}=0 \big\}
\]
which is a subcomplex of the de Rham complex.
Let $\iHE:\Omegacc(\Mfd)\hookrightarrow \Omega(\Mfd,\p\Mfd)$ be 
the inclusion.
Then there exist $\R$-linear maps 
\begin{align*}
\GHE :\; \Omega^k(\Mfd,\p\Mfd)&\to \Omega^{k-1}(\Mfd,\p\Mfd)\\
\pHE :\;\Omega^k(\Mfd,\p\Mfd)&\to \Omegacc^k(\Mfd)
%
\end{align*}
defined for every integer $k$,
with the following properties:
\begin{enumerate}[label=(a\arabic*)]
\item \label{item:HEHomotopy}
Algebraic control:
The restriction of $\GHE$ to $\Omegacc(\Mfd)$ 
maps to $\Omegacc(\Mfd)$.
We denote this restriction by $\GHEcc$.
The four-tuple $(\pHE,\iHE,\GHE,\GHEcc)$ is a homotopy equivalence
(see \eqref{eq:HE})
between $\Omega(\Mfd,\p\Mfd)$ and $\Omegacc(\Mfd)$.
\item \label{item:HESupport}
Support control: 
For all $0<\eps\le\smash{\eps_0}$,
all $i=1\dots\nBC$,
and all $\omega\in \Omega(\Mfd,\p\Mfd)$ one has 
$\omega|_{\bdf_i\le\eps}=0\, \Rightarrow\, (\GHE\omega)|_{\bdf_i\le\eps}=0$.
\item \label{item:HEAnalytic}
Analytic control: 
\begin{align*}
\GHE &\in \bPsdo^{-1,\order{(0,\emptyset,0)}}(\Mfd;\VBrel{}{\Mfd}{\p\Mfd})\\
\iHE\pHE &\in \Psdo^{0,\order{(\emptyset,0)}}(\Mfd;\VBrel{}{\Mfd}{\p\Mfd})
\end{align*}
\end{enumerate}
\end{lemma}
\begin{proof}
We first prove the lemma for $\dim\Mfd\ge2$ and
in the special case $\nBC=1$, denoting $\bdf_1=\bdf$. 
Then $\p\Mfd$ is a connected closed manifold of dimension $\ge1$.
Without loss of generality we can assume that $\eps_0=1$.
Then, by assumption, $\bdf^{-1}([0,1])$ is a collar neighborhood of $\p\Mfd$.
Fix an identification
\begin{equation}\label{eq:ident}
\bdf^{-1}([0,1]) \simeq [0,1]_{\bdf}\times\p\Mfd
\end{equation}
Fix a smooth cutoff $\cutoff:\Mfd\to[0,1]$
that only depends on the value of $\bdf$, 
that is equal to $1$ for $\bdf\le\frac12$,
and vanishes for $\bdf\ge\frac34$.
Fix a map 
$\Gcollar:
\Omega^k([0,1]\times\p\Mfd,\{0\}\times\p\Mfd)
\to 
\Omega^{k-1}([0,1]\times\p\Mfd,\{0\}\times\p\Mfd)$
as in Lemma \ref{lem:HomotopyCollar} applied to $\Mclosed=\p\Mfd$.
Define 
\begin{align}\label{eq:GHEdef}
\begin{aligned}
\GHE &= \cutoff \Gcollar \cutoff\\
\pHE &= \one - \ddR \GHE - \GHE \ddR 
\end{aligned}
\end{align}
Note that $\GHE$ maps to smooth forms
since $\cutoff=0$ for $\bdf\ge\frac34$.
We check that $\pHE$ indeed maps relative to compactly supported forms.
Let $\omega\in \Omega(\Mfd,\p\Mfd)$. Then 
\begin{align}
\ddR\GHE(\omega) &= (\ddR\cutoff)\wedge\Gcollar (\cutoff \omega)+\cutoff \ddR \Gcollar (\cutoff \omega)\label{eq:dcut1}\\
\GHE(\ddR\omega) &= 
\cutoff \Gcollar (\ddR (\cutoff \omega)) - \cutoff \Gcollar ((\ddR \cutoff)\wedge\omega)
\label{eq:dcut2}
\end{align}
For $\bdf\le\frac12$ we have $\cutoff=1$,
hence the first term in \eqref{eq:dcut1} vanishes there.
Furthermore the second term in \eqref{eq:dcut2} vanishes there,
using \ref{item:CollarSupport} in Lemma \ref{lem:HomotopyCollar}.
Hence
\[ 
\pHE(\omega)|_{\bdf\le\frac12}
= 
\omega - \cutoff (\ddR \Gcollar+\Gcollar \ddR) (\cutoff \omega)
=
(1-\cutoff^2)\omega 
=
0
\]
where for the second equality we use \ref{item:CollarHomotopy} 
in Lemma \ref{lem:HomotopyCollar}.

We check that \eqref{eq:GHEdef} have the stated
properties.
\ref{item:HESupport}: By \ref{item:CollarSupport} in Lemma \ref{lem:HomotopyCollar}.
\ref{item:HEHomotopy}: 
By \ref{item:HESupport} with $\eps=\frac12$
the restriction of $\GHE$ to $\Omegacc(\Mfd)$ 
maps to $\Omegacc(\Mfd)$.
We check the properties \eqref{eq:HEcond}.
Clearly $\iHE$ is a chain map;
and $\pHE$ is a chain map by $\pHE\ddR = \ddR - \ddR \GHE\ddR = \ddR\pHE$
using $\ddR^2=0$;
the remaining identities
are immediate from the definition of $\pHE$ and $\iHE$.
\ref{item:HEAnalytic}:
For $\GHE$ this follows from \ref{item:CollarAnalytic} in Lemma \ref{lem:HomotopyCollar}
and the fact that $\cutoff$ vanishes for $\bdf\ge\frac34$.
For $\iHE\pHE$ we first show that
\begin{align}\label{eq:ipfirst}
\iHE\pHE &\in \bPsdo^{0,\order{(0,\emptyset,0)}}(\Mfd;\VBrel{}{\Mfd}{\p\Mfd})
\end{align}
This holds separately for each of the three terms in 
$\iHE\pHE=\one-\ddR\GHE-\GHE\ddR$,
using Lemma \ref{lem:ddRNEW} and the composition rule in Lemma \ref{lem:composition}.
Since $\iHE\pHE$ maps to forms that vanish for $\bdf\le\frac12$,
its Schwartz kernel vanishes in a neighborhood of 
$\p\Mfd\times\Mfd\subset\Mfd^2$, which includes the corner $(\p\Mfd)^2$.
Hence its pullback to the b-double space
vanishes in a neighborhood of the front face.
Together with \eqref{eq:ipfirst} this implies 
\ref{item:HEAnalytic} for $\iHE\pHE$.
This concludes the proof for $\nBC=1$.

The proof for $\nBC\ge2$ is analogous,
using the assumption that the collars are disjoint.
The definition of $\GHE$ in \eqref{eq:GHEdef}
is replaced by $\GHE = \sum_{i=1}^{\nBC} \cutoff_i\Gcollar_i\cutoff_i$ 
with $\cutoff_i$ a cutoff for the collar of $\MfdBC_i$,
and $\Gcollar_i$ as in Lemma \ref{lem:HomotopyCollar} applied to $\Mclosed=\MfdBC_i$.

We comment on the case $\dim\Mfd=1$,
where the manifold is the interval $[0,1]$.
Here one must apply Lemma \ref{lem:HomotopyCollar}
with $\Mcollar=[0,1]$ to each boundary component $0$ and $1$ 
of $\Mfd$ separately.
It is easy to see that Lemma \ref{lem:HomotopyCollar} also holds
when $\Mcollar=[0,1]$, here $\Gcollar$ is given by $\pi_*h^*$ 
using the topological homotopy $h(s,t)=st$.\qed
\end{proof}

\subsection{Contraction from relative forms to homology}\label{sec:HodgeNEW}

We construct the contraction \eqref{eq:HEsSummary2}
using the Hodge decomposition.
For a compact Riemannian manifold $(\Mfd,g)$ with boundary let
\[ 
\Harm_g(\Mfd) = 
\big\{ \omega\in \Omega(\Mfd,\p\Mfd)\mid \ddR\omega=0,\ \delta_g\omega=0 \big\}
\]
be the finite-dimensional space of relative harmonic forms.
Here $\delta_g$ is the codifferential,
given by $\delta_g = (-1)^{\dim(\Mfd)(k+1)+1}\hodge_g\ddR\hodge_g$
on $k$-forms, with 
$\hodge_g$ the Hodge star operator.
Recall that the inclusion of $\Harm_g(\Mfd)$
into the kernel of $\ddR$ induces an isomorphism 
$$
\Harm_g(\Mfd)\simeq H(\Mfd,\p\Mfd)
$$
\begin{lemma}\label{lem:hodgecontractionNEW}
Let $(\Mfd,g)$ be a smooth compact Riemannian manifold with boundary.
Let $\phodge : \Omega(\Mfd,\p\Mfd) \to \Harm_g(\Mfd)$
be the $L^2$-orthogonal projection
using the inner product induced by the metric, and 
let $\ihodge: \Harm_g(\Mfd)\hookrightarrow\Omega(\Mfd,\p\Mfd)$
be the inclusion.
Then there exists an $\R$-linear map 
\begin{align*}
\Ghodge :\;
\Omega^k(\Mfd,\p\Mfd)\to\Omega^{k-1}(\Mfd,\p\Mfd)
\end{align*}
defined for every integer $k$, with the following properties:
\begin{enumerate}[label=(a\arabic*)]
\item \label{item:AlgHodgeNEW}
Algebraic control:
Viewing $\Harm_g(\Mfd)$ as a complex with zero differential,
the triple $(\phodge,\ihodge,\Ghodge)$ is a contraction
(see \eqref{eq:contraction}) from 
$\Omega(\Mfd,\p\Mfd)$ to $\Harm_g(\Mfd)$.
\item \label{item:AnalysisHodgeNEW}
Analytic control:
If $\cutoff\in C^\infty(\Mfd)$ vanishes 
in a neighborhood of $\p \Mfd$ then
\begin{align*} 
\Ghodge\cutoff &\in 
\Psdo^{-1,\order{(0,\emptyset)}}(\Mfd;\VBrel{}{\Mfd}{\p\Mfd})\\
\ihodge\phodge\cutoff &\in \Psdo^{-\infty,\order{(0,\emptyset)}}(\Mfd;\VBrel{}{\Mfd}{\p\Mfd})
\end{align*}
\end{enumerate}
\end{lemma}
\begin{proof}
Denote by $i:\p\Mfd\hookrightarrow\Mfd$ the natural inclusion.
Denote by $\Ptay:\Omega(\Mfd) \to \Harm_g(\Mfd)$
the $L^2$-orthogonal projection, in particular
$\phodge$ is the restriction of $\Ptay$ to relative forms.
The Greens function of the Hodge Laplacian
$\Delta_g = -(\delta_g\ddR +\ddR\delta_g)$
is \cite[Chapter 5.9]{taylor1} the unique $\R$-linear map 
\begin{align}\label{eq:GtaymapNEW}
\Gtay :\; \Omega^k(\Mfd) \to \{ \omega \in\Omega^k(\Mfd)\mid i^*\omega = 0,\ i^*\delta_g\omega=0 \}
\end{align}
that satisfies
\begin{align}\label{eq:GtaypropNEW}
-\Delta_g \Gtay &= \one - \Ptay
\end{align}
and $\Gtay\Ptay =0$ and $\Ptay\Gtay =0$.
By \eqref{eq:GtaymapNEW} the 
composition $\delta_g\Gtay$ maps relative forms to relative forms,
and we now define 
$$
\Ghodge = (\delta_g\Gtay)|_{\Omega(\Mfd,\p\Mfd)}
$$
We show that this has the stated properties.
\ref{item:AlgHodgeNEW}:
We check \eqref{eq:HEcond}.
Clearly $\ihodge$ is a chain map;
$\phodge$ is a chain map because $\phodge\ddR=0$,
which follows from integrating by parts
where the boundary terms vanish for relative forms;
%
we have $\one-\ihodge \phodge = \ddR\Ghodge+\Ghodge\ddR$
by \eqref{eq:GtaypropNEW} 
and the fact that for all relative forms $\omega$ one has
$(\ddR\Gtay-\Gtay\ddR)\omega=0$ because the left hand side
is in the common kernel of $\Delta_g$ and $\Ptay$
and satisfies homogeneous relative boundary conditions;
clearly $\phodge\ihodge$ is the identity on $\Harm_g(\Mfd)$.
\ref{item:AnalysisHodgeNEW}:
We first show that
\[ 
\Gtay\cutoff \in \Psdo^{-2,\order{(0,\emptyset)}}(\Mfd;\wedge T^*\Mfd)
\]
Since $\Delta_g$ is elliptic, $\Gtay$ is a classical pseudodifferential
operator of order $-2$ in the interior of $\Mfd$.
Hence it remains to show that the Schwartz kernel of $\Gtay$ is smooth near every point in $\p\Mfd\times(\Mfd\setminus\p\Mfd)$.
Near such a point the kernel of the right hand side of \eqref{eq:GtaypropNEW} is smooth, 
therefore also the kernel of $\Gtay$ is smooth,
because it solves a regular elliptic boundary value problem 
with smooth data in the left factor of $\Mfd^2$, 
depending parametrically on the right factor.
Since $\delta_g$ is a first order differential operator we obtain 
$\delta_g\Gtay\cutoff \in \Psdo^{-1,\order{(0,\emptyset)}}(\Mfd;\wedge T^*\Mfd)$.
Together with the fact that $\Ghodge$ maps relative to relative forms
this implies \ref{item:AnalysisHodgeNEW} for $\Ghodge$.
The kernel of $\ihodge\phodge$ is smooth on $\Mfd^2$,
which implies \ref{item:AnalysisHodgeNEW} for $\ihodge\phodge$.\qed
\end{proof}

\subsection{Composition of homotopy equivalences}\label{sec:proof}
\newcommand{\phodgeprime}{\tc{lcolor}{P'}}
\newcommand{\ihodgeprime}{\tc{lcolor}{I'}}
\newcommand{\Ghodgeprime}{\tc{lcolor}{G'}}
\begin{proof}[of Theorem \ref{thm:main}]
Define the complex $\Omegacc(\Mfd)$ as in Lemma \ref{lem:contractiontocompact}.
Fix:
\begin{itemize}
\item 
Maps $\pHE,\iHE,\GHE,\GHEcc$ as in Lemma \ref{lem:contractiontocompact}.
\item 
A smooth Riemannian metric $g$ on $\Mfd$.\\
Then fix maps $\phodge,\ihodge,\Ghodge$ as in Lemma \ref{lem:hodgecontractionNEW}.
\end{itemize}

We compose the second and third homotopy equivalences in \eqref{eq:HEcompNEWNEW}.
By \ref{item:HEHomotopy} in Lemma \ref{lem:contractiontocompact},
\ref{item:AlgHodgeNEW} in Lemma \ref{lem:hodgecontractionNEW},
and Lemma \ref{lem:compHE}, the triple
\begin{align}\label{eq:HEcpthodgeNEW}
\big(\,\phodgeprime = \phodge \iHE\,,\ 
\ihodgeprime = \pHE \ihodge\,,\ 
\Ghodgeprime = \GHEcc + \pHE \Ghodge \iHE\,\big)
\end{align}
is a contraction
(see \eqref{eq:contraction})
from $\Omegacc(\Mfd)$ to $\Harm_g(\Mfd)$.
We now compose the first homotopy equivalence in \eqref{eq:HEcompNEWNEW} with \eqref{eq:HEcpthodgeNEW}.
By \ref{item:HEHomotopy} in Lemma \ref{lem:contractiontocompact},
the fact that \eqref{eq:HEcpthodgeNEW} is a contraction,
and Lemma \ref{lem:compHE}, the triple
\begin{align}
\label{eq:comp123NEW}
\big(\,\Pfull = \phodgeprime\pHE \,,\ 
\Ifull = \iHE \ihodgeprime \,,\ 
\Gfull = \GHE + \iHE \Ghodgeprime \pHE \,\big)
\end{align}
is a contraction from $\Omega(\Mfd,\p\Mfd)$ to $\Harm_g(\Mfd)$.
In particular, see \eqref{eq:HEcond},
\begin{subequations}
\begin{align}
&\text{$\Ifull,\Pfull$ are chain maps}\label{eq:IPchain}\\
&\one-\Ifull\Pfull=\ddR\Gfull+\Gfull\ddR \label{eq:IP}\\
&\one-\Pfull\Ifull=0\label{eq:PI}
\end{align}
\end{subequations}
We show that $\Gfull$ has the properties stated in Theorem \ref{thm:main}.
\ref{item:AlgFull}:
By \eqref{eq:IP} we have $\Pifull=\Ifull\Pfull$,
hence $\Pifull$ is a projection by \eqref{eq:PI},
and $\Pifull \ddR = 0$, $\ddR\Pifull = 0$
by \eqref{eq:IPchain} and the fact that the differential
on $\Harm_g(\Mfd)$ is zero.
\ref{item:SupportFull}:
By \eqref{eq:HEcpthodgeNEW} and \eqref{eq:comp123NEW},
\begin{align} \label{eq:Gfullformula}
\Gfull 
&= \GHE + \iHE\GHEcc\pHE + \iHE\pHE \Ghodge \iHE\pHE
\end{align}
The first term satisfies \eqref{eq:SupportFull}
by \ref{item:HESupport} in Lemma \ref{lem:contractiontocompact}.
The second and third terms both map to $\Omegacc(\Mfd)$, 
because $\pHE$ maps to $\Omegacc(\Mfd)$, 
$\GHEcc$ acts on $\Omegacc(\Mfd)$, 
and $\iHE$ is the inclusion.
Hence these two terms satisfy \eqref{eq:SupportFull} as well.
By 
\eqref{eq:HEcpthodgeNEW} and \eqref{eq:comp123NEW} we have
$\Pifull 
= \Ifull \Pfull 
= \iHE \pHE \ihodge  \phodge \iHE \pHE$, 
which maps to $\Omegacc(\Mfd)$ because $\pHE$ does.
\ref{item:AnalysisFull}:
We show that each of the three terms in \eqref{eq:Gfullformula} is in 
$\bPsdo^{-1,\order{(0,\emptyset,0)}}(\Mfd;\VBrel{}{\Mfd}{\p\Mfd})$:
\begin{itemize}
\item 
The first by \ref{item:HEAnalytic} in Lemma \ref{lem:contractiontocompact}.
\item 
The second is equal to $\GHE(\iHE\pHE)$, see \ref{item:HEHomotopy}
in Lemma \ref{lem:contractiontocompact}.
Thus it is in the given space by 
\ref{item:HEAnalytic} in Lemma \ref{lem:contractiontocompact}
and the composition rule in Lemma \ref{lem:composition}.
\item 
The third is equal to 
$(\iHE\pHE) (\Ghodge \cutoff) (\iHE\pHE)$
where $\cutoff\in C^\infty(\Mfd)$ is a cutoff
that vanishes identically in a neighborhood of $\p\Mfd$ 
and is equal to $1$ on the complement of $\cup_{i=1}^{\nBC}\bdf_i^{-1}([0,\frac{\eps_0}{4}])$.
By \ref{item:HEAnalytic} in Lemma \ref{lem:contractiontocompact},
\ref{item:AnalysisHodgeNEW} in Lemma \ref{lem:hodgecontractionNEW},
and the composition rule in Lemma \ref{lem:composition}
we have $(\iHE\pHE) (\Ghodge \cutoff)\in \Psdo^{-1,\order{(\emptyset,\emptyset)}}(\Mfd;\VBrel{}{\Mfd}{\p\Mfd})$.
Applying Lemma \ref{lem:composition} again yields the claim.
\end{itemize}
The statement for $\Pifull$ follows from 
$\Pifull = (\iHE \pHE)(\ihodge\phodge \cutoff)(\iHE\pHE)$, 
\ref{item:HEAnalytic} in Lemma \ref{lem:contractiontocompact},
\ref{item:AnalysisHodgeNEW} in Lemma \ref{lem:hodgecontractionNEW},
and Lemma \ref{lem:composition}.
Hence \ref{item:AnalysisFull} holds. \qed
\end{proof}
\section{Bogovskii type homotopy for de Rham on $\R^{\Mdim}$}\label{sec:Rdcor}
In this section we apply Theorem \ref{thm:main} 
to the radial compactification of $\R^{\Mdim}$ with $\Mdim\ge2$.
By definition, this is the compact manifold 
\begin{equation}\label{eq:radialcpt}
\Rrad{\Mdim}
=
\left( \R^{\Mdim} \sqcup ([0,\infty)\times S^{\Mdim-1})  \right)/{\sim}
\end{equation}
where $x\in\R^{\Mdim}$ and $(t,\omega)\in[0,\infty)\times S^{\Mdim-1}$
are equivalent, $x\sim(t,\omega)$, if and only if $tx=\omega$.
It is diffeomorphic to the closed unit ball in $\R^{\Mdim}$.
\step
Let $\jap{x} = \sqrt{1+\|x\|^2}$.
Then $\jap{x}^{-1}$ is smooth on $\Rrad{\Mdim}$,
and a boundary defining function for $\p\Rrad{\Mdim}$. 
The vector fields $\jap{x}\p_{x^1},\dots,\jap{x}\p_{x^{\Mdim}}$
are a basis of the b-vector fields $\bVF$ on $\Rrad{\Mdim}$,
see \eqref{eq:bvfbasis}.
This motivates the definition of the 
weighted b-Sobolev norms
$\|\cdot\|_{\jap{x}^{-\delta}\Hx^s(\R^{\Mdim})}$
on $\R^{\Mdim}$ in \eqref{eq:HxnormNEW}.
The smooth compactly supported functions
are dense in the corresponding weighted b-Sobolev
spaces $\jap{x}^{-\delta}\Hx^s(\R^{\Mdim})$.
\step
We denote by $\Omega_c(\R^{\Mdim})$ the 
space of forms with compact support in $\R^{\Mdim}$.
By the standard $dx$-basis we mean the
basis of $\Omega(\R^{\Mdim})$ given by all wedge products 
of zero or more of the $dx^1,\dots,dx^{\Mdim}$.

\begin{theorem}\label{thm:Mainx}
For every real number $R>0$ there exists an $\R$-linear map
\[ 
\Gfullx :\; 
\Omega^k(\Rrad{\Mdim},\p\Rrad{\Mdim})
\to
\Omega^{k-1}(\Rrad{\Mdim},\p\Rrad{\Mdim})
\]
defined for every integer $k$,
with the following properties:
\begin{enumerate}[label=(a\arabic*)]
\item \label{item:AlgFullx}
Algebraic control:
The map $\Pifullx$ defined by 
\begin{equation*}
\one-\Pifullx \;=\; \ddR\Gfullx + \Gfullx\ddR
\end{equation*}
satisfies $\Pifullx^2=\Pifullx$ and $\Pifullx\ddR =0$ and $\ddR\Pifullx=0$.
In particular, $\Pifullx$ is a projection onto 
a complement of the image of $\ddR$ in the kernel of $\ddR$.
It is equivalently given as follows:
There exists $\mu\in\Omega^{\Mdim}(\R^{\Mdim})$
with $\mu|_{\|x\|\ge R}=0$ and $\int_{\R^{\Mdim}}\mu=1$ such that 
for all $\omega\in \Omega^{k}(\Rrad{\Mdim},\p\Rrad{\Mdim})$,
\begin{align}\label{eq:piformula}
\Pifullx\omega
\;=\;
\begin{cases}
\mu \int_{\R^{\Mdim}}\omega & \text{if $k=\Mdim$}\\
0 & \text{if $k\le\Mdim-1$}
\end{cases} 
\end{align}
\item \label{item:SupportFullx}
Support control: 
For all $r\ge R$ and all 
$\omega\in \Omega(\Rrad{\Mdim},\p\Rrad{\Mdim})$,
\begin{align*}
\omega|_{\|x\|\ge r}=0 
\qquad&\Rightarrow\qquad
(\Gfullx\omega)|_{\|x\|\ge r}=0
\end{align*}
\item \label{item:AnalysisFullx}
Analytic control: 
For all $k=0\dots\Mdim$ and
$s,\delta\in\R$ with $\delta>k-\frac{\Mdim}{2}$ 
there exists a constant $C>0$ such that for all 
$\omega\in\Omega^k_c(\R^{\Mdim})$,
\begin{equation}
\|\Gfullx\omega\|_{\jap{x}^{-\delta+1}\Hx^{s+1}(\R^{\Mdim})}
\;\le\;
C
\|\omega\|_{\jap{x}^{-\delta}\Hx^{s}(\R^{\Mdim})}
\label{eq:estGx}
\end{equation}
where the norms are taken componentwise 
using the standard $dx$-basis.

\item \label{item:CommutatorFullx}
For all $k=0\dots\Mdim$ and $s,\delta\in\R$ with $\delta>k-\frac{\Mdim}{2}$,
all $i=1\dots\Mdim$ and all $\ell\in\Nzero$,
there exists $C>0$ such that for all $\omega\in\Omega^k_c(\R^{\Mdim})$,
\[ 
\left\|\left[\Gfullx,\jap{x}^{1-\ell}\Lie_{\p_{x^i}}\right]\omega\right\|_{\jap{x}^{-\delta+1-\ell}\Hx^{s+1}(\R^{\Mdim})}
\;\le\;
C
\|\omega\|_{\jap{x}^{-\delta}\Hx^{s}(\R^{\Mdim})}
\]
where the norms are taken componentwise using the standard $dx$-basis.
Here $[\,\cdot\,,\,\cdot\,]$ is the commutator, 
and $\Lie_{\p_{x^i}}$ is the Lie derivative of forms.
\end{enumerate}
\end{theorem}
Note that by \ref{item:AnalysisFullx} and by density,
the map $\Gfullx$ extends
to a bounded linear map between the corresponding weighted b-Sobolev spaces.

Note that $\Gfullx$ restricts to
a chain homotopy for the complex of compactly supported forms $\Omega_c(\R^{\Mdim})$.
It also restricts to a chain homotopy for the complex of forms 
whose components are Schwartz functions,
and the estimates \ref{item:AnalysisFullx} and \ref{item:CommutatorFullx} 
also hold for such forms with Schwartz components.
\step
For later reference observe that 
a $C^\infty(\Rrad{\Mdim})$-basis 
of $\Omega^k(\Rrad{\Mdim},\p\Rrad{\Mdim})$ is given by
\begin{equation}\label{eq:relbasNEW}
\jap{x}^{-(k+1)} dx^{i_1} \wedge\cdots\wedge dx^{i_k}\;,
\qquad
1\le i_1<\dots<i_k\le \Mdim
\end{equation}

\begin{proof}
Fix $\eps_0=\frac{2}{R}$. Fix a boundary defining function
$\bdf\in C^\infty(\Rrad{\Mdim})$ for $\p\Rrad{\Mdim}$
such that $\bdf^{-1}([0,\eps_0])$ is a collar neighborhood
in the sense of Theorem \ref{thm:main} and 
\begin{equation}\label{eq:bdf1/x}
\bdf(x) = \tfrac{1}{\|x\|} \qquad \text{on $\bdf^{-1}([0,\eps_0])$}
\end{equation}
One may set \smash{$\bdf(x) =(\|x\|^2+\cutoff(x)^2)^{-\frac12}$}
where \smash{$\cutoff:\Rrad{\Mdim}\to [0,\frac{R}{4}]$} is a smooth 
function that satisfies
$\cutoff(x)=0$ when $\|x\|\ge\frac{R}{4}$
and $\cutoff(x)=\tfrac{R}{4}$ when $\|x\|\le\frac{R}{8}$.

Let $\Gfullx$ be a map as in Theorem \ref{thm:main} 
applied to $\Mfd=\Rrad{\Mdim}$,
the boundary defining function $\bdf$, 
and $\eps_0=\frac{2}{R}$.
We check that this has the stated properties.

\ref{item:AlgFullx}:
The identities $\Pifullx^2=\Pifullx$, $\Pifullx\ddR =0$, $\ddR\Pifullx=0$
hold by \ref{item:AlgFull} in Theorem \ref{thm:main}.
We show \eqref{eq:piformula}. 
Recall from standard de Rham homology that 
\begin{align}\label{eq:Hkrel}
\dim H^k(\Rrad{\Mdim},\p\Rrad{\Mdim}) 
=
\begin{cases}
1 & \text{if $k=\Mdim$}\\
0 & \text{if $k\le\Mdim-1$}
\end{cases} 
\end{align}
Hence $\Pifullx$ vanishes in degree $k\le \Mdim-1$
and its restriction to degree $\Mdim$ has rank one.
Elements in $\image(\Pifullx)$ vanish on $\|x\|\ge R$
by \ref{item:SupportFull} in Theorem \ref{thm:main} and \eqref{eq:bdf1/x}.
Hence there exists $\mu$ as in \ref{item:AlgFullx},
and a linear map $p:\Omega^{\Mdim}(\Rrad{\Mdim},\p\Rrad{\Mdim})\to\R$,
such that 
$\Pifullx \omega = \mu\, p (\omega)$
for all $\omega\in\Omega^{\Mdim}(\Rrad{\Mdim},\p\Rrad{\Mdim})$.
Since $\image(\ddR)\subset\ker(p)$,
the map $p$ is the integral over $\R^{\Mdim}$, up to a multiplicative constant.
Since $\Pifullx^2=\Pifullx$ and $\int_{\R^{\Mdim}}\mu=1$, this multiplicative constant is one.
This shows \eqref{eq:piformula}. 

\ref{item:SupportFullx}: This is a direct consequence of 
\ref{item:SupportFull} in Theorem \ref{thm:main}, using \eqref{eq:bdf1/x}.

\ref{item:AnalysisFullx}:
By \ref{item:AnalysisFull} in Theorem \ref{thm:main} 
and Lemma \ref{lem:Mappingprop}, 
see also Remark \ref{rem:normsVB},
for all $s,\delta_0\in\R$
with $\delta_0>-\frac12$ 
there exists $C_0>0$ such that for all 
$\omega\in\Omega^k_c(\R^{\Mdim})$,
\[
\|\Gfullx\omega\|_{\jap{x}^{-\delta_0}\Hsob^{s+1}(\Rrad{\Mdim})}
\le 
C_0 \|\omega \|_{\jap{x}^{-\delta_0}\Hsob^s(\Rrad{\Mdim})}
\]
where the norms are taken componentwise 
using the $C^\infty(\Rrad{\Mdim})$-basis \eqref{eq:relbasNEW} of 
$\Omega(\Rrad{\Mdim},\p\Rrad{\Mdim})$.
To be sure, this is a different basis than used in \eqref{eq:estGx}.
Now the claim follows from Lemma \ref{lem:ComparabilityOfNorms} below
and setting $\delta = \delta_0+k+1-\frac{\Mdim+1}{2}$.

\ref{item:CommutatorFullx}:
Using \eqref{eq:relbasNEW} it is easy to see that 
$\jap{x}\Lie_{\p_{x^i}}:
\Omega(\Rrad{\Mdim},\p\Rrad{\Mdim})\to\Omega(\Rrad{\Mdim},\p\Rrad{\Mdim})$
is a first order b-differential operator, see \eqref{eq:firstorderb}.
Hence by Lemma \ref{lem:commutator},
\[ 
\left[\Gfullx,\jap{x}^{-\ell+1}\Lie_{\p_{x^i}}\right]
\;\in\;
\bPsdo^{-1,\order{(\ell,\emptyset,0)}}(\Rrad{\Mdim};\VBrel{}{\Rrad{\Mdim}}{\p\Rrad{\Mdim}})
\]
Now one can conclude \ref{item:CommutatorFullx} analogously
to the proof of \ref{item:AnalysisFullx}.\qed
\end{proof}
\begin{remark}\label{rem:CompareBog}
The restriction of $\Gfullx$ in Theorem \ref{thm:Mainx} to top degree forms
is a right inverse of 
the divergence operator on $\R^{\Mdim}$ for the Euclidean metric,
up to the usual integrability condition.
We compare this to the Bogovskii right inverse  \eqref{eq:translaverage}.
The Bogovskii inverse satisfies \ref{item:SupportFullx}.
It does not satisfy \ref{item:AnalysisFullx}, 
that is, there exists no constant $C>0$ such that 
\eqref{eq:estGx} holds for all compactly supported forms $\omega$, 
for no $s$ and $\delta$. To see this, one can apply the Bogovskii inverse
to a moving bump for which the norm on the right hand side is fixed,
and show that the left hand side grows as the bump
moves to infinity.
\new{The Bogovskii inverse does satisfy \ref{item:AnalysisFullx}
if one restricts to forms $\omega$ whose support is contained
in a fixed large ball centered at the origin.}
\end{remark}
\begin{lemma}[Equivalence of norms]\label{lem:ComparabilityOfNorms}
Let $\delta,s\in\R$ and $k=0\dots\Mdim$.
Recall the following b-Sobolev norms for 
compactly supported functions $f$ on $\R^{\Mdim}$:
\begin{itemize}
\item 
$\|f\|_{\Hsob^s(\Rrad{\Mdim})}$ is the norm \eqref{eq:Hsobnorm} with $\Mfd=\Rrad{\Mdim}$.
\item 
$\|f\|_{\Hx^s(\R^{\Mdim})}$ is the norm \eqref{eq:HxnormNEW}.
\end{itemize}
We use the following b-Sobolev norms for 
$k$-forms $\omega\in \Omega_c^k(\R^{\Mdim})$:
\begin{itemize}
\item 
$\|\omega\|_{\Hsob^s(\Rrad{\Mdim},\VBrel{k}{\Rrad{\Mdim}}{\p\Rrad{\Mdim}})}$
is the norm given by applying $\|\cdot\|_{\Hsob^s(\Rrad{\Mdim})}$ 
componentwise, using the $C^\infty(\Rrad{\Mdim})$-basis 
\eqref{eq:relbasNEW} of the relative forms
$\Omega^k(\Rrad{\Mdim},\p\Rrad{\Mdim})$.
\item 
$\|\omega\|_{\Hx^s(\R^{\Mdim},\wedge^kT^*\R^{\Mdim})}$
is the norm given by applying $\|\cdot\|_{\Hx^s(\R^{\Mdim})}$ 
componentwise, using the standard $dx$-basis of the forms $\Omega^k(\R^{\Mdim})$.
\end{itemize}
Then the following norms are equivalent:
\begin{align}
&\|\cdot\|_{\jap{x}^{-\delta}\Hsob^s(\Rrad{\Mdim})}
&\text{and}&
&&\|\cdot\|_{\jap{x}^{-(\delta-\smash{\frac{\Mdim+1}{2}})}\Hx^s(\R^{\Mdim})}
\label{eq:SobolevCompare1}\\
&\|\cdot\|_{\jap{x}^{-\delta}\Hsob^s(\Rrad{\Mdim},
\VBrel{k}{\Rrad{\Mdim}}{\p\Rrad{\Mdim}})}
&\text{and}&
&&\|\cdot\|_{\jap{x}^{-(\delta+k+1-\smash{\frac{\Mdim+1}{2}})}\Hx^s(\R^{\Mdim},\wedge^k T^*\R^{\Mdim})}
\label{eq:SobolevCompare2}
\end{align}
\end{lemma}
\begin{proof}
It suffices to consider $\delta=0$. 
It suffices to consider $s=0$,
using interpolation and duality, and the fact that 
$\jap{x}\p_{x^1},\dots,\jap{x}\p_{x^{\Mdim}}$ is a basis of $\bVF$.
Then \eqref{eq:SobolevCompare1} holds because $\jap{x}^{-(\Mdim+1)}dx^1\cdots dx^{\Mdim}$ is a smooth 
positive measure on $\Rrad{\Mdim}$;
for \eqref{eq:SobolevCompare2} 
use the change of basis between 
\eqref{eq:relbasNEW} and the standard $dx$-basis.
\qed
\end{proof}
For some applications, the following corollary is useful.
\begin{corollary}\label{cor:thm2}
For every $R>0$
there exists a map $\Gfullx$ as in Theorem \ref{thm:Mainx}
that in addition satisfies $\Pifullx\Gfullx=0$ and $\Gfullx\Pifullx=0$ and $\Gfullx^2=0$.
\end{corollary}
\begin{proof}
Let $\Gfullx_0$, $\Pifullx_0$ 
be as in Theorem \ref{thm:Mainx} (called $\Gfullx$, $\Pifullx$ there) and set 
$\Gfullx = \Gfullx_0\ddR \Gfullx_0(\one-\Pifullx_0)$;
this agrees with the formula in Remark \ref{rem:sideconditions} by
$\Pifullx_0\ddR=0$, $\ddR\Pifullx_0=0$, $\Pifullx_0\Gfullx_0=0$.
By \ref{item:AlgFullx} in Theorem \ref{thm:Mainx} for $\Gfullx_0$ 
and Remark \ref{rem:sideconditions},
$\Gfullx$ satisfies \ref{item:AlgFullx} 
(with $\Pifullx=\Pifullx_0$)
and the additional identities in the lemma.
The remaining properties in Theorem \ref{thm:Mainx} for $\Gfullx$
are immediate from the corresponding properties for $\Gfullx_0$.\qed
\end{proof}

\appendix

\section[Right inverse of
divergence on symmetric traceless matrices]{Right inverse of \\
divergence on symmetric traceless matrices}\label{ap:div2}
Let $\div$ be the divergence operator on $\R^{3}$,
viewed as a map from vector fields to functions.
Denote by $\divst$ the divergence from symmetric traceless 3-by-3 matrices
to vectors, which is given by applying $\div$ to each matrix column.
This map appears for example in the linearized
constraint equations of general relativity about flat space,
see Remark \ref{rem:linconstraints}.
In this section we construct a right inverse of $\divst$
up to ten necessary integrability conditions.
This right inverse has the same properties with regard to support and estimates 
as the right inverse of $\div$ obtained in Theorem \ref{thm:Mainx}.
In fact it 
will be obtained from the right inverse of $\div$ 
via a purely algebraic construction, 
no further analysis is required.
\step
Let $\func = C^\infty_c(\R^3)$ be the smooth
compactly supported functions on $\R^3$ and define
\begin{align*}
\vect & = \func^{\oplus3} && \text{column vectors with 3 entries}\\
\mat & = \vect^{\oplus3} &&\text{3-by-3 matrices}\\
\smash{\matst}&\simeq \func^{\oplus5} && \text{symmetric traceless 3-by-3 matrices}
\end{align*}
The operator $\divst$ is part of the elliptic complex 
\begin{align}\label{eq:CCdef}
0 \to 
\vect \xrightarrow{\gradst} 
\matst \xrightarrow{\CURL} 
\matst \xrightarrow{\divst}
\vect \to0
\end{align}
where the maps are defined as follows.
Let $\grad:\func\to\vect$, $\curl:\vect\to\vect$, $\div:\vect\to\func$
be the usual gradient, curl and divergence. 
We also denote
\begin{align}\label{eq:gradcurldivonmatrices}
\begin{aligned}
\grad: \vect&\to\mat 
&
v&\mapsto 
\begin{pmatrix}  \grad v_1 & \grad v_2 & \grad v_3\end{pmatrix} \\
\curl: \mat&\to\mat 
&
m&\mapsto 
\begin{pmatrix}
\curl{m}_1 &
\curl{m}_2 &
\curl{m}_3
\end{pmatrix} \\
\div: \mat&\to\vect
&
m&\mapsto
\begin{pmatrix}
\div m_{1}&
\div m_{2}&
\div m_{3}
\end{pmatrix}^{\smash{T}}
\end{aligned}
\end{align}
where $v_i\in\func$ is the $i$th entry of the vector $v$ 
and $m_i\in\vect$ is the $i$th column of the matrix $m$.
Let $\symtr:\mat\to\matst$,
$m\mapsto \frac12(m + m^T) - \frac13 \tr(m) \one_3$
with $\one_3$ the 3-by-3 identity matrix.
Further let $\Mincl:\matst\to\mat$ be the inclusion.
Set
\begin{align*}
\qquad\qquad\qquad\qquad\qquad\qquad\gradst &= \symtr\circ \grad :& \vect &\to \matst\qquad\qquad\qquad\qquad\qquad\qquad\\
\curlst &= \symtr\circ \curl\circ \Mincl :& \matst &\to \matst\\
\divst &= \div\circ \Mincl :& \matst &\to \vect
\end{align*}
which in particular defines the first arrow in \eqref{eq:CCdef}.
The second arrow is 
\[ 
\CURL\;=\; \curlst\circ\gradst \circ \divst 
-\tfrac12  \Delta\circ \curlst 
\;:\;\;
\matst\to\matst
\]
where $\Delta=\p_{x^1}^2+\p_{x^2}^2+\p_{x^3}^2$ is
applied entrywise.
This is the up-to-normalization unique constant coefficient
homogeneous third order differential operator
that is $\SO(3)$-equivariant and satisfies
$\CURL\circ\gradst=0$ and $\divst\circ\CURL = 0$.

We denote \eqref{eq:CCdef} by $(\CC,\DD)$,
i.e., $\CC =\oplus_{k=0}^3 \CC^k$ with $\CC^0,\dots,\CC^3$ 
the four spaces in \eqref{eq:CCdef}, 
and $\DD:\CC\to\CC$ is the differential given by the arrows
in \eqref{eq:CCdef}.
\step
As a corollary of the next theorem, 
the homology of $(\CC,\DD)$ is isomorphic to 
$0\oplus0\oplus0\oplus\R^{10}$.
This means for example that the kernel of $\divst$
is equal to the image of $\CURL$,
and that the cokernel of $\divst$ is ten-dimensional.
To prepare for the theorem, we give a simple argument which shows that 
the cokernel of $\divst$ is at least ten-dimensional.
Let $K_1,\dots,K_{10}$ be the conformal Killing fields
\begin{align}\label{eq:confR3}
\begin{aligned}
&e_i \\
&\textstyle\sum_{j=1}^3 x^j e_j\\
&\textstyle\sum_{j,k=1}^{3}\eps_{ijk} x^j e_k \\
&x^i \textstyle\sum_{j=1}^3 x^j e_j - \frac12\|x\|^2 e_i
\end{aligned}
\end{align}
with $i=1,2,3$ and $\eps_{ijk}$ the Levi-Civita symbol,
and where $e_i$ is the $i$th unit vector in $\R^3$.
Then for all $m\in\matst$ and $\ell=1\dots10$,
\begin{equation}\label{eq:divKKdiv}
\divst(m)\cdot K_{\ell} = \div(m\cdot K_\ell)
\end{equation}
Thus by Stokes' theorem the maps
$\vect\to\R$, 
$v\mapsto \int_{\R^3} (K_\ell \cdot v ) \dx$
descend to ten linear functionals on $\coker(\smash{\divst})$,
which are easily seen to be linearly independent.
Here $\dx=dx^1dx^2dx^3$ is the Lebesgue measure.
\begin{theorem}\label{thm:DIV}
For every real number $R>0$ there exists an $\R$-linear map 
\[ 
\GDD:\;  \CC^{k}\to \CC^{k-1}
\]
defined for every integer $k$, with the following properties:
\begin{enumerate}[label=(a\arabic*)]
\item \label{item:AlgGDIV}
The map $\PiDD:\CC^{k}\to \CC^{k}$ defined by 
$$\one - \PiDD \;=\; \DD \GDD + \GDD\DD$$
satisfies $\PiDD^2=\PiDD$ and $\PiDD\DD =0$ and $\DD\PiDD=0$.
In particular, $\PiDD$ is a projection onto 
a complement of the image of $\DD$ in the kernel of $\DD$. 
It is equivalently given as follows:
There exist $\nu_{1},\dots,\nu_{10}\in \vect$ 
with $\nu_\ell|_{\|x\|\ge R}=0$ and
$\int_{\R^{3}}(K_\ell \cdot \nu_{\ell'}) \dx = \delta_{\ell\ell'}$
for $\ell,\ell'=1\dots10$,
such that for all $w\in \CC^{k}$,
\begin{align}\label{eq:PIDDformula}
\PiDD w
&=
\begin{cases}
\textstyle\sum_{\ell=1}^{10}\nu_\ell\int_{\R^{3}}(K_\ell \cdot w) \dx & \text{if $k=3$}\\
0 &\text{if $k=0,1,2$}
\end{cases}
\end{align}

\item \label{item:SuppGDIV}
For all $r\ge R$ and $w\in \CC$ one has:
$w|_{\|x\|\ge r}=0 
\;\Rightarrow\;
(\GDD w)|_{\|x\|\ge r}=0$.

\item \label{item:EstGDIV}
For all $k=1,2,3$ and
$s,\delta\in\R$ with 
$\delta>k-\frac32+2$ when $k=2,3$
and $\delta>k-\frac32$ when $k=1$, 
there exists $C>0$ such that for all 
$w\in\CC^k$,
\begin{align}\label{eq:GDDestimates}
\begin{aligned}
\|\GDD w\|_{\jap{x}^{-\delta+1}\Hx^{s+1}(\R^{3})}
&\;\le\;
C
\|w\|_{\jap{x}^{-\delta}\Hx^{s}(\R^{3})}& \text{if }k&=1,3\\
\|\GDD w\|_{\jap{x}^{-\delta+3}\Hx^{s+3}(\R^{3})}
&\;\le\;
C
\|w\|_{\jap{x}^{-\delta}\Hx^{s}(\R^{3})}& \text{if }k&=2
\end{aligned}
\end{align}
where we use the norms \eqref{eq:HxnormNEW} componentwise.
\item \label{item:ComEstGDIV}
For all $k,s,\delta$ as in \ref{item:EstGDIV},
and all $i=1,2,3$ and $\ell\in\Nzero$, 
there exists $C>0$ such that for all $w\in\CC^k$,
\begin{align*}
\begin{aligned}
\left\|\left[\GDD,\jap{x}^{1-\ell}\p_{x^i}\right] w\right\|_{\jap{x}^{-\delta+1-\ell}\Hx^{s+1}(\R^{3})}
&\;\le\;
C
\|w\|_{\jap{x}^{-\delta}\Hx^{s}(\R^{3})}& \text{if }k&=1,3\\
\left\|\left[\GDD,\jap{x}^{1-\ell}\p_{x^i}\right] w\right\|_{\jap{x}^{-\delta+3-\ell}\Hx^{s+3}(\R^{3})}
&\;\le\;
C
\|w\|_{\jap{x}^{-\delta}\Hx^{s}(\R^{3})}& \text{if }k&=2
\end{aligned}
\end{align*}
\end{enumerate}
\end{theorem}
Note that for $v\in\CC^3=\vect$ we have $v-\PiDD v = \smash{\divst}(\GDD v)$,
so $\GDD$ yields in particular a right inverse of \smash{$\divst$}
up to the necessary integrability conditions \eqref{eq:PIDDformula}.

We outline the proof.
The main task is to 
construct a homotopy equivalence 
\begin{equation}\label{eq:contrdtenddivNEW}
\begin{tikzpicture}[baseline=(current  bounding  box.center)]
  \matrix (m) [matrix of math nodes, column sep = 20mm, minimum width = 4mm ]
  {
    (\CC,\DD) & (\smash{\Wten},\smash{\dten}) \\
  };
  \path[-stealth]
    (m-1-1) edge [right hook-stealth,transform canvas={yshift=1.2mm}] node [above] {\footnotesize $\DIVtodr$} (m-1-2)
            edge [out=180+13,in=180-13,min distance=5mm] node
                 [left,xshift=-1mm] {
                 \footnotesize $0$  homotopy 
                 } (m-1-1)
    (m-1-2) edge [transform canvas={yshift=-1.2mm}] node [below] {\footnotesize $\drtoDIV$} (m-1-1)
            edge [out=-17,in=+17,min distance=5mm] node
                 [right,xshift=1mm
                 ] {
                 \footnotesize $\HHdr$  homotopy
                 } (m-1-2);
\end{tikzpicture}
\end{equation}
with a complex $(\Wten,\dten)$ given by 
ten copies of the de Rham complex
where the differential is perturbed by a nilpotent term.
Theorem \ref{thm:Mainx} and the homological perturbation
lemma (Lemma \ref{lem:HPL}) yield a contraction
from $(\Wten,\dten)$ to its homology,
\begin{equation}\label{eq:contrdtenhomNEW}
\begin{tikzpicture}[baseline=(current  bounding  box.center)]
  \matrix (m) [matrix of math nodes, column sep = 20mm, minimum width = 4mm ]
  {
    (\smash{\Wten},\smash{\dten}) & (\smash{0\oplus0\oplus0\oplus\R^{10}},0) \\
  };
  \path[-stealth]
    (m-1-1) edge [transform canvas={yshift=1.2mm}] node [above] {\footnotesize $\Pten$} (m-1-2)
            edge [out=180+13,in=180-13,min distance=5mm] node
                 [left,xshift=-1mm] {
                 \footnotesize $\Gten$  homotopy 
                 } (m-1-1)
    (m-1-2) edge [left hook-stealth,transform canvas={yshift=-1.2mm}] node [below] {\footnotesize $\Iten$} (m-1-1)
            edge [out=-6,in=+7,min distance=4mm] node
                 [right,xshift=1mm
                 ] {
                 \footnotesize $0$  homotopy
                 } (m-1-2);
\end{tikzpicture}
\end{equation}
Composing \eqref{eq:contrdtenddivNEW} and 
\eqref{eq:contrdtenhomNEW} (using Lemma \ref{lem:compHE}) yields a contraction from $(\CC,\DD)$ to the homology.
The resulting homotopy on $\CC$ will be the map $\GDD$.
\step
We now construct \eqref{eq:contrdtenddivNEW}.
Let $(\WdR,\WddR)$ be the complex
\begin{align*}
0 \to 
\func \xrightarrow{\grad} 
\vect \xrightarrow{\curl} 
\vect \xrightarrow{\div}
\func \to0
\end{align*}
That is, $\WdR=\oplus_{k=0}^3\WdR^k$ is given by 
the four spaces, and $\WddR$ is given by the arrows.
This is the de Rham complex $(\Omega_c(\R^{3}),\ddR)$ written with respect
to the standard $dx$-basis.
Let $\Wten= \WdR^{\oplus10}$ be the direct sum of ten copies of $\WdR$,
with grading given by 
$\Wten=\oplus_{k=0}^3\Wten^k$ where $\Wten^k = (\WdR^k)^{\oplus10}$.
We decompose 
\begin{align}\label{eq:Wdecomp}
\Wten
\;=\; 
\Wten_0
\oplus
\Wten_1
\oplus
\Wten_2
\end{align}
where 
$\Wten_0=\WdR^{\oplus3}$ and 
$\Wten_1=\WdR\oplus\WdR^{\oplus3}$ and
$\Wten_2=\WdR^{\oplus3}$.
Note that 
\begin{align}\label{eq:Wtenid}
\begin{aligned}
\Wten_0&\;=\; \vect \oplus \mat\oplus\mat\oplus \vect \\
\Wten_1&\;=\; (\func\oplus\vect) \oplus (\vect\oplus\mat)\oplus(\vect\oplus\mat)\oplus(\func\oplus\vect)\\
\Wten_2&\;=\; \vect \oplus \mat\oplus\mat\oplus \vect
\end{aligned}
\end{align}
where the four summands on the right correspond to 
$\Wten_{i}^{0}\oplus\Wten_{i}^{1}\oplus\Wten_{i}^{2}\oplus\Wten_{i}^{3}$
for $i=0,1,2$.
For example $\Wten_0^0=\func^{\oplus3}=\vect$ are column vectors,
and $\Wten_0^1=\vect^{\oplus3}=\mat$ are 3-by-3 matrices,
where the three summands of $\vect^{\oplus3}$ are the matrix columns.

We now define a differential 
\begin{equation}\label{eq:dten}
\dten \;=\; \WddR^{\oplus10} + \deltaten \;:\;\;\Wten\to\Wten
\end{equation}
Note that the first term acts separately on each summand of \eqref{eq:Wdecomp}.
For example the action on $\Wten_0$ is
\smash{$
\vect \xrightarrow{\grad} 
\mat \xrightarrow{\curl} 
\mat \xrightarrow{\div}
\vect$,
using \eqref{eq:Wtenid}
and \eqref{eq:gradcurldivonmatrices}.}
The second term is, relative to \eqref{eq:Wdecomp}, the 3-by-3 block matrix
\begin{align}\label{eq:dtendef}
\deltaten = 
\begin{pmatrix}
0 & 0 & 0 \\
\deltaten_{10} & 0 & 0\\
0 & \deltaten_{21} & 0
\end{pmatrix}
\qquad\qquad
\begin{aligned}
\deltaten_{10}: \Wten_0\to\Wten_1\\
\deltaten_{21}: \Wten_1\to\Wten_2
\end{aligned}
\end{align}
The two entries are given as follows.
Let
$\tovec:\mat\to\vect$,
$\tovec(m)_{i} = \textstyle\sum_{j,k=1}^3\eps_{ijk}m_{jk}$
and 
$\tomat:\vect\to\mat$,
$\tomat(v)_{ij} = \textstyle\sum_{k=1}^3\eps_{ijk} v_k$.
Using \eqref{eq:Wtenid},
for $v_0\oplus m_1\oplus m_2\oplus v_3 \in\Wten_0$,
\begin{align}%
\deltaten_{10}(v_0)&= v_0\oplus (-\tomat(v_0))\nonumber\\
\deltaten_{10}(m_1)&= (-\tovec(m_1))\oplus (-m_1^T+\tr(m_1)\one_3)\nonumber\\
\deltaten_{10}(m_2)&= \tr(m_2)\oplus \tovec(m_2)\nonumber\\
\deltaten_{10}(v_3)&= 0\nonumber
\intertext{
and for $(f_0\oplus v_0)\oplus (v_1\oplus m_1)\oplus (v_2\oplus m_2)\oplus (f_3\oplus v_3) \in\Wten_1$,}
\deltaten_{21}(f_0\oplus v_0)&= f_0\one_3 - \tomat(v_0) \nonumber\\
\deltaten_{21}(v_1\oplus m_1)&= \tomat(v_1) - m_1^T+\tr(m_1)\one_3  \nonumber\\
\deltaten_{21}(v_2\oplus m_2)&= v_2 + \tovec(m_2) \nonumber\\
\deltaten_{21}(f_3\oplus v_3)&= 0  \label{eq:dtensub}
\end{align}
It is a routine calculation to check that \eqref{eq:dten} is indeed a differential,
$\dten^2=0$.
\begin{lemma}\label{lem:retract}
The triple
$$\big(\,\drtoDIV:\Wten\to\CC,\,
\DIVtodr:\CC\to\Wten,\,\HHdr:\Wten\to\Wten\,\big)$$
defined in \eqref{eq:IPHdef} below 
is a contraction (see \eqref{eq:contraction}) with side conditions
(see \eqref{eq:sidecond})
from $(\Wten,\dten)$ to $(\CC,\DD)$, c.f.~\eqref{eq:contrdtenddivNEW}.
That is, 
\begin{align}\label{eq:ipgcontr}
\begin{aligned}
&\DD\drtoDIV=\drtoDIV\dten,\;\;
\DIVtodr\DD = \dten\DIVtodr \\
&\one-\DIVtodr \drtoDIV = \dten \HHdr+\HHdr\dten\\
&\one-\drtoDIV\DIVtodr = 0\\
&\drtoDIV\HHdr=0,\;\; \HHdr\DIVtodr=0,\;\; \HHdr^2=0
\end{aligned}
\end{align}
Relative to \eqref{eq:Wdecomp}, the maps are 
given by the block matrices 
\begin{align}\label{eq:PIHRretract}
\drtoDIV = \begin{pmatrix}\drtoDIV_0&\drtoDIV_1&\drtoDIV_2  \end{pmatrix}
&&
\DIVtodr = \begin{pmatrix}\DIVtodr_0\\\DIVtodr_1\\\DIVtodr_2  \end{pmatrix}
&&
\HHdr = 
\begin{pmatrix}
0 & \HHdr_{01} & \HHdr_{02} \\
0 & 0 & \HHdr_{12}\\
0 & 0 & 0
\end{pmatrix}
\end{align}
To define the entries we use \eqref{eq:Wtenid}.
For $v_0\oplus m_1\oplus m_2\oplus v_3\in\Wten_0$,
\begingroup
\allowdisplaybreaks
\begin{subequations}\label{eq:IPHdef}
\begin{align}
\drtoDIV_0(v_0) &= 0 \nonumber\\
\drtoDIV_0(m_1) &= 0 \nonumber\\
\drtoDIV_0(m_2) &= \symtr (m_2) \nonumber\\
\drtoDIV_0(v_3) &= v_3 \nonumber 
\intertext{For $(f_0\oplus v_0)\oplus (v_1\oplus m_1)\oplus (v_2\oplus m_2)\oplus (f_3\oplus v_3)\in\Wten_1$,}
\drtoDIV_1(f_0\oplus v_0) &=0  \nonumber\\
\drtoDIV_1(v_1\oplus m_1) &=0 \nonumber\\
\drtoDIV_1(v_2\oplus m_2) &=\symtr (\curl (m_2^T))\nonumber\\
\drtoDIV_1(f_3\oplus v_3) &=-\tfrac13 \grad (f_3) + \tfrac12 \curl (v_3)\nonumber
\intertext{For $v_0\oplus m_1\oplus m_2\oplus v_3\in\Wten_2$,}
\drtoDIV_2(v_0) &= 2v_0 \nonumber\\
\drtoDIV_2(m_1) &= 2\symtr(m_1) \nonumber\\
\drtoDIV_2(m_2) &= 
2(\symtr\circ\grad\circ\div\circ(\symtr+\tfrac{1}{12} \one_3 \tr))(m_2)- \symtr(\Delta m_2)  \nonumber\\
\drtoDIV_2(v_3) &= \tfrac13 \grad(\div (v_3))\label{eq:Pdef}
\intertext{For $v_0\oplus m_1\oplus m_2\oplus v_3\in\CC$,}
\DIVtodr_0(v_0) &= 
\tfrac16 \grad (\div (v_0)) \nonumber\\
\DIVtodr_0(m_1) &= 
((\symtr+\tfrac{1}{12}\one_3\tr)\circ\grad\circ\div) (m_1)-\tfrac12\Delta m_1 \nonumber\\
\DIVtodr_0(m_2) &= m_2 \nonumber\\
\DIVtodr_0(v_3) &= v_3 \nonumber\\[1mm]
\DIVtodr_1(v_0) &= -\tfrac16  \div(v_0) \oplus \tfrac14 \curl(v_0) \nonumber\\
\DIVtodr_1(m_1) &= 0\oplus \tfrac12\curl(m_1)^T\nonumber\\
\DIVtodr_1(m_2) &= 0 \nonumber\\
\DIVtodr_1(v_3) &= 0 \nonumber\\[1mm]
\DIVtodr_2(v_0) &= \tfrac12 v_0 \nonumber\\
\DIVtodr_2(m_1) &= \tfrac12 m_1 \nonumber\\
\DIVtodr_2(m_2) &= 0 \nonumber\\
\DIVtodr_2(v_3) &= 0 \label{eq:Idef}
\intertext{For $(f_0\oplus v_0)\oplus (v_1\oplus m_1)
\oplus(v_0\oplus m_2)\oplus (f_3\oplus v_3)\in\Wten_1$,}
\HHdr_{01}(f_0\oplus v_0)&=0\nonumber\\
\HHdr_{01}(v_1\oplus m_1)&=-\tfrac12\tovec(m_1) \nonumber\\
\HHdr_{01}(v_2\oplus m_2)&= - m_2^T + \tfrac12 \tr(m_2)\one_3 \nonumber\\
\HHdr_{01}(f_3\oplus v_3)&= \tfrac13 f_3 \one_3 + \tfrac12 \tomat(v_3)\nonumber
\intertext{For $v_0\oplus m_1\oplus m_2\oplus v_3\in\Wten_2$,}
\HHdr_{12}(v_0) &= 0 \nonumber\\
\HHdr_{12}(m_1) &= \tfrac13 \tr(m_1) \oplus (-\tfrac12 \tovec(m_1)) \nonumber\\
\HHdr_{12}(m_2) &= \tfrac12\tovec(m_2) \oplus (-\symtr(m_2) +\tfrac16 \tr(m_2)\one_3 )\nonumber\\
\HHdr_{12}(v_3) &= v_3\oplus0 \nonumber\\[1mm]
\HHdr_{02}(v_0) &= 0 \nonumber\\
\HHdr_{02}(m_1) &= 
-\tfrac12 \div (\symtr(m_1)) -\tfrac14 \curl (\tovec(m_1)) \nonumber\\
\HHdr_{02}(m_2) &= 
-\tfrac12(\curl(m_2)^T+\curl(m_2^T)) + \tfrac12 \tomat (\div(m_2)) \nonumber\\
\HHdr_{02}(v_3) &= -\tfrac13 \div(v_3) \one_3
\end{align}
\end{subequations}
\endgroup
\end{lemma}
\begin{proof}
The identities \eqref{eq:ipgcontr} are checked by 
direct, laborious, calculation.\qed
\end{proof}
The structure of \eqref{eq:dtensub} and \eqref{eq:IPHdef} 
that will be relevant for the proof of Theorem \ref{thm:DIV}
is indicated in Figure \ref{fig:retract}.%
\begin{figure}%
\begin{equation*}
\vcenter{\hbox{%
\begin{tikzpicture}[node distance = 9.5mm and 17mm, auto]
  \node (a0) {$\vect$};
  \node (a1) [right=of a0] {$\matst$};
  \node (a2) [right=of a1] {$\matst$};
  \node (a3) [right=of a2] {$\vect$};
  \node (b0) [below=1.2cm of a0] {$\vect$};
  \node (b1) [right=of b0] {$\mat$};
  \node (b2) [right=of b1] {$\mat$};
  \node (b3) [right=of b2] {$\vect$};
  \node (c0) [below=0.6cm of b0] {$\func\oplus\vect$};
  \node (c1) [below=0.6cm of b1] {$\vect\oplus\mat$};
  \node (c2) [below=0.6cm of b2] {$\vect\oplus\mat$};
  \node (c3) [below=0.6cm of b3] {$\func\oplus\vect$};
  \node (d0) [below=0.6cm of c0] {$\vect$};
  \node (d1) [right=of d0] {$\mat$};
  \node (d2) [right=of d1] {$\mat$};
  \node (d3) [right=of d2] {$\vect$};
\draw[->] (a0) to node {\footnotesize $\gradst$} (a1);
\draw[->] (a1) to node {\footnotesize $\CURL$} (a2);
\draw[->] (a2) to node {\footnotesize $\divst$} (a3);
\draw[->] (b0) to node {\footnotesize $\grad$} (b1);
\draw[->] (b1) to node {\footnotesize $\curl$} (b2);
\draw[->] (b2) to node {\footnotesize $\div$} (b3);
\draw[->] (c0) to node {\footnotesize $\grad$} (c1);
\draw[->] (c1) to node {\footnotesize $\curl$} (c2);
\draw[->] (c2) to node {\footnotesize $\div$} (c3);
\draw[->] (d0) to node {\footnotesize $\grad$} (d1);
\draw[->] (d1) to node {\footnotesize $\curl$} (d2);
\draw[->] (d2) to node {\footnotesize $\div$} (d3);
\draw[->,Lcol,dashed] (b0) to node[above=-0.07cm] {\footnotesize $1$} (c1);
\draw[->,Lcol,dashed] (b1) to node[above=-0.07cm] {\footnotesize $1$} (c2);
\draw[->,Lcol,dashed] (b2) to node[above=-0.07cm] {\footnotesize $1$} (c3);
\draw[->,Lcol,dashed] (c0) to node[above=-0.07cm] {\footnotesize $1$} (d1);
\draw[->,Lcol,dashed] (c1) to node[above=-0.07cm] {\footnotesize $1$} (d2);
\draw[->,Lcol,dashed] (c2) to node[above=-0.07cm] {\footnotesize $1$} (d3);
\draw[->,bend right,Icol] (a0) to node[at end,xshift=-2mm,yshift=1mm] {\footnotesize $\p^2$} (b0);
\draw[->,bend right,Icol] (a0) to node[at end,xshift=-2mm,yshift=1mm] {\footnotesize $\p$} (c0);
\draw[->,bend right,Icol] (a0) to node[at end,xshift=-2mm,yshift=1mm] {\footnotesize $1$} (d0);
\draw[->,bend right,Icol] (a1) to node[at end,xshift=-2mm,yshift=1mm] {\footnotesize $\p^2$} (b1);
\draw[->,bend right,Icol] (a1) to node[at end,xshift=-2mm,yshift=1mm] {\footnotesize $\p$} (c1);
\draw[->,bend right,Icol] (a1) to node[at end,xshift=-2mm,yshift=1mm] {\footnotesize $1$} (d1);
\draw[->,bend right,Icol] (a2) to node[at end,xshift=-2mm,yshift=1mm] {\footnotesize $1$} (b2);
\draw[->,bend right,Icol] (a3) to node[at end,xshift=-2mm,yshift=1mm] {\footnotesize $1$} (b3);
\draw[->,bend right,Pcol] (d0) to node[at start,xshift=2mm,yshift=1mm] {\footnotesize $1$} (a0);
\draw[->,bend right,Pcol] (d1) to node[at start,xshift=2mm,yshift=1mm] {\footnotesize $1$} (a1);
\draw[->,bend right,Pcol] (b2) to node[at start,xshift=2mm,yshift=1mm] {\footnotesize $1$} (a2);
\draw[->,bend right,Pcol] (c2) to node[at start,xshift=2mm,yshift=1mm] {\footnotesize $\p$} (a2);
\draw[->,bend right,Pcol] (d2) to node[at start,xshift=2mm,yshift=1mm] {\footnotesize $\p^2$} (a2);
\draw[->,bend right,Pcol] (b3) to node[at start,xshift=2mm,yshift=1mm] {\footnotesize $1$} (a3);
\draw[->,bend right,Pcol] (c3) to node[at start,xshift=2mm,yshift=1mm] {\footnotesize $\p$} (a3);
\draw[->,bend right,Pcol] (d3) to node[at start,xshift=2mm,yshift=1mm] {\footnotesize $\p^2$} (a3);
\end{tikzpicture}}}
\end{equation*}
\captionsetup{width=115mm}
\caption{%
The row at the top is the complex $(\CC,\DD)$,
and the three rows at the bottom are the complex $(\Wten,\dten)$,
where the dashed arrows indicate $\deltaten$.
The downwards arrows indicate 
the nonzero components of $\smash{\DIVtodr}$, 
and the upwards arrows the nonzero components of $\drtoDIV$.
The labels $1,\p,\p^2$ indicate whether the maps
are zeroth, first, second order 
constant coefficient
homogeneous differential operators.
The homotopy $\HHdr$ is not shown.}
\label{fig:retract}
\end{figure}%
\begin{remark}
To find the differential \smash{$\WddR^{\oplus10}+\deltaten$}
and the contraction $(\drtoDIV,\DIVtodr,\HHdr)$
we made an ansatz where we imposed 
that $\deltaten$ and $\HHdr$ are strictly lower respectively
upper triangular relative to \eqref{eq:Wdecomp},
and that all maps are constant coefficient homogeneous differential operators,
$\SO(3)$-equivariant, and have suitable homogeneity.
Then a solution of \smash{$(\WddR^{\oplus10}+\deltaten)^2=0$} and 
\eqref{eq:ipgcontr} exists and is essentially unique.
It can conveniently be obtained using computer algebra.
\end{remark}
\begin{proof}[of Theorem \ref{thm:DIV}]
Let $\GdR$, $\PidR$ be as in Theorem \ref{thm:Mainx}
applied to $n=3$ and the radius $R$. 
We view $\GdR$, $\PidR$ as maps $\WdR\to\WdR$,
using the standard $dx$-basis.
By Corollary \ref{cor:thm2} we can assume that 
\begin{align}\label{eq:sideconcr}
\PidR\GdR=0 && \GdR\PidR=0 && \GdR^2=0
\end{align}
Let $\HWdr = 0\oplus 0\oplus 0\oplus \R$
and fix a factorization $\PidR=\IdR\PdR$
with maps $\PdR:\WdR\to\HWdr$ and $\IdR:\HWdr\to\WdR$.
Then by \ref{item:AlgFullx} in Theorem \ref{thm:Mainx},
the triple $(\PdR,\IdR,\GdR)$ is a contraction
from $(\WdR,\WddR)$ to $(\HWdr,0)$,
and by \eqref{eq:sideconcr} it satisfies the side conditions
\eqref{eq:sidecond}.
Then 
$(\PdR^{\oplus10},\IdR^{\oplus10},\GdR^{\oplus10})$
is a contraction with side conditions from $(\Wten,\WddR^{\oplus10})$ 
to $(\HWdr^{\oplus10},0)$. 
Lemma \ref{lem:HPL} and 
the fact that $\dten=\WddR^{\oplus10}+\deltaten$
is a differential with 
nilpotent $\deltaten \GdR^{\oplus10}$ (its cube vanishes)
then imply that the triple
\begin{equation}\label{eq:contrten}
\big(
\Pten = \PdR^{\oplus10}(\one + \deltaten \GdR^{\oplus10})^{-1},\ 
\Iten = \IdR^{\oplus10},\ 
\Gten = \GdR^{\oplus10}(\one + \deltaten \GdR^{\oplus10})^{-1}\,
\big)
\end{equation}
is a contraction with side conditions from $(\smash{\Wten},\smash{\dten})$
to $(\smash{\HWdr^{\oplus10}},\Pten\smash{\dten}\Iten=0)$, c.f.~\eqref{eq:contrdtenhomNEW}.
Note that we used
$\Iten=(\one + \GdR^{\oplus10}\deltaten)^{-1} \IdR^{\oplus10}=\IdR^{\oplus10}$.

Recall the contraction $(\drtoDIV,\DIVtodr,\HHdr)$ in Lemma \ref{lem:retract},
and note that $(\DIVtodr,\drtoDIV,0,\HHdr)$ is a homotopy equivalence
(see \eqref{eq:HE})
between $(\CC,\DD)$ and $(\smash{\Wten},\smash{\dten})$, c.f.~\eqref{eq:contrdtenddivNEW}.
We compose this with the contraction \eqref{eq:contrten},
which by Lemma \ref{lem:compHE} yields the contraction 
\begin{align}\label{eq:ContractionDIV}
\big(\,
\PDD = \Pten \DIVtodr,\ 
\IDD = \drtoDIV \Iten,\ 
\GDD = \drtoDIV\Gten\DIVtodr
\,\big)
\end{align}
from $(\CC,\DD)$ to $(\HWdr^{\oplus10},0)$. 
In particular, see \eqref{eq:HEcond},
\begin{subequations}\label{eq:allDIV}
\begin{align}
&\text{$\IDD,\PDD$ are chain maps}\label{eq:IPchainDIV}\\
&\one-\IDD\PDD=\DD\GDD+\GDD\DD \label{eq:IPDIV}\\
&\one-\PDD\IDD=0\label{eq:PIDIV}
\end{align}
\end{subequations}
We check that $\GDD$ has the properties stated in Theorem \ref{thm:DIV}.

\ref{item:AlgGDIV}: 
By \eqref{eq:IPDIV} we have $\PiDD=\IDD\PDD$,
which is a projection by \eqref{eq:PIDIV}, and satisfies
$\PiDD\DD=0$, $\DD\PiDD=0$ by \eqref{eq:IPchainDIV}. 
We check \eqref{eq:PIDDformula}.
The map $\PiDD$ vanishes in degree $0,1,2$ because $\PidR$ does, 
see \eqref{eq:contrten} and \eqref{eq:ContractionDIV}.
Recall that the homology of $(\CC,\DD)$ in degree 3 is given by $\coker(\divst)$.
By \eqref{eq:allDIV}, the map $\PDD$ induces an isomorphism
$\coker(\divst)\to\R^{10}$.
By \eqref{eq:divKKdiv}, integration against the conformal 
Killing fields $K_1,\dots,K_{10}$
in \eqref{eq:confR3}
induces ten linear functionals on $\coker(\divst)$,
which are easily seen to be linearly independent.
Therefore there is an invertible matrix
$A\in \R^{10\times10}$ such that for all 
$w\in\CC^3=\vect$,
\begin{align}\label{eq:PDDexpl}
\PDD(w) = 
A
\begin{pmatrix}
\int_{\R^{3}}(K_1 \cdot w)\dx\\
\vdots\\
\int_{\R^{3}}(K_{10} \cdot w)\dx
\end{pmatrix}
\end{align}
After changing basis on $\R^{10}$ we can assume that $A=\one_{10}$.
Now consider $\IDD$.
By \eqref{eq:PIDIV} this is injective.
Further all elements in its image vanish on $\|x\|\ge R$.
This follows from $\IDD = \drtoDIV \IdR^{\oplus10}$ 
by \eqref{eq:contrten} and \eqref{eq:ContractionDIV}, 
the fact that $\drtoDIV$ is local, see \eqref{eq:Pdef},
and because elements in $\image(\IdR)=\image(\PidR)$ vanish on $\|x\|\ge R$.
Hence there exist linearly independent $\nu_1,\dots,\nu_{10}\in\vect$
that vanish on $\|x\|\ge R$ and such that for all $(c_1,\dots,c_{10})\in\R^{10}$
one has
$\IDD(c_1,\dots,c_{10}) = \textstyle\sum_{\ell=1}^{10}  \nu_\ell c_\ell$.
Together with \eqref{eq:PDDexpl} this implies \eqref{eq:PIDDformula}. 
Note that \eqref{eq:PIDDformula} and $\PiDD^2=\PiDD$ imply 
$\int_{\R^{3}}(K_\ell \cdot \nu_{\ell'}) \dx = \delta_{\ell\ell'}$.

\ref{item:SuppGDIV}:
By \eqref{eq:contrten} and \eqref{eq:ContractionDIV},
\begin{equation}\label{eq:GDDfull}
\GDD = \drtoDIV\GdR^{\oplus10}(\one + \deltaten \GdR^{\oplus10})^{-1}\DIVtodr
\end{equation}
The maps $\deltaten,\drtoDIV,\DIVtodr$ are local, 
see \eqref{eq:dtensub} and \eqref{eq:IPHdef}.
Hence \ref{item:SuppGDIV} follows from 
the property \ref{item:SupportFullx} in Theorem \ref{thm:Mainx} for $\GdR$.

\ref{item:EstGDIV}:
Using \eqref{eq:GDDfull}, \eqref{eq:dtendef} and \eqref{eq:PIHRretract}, 
\begin{align}\label{eq:GDDexpl}
\GDD
=
\Big( \drtoDIV_0\ \drtoDIV_1\ \drtoDIV_2  \Big)
\begin{pmatrix}
\GdR^{\oplus3} & 0 & 0\\
-\GdR^{\oplus4}\deltaten_{10}\GdR^{\oplus3} & \GdR^{\oplus4} & 0\\
\GdR^{\oplus3}\deltaten_{21}\GdR^{\oplus4} \deltaten_{10} \GdR^{\oplus3}
& -\GdR^{\oplus3}\deltaten_{21}\GdR^{\oplus4} & \GdR^{\oplus3}
\end{pmatrix}
\begin{pmatrix}\DIVtodr_0\\\DIVtodr_1\\\DIVtodr_2  \end{pmatrix}
\end{align}
Recall that $\deltaten_{10},\deltaten_{21}$ 
are zeroth order constant coefficient homogeneous
differential operators, see \eqref{eq:dtensub}. 
Recall the following standard estimate:
For all $s,\delta\in\R$ and $i=1,2,3$ there exists $C'>0$
such that for all $f\in \func$,
\begin{align}\label{eq:pest}
\|\p_{x^i}f\|_{\jap{x}^{-\delta-1}\Hx^{s-1}(\R^{3})}
\le C'\|f\|_{\jap{x}^{-\delta}\Hx^{s}(\R^{3})}
\end{align}
We check \eqref{eq:GDDestimates}
(the structure of $\DIVtodr,\drtoDIV$ that will be used
is indicated in Figure \ref{fig:retract}):
\begin{itemize}
\item 
$k=3$:
Applying \eqref{eq:GDDexpl} to $w\in\CC^3$ 
we have $\DIVtodr_0(w)=w$, $\DIVtodr_1(w)=0$, $\DIVtodr_2(w)=0$,
see \eqref{eq:Idef}.
Further $\drtoDIV_0$, $\drtoDIV_1$, $\drtoDIV_2$
act as constant coefficient homogeneous zeroth, first respectively second order
differential operators, see \eqref{eq:Pdef}. 
Hence \eqref{eq:GDDestimates} follows by 
repeated application of \eqref{eq:estGx} and \eqref{eq:pest}.
\item 
$k=2$:
Applying \eqref{eq:GDDexpl} to $w\in\CC^2$
we have $\DIVtodr_0(w)=w$, $\DIVtodr_1(w)=0$, $\DIVtodr_2(w)=0$.
Further $\drtoDIV_0$ and $\drtoDIV_1$ vanish, 
and $\drtoDIV_2$ is of order zero.
Hence \eqref{eq:GDDestimates} follows by 
repeated application of \eqref{eq:estGx}.
\item 
$k=1$:
Applying \eqref{eq:GDDexpl} to $w\in\CC^1$, the maps 
$\DIVtodr_0$, $\DIVtodr_1$, $\DIVtodr_2$ 
are constant coefficient homogeneous
second, first respectively zeroth order differential operators.
Further $\drtoDIV_0$ and $\drtoDIV_1$ vanish,
and $\drtoDIV_2$ is of order zero.
Hence \eqref{eq:GDDestimates} follows by 
repeated application of \eqref{eq:estGx}
and \eqref{eq:pest}.
\end{itemize}
The commutator estimates \ref{item:ComEstGDIV} follow analogously, using 
\ref{item:CommutatorFullx} in Theorem \ref{thm:Mainx}.
\qed
\end{proof}

\begin{remark}[Linearized constraints of general relativity]\label{rem:linconstraints}
Let $\matsym\simeq\func^{\oplus6}$ be the symmetric 3-by-3 matrices.
The linearization of the constraint equations 
about flat space in three dimensions
is given by \cite{MaoTao,MaoOhTao} the two operators
$\PGRg:\matsym\to\func$ and $\PGRk:\matsym\to\vect$, 
\begin{align*}
\PGRg(m) &= \div(\div(m))&
\PGRk(m) &=\div(m)
\end{align*}
using \eqref{eq:gradcurldivonmatrices}.
Using Theorems \ref{thm:Mainx}, \ref{thm:DIV} 
it is easy to construct right inverses of $\PGRg$, $\PGRk$:
\begin{itemize}
\item 
Let $\Gfullx:\func\to\vect$ be the map in Theorem \ref{thm:Mainx} with $\Mdim=3$
and restricted to degree 3.
Then $\GGRg:\func\to\matsym$ given by
$$\GGRg(f)_{ij}=\tfrac12\Gfullx(\Gfullx(f)_i)_j+\tfrac12\Gfullx(\Gfullx(f)_j)_i$$
is a right inverse of $\PGRg$ up to 4 necessary integrability conditions:
There exist $\varphi_0,\dots,\varphi_3\in\func$
such that for all $f\in\func$,
\begin{align*}
\PGRg \GGRg(f) &= 
\textstyle f - 
\varphi_0(\int_{\R^3}f\dx) - \sum_{i=1}^3 \varphi_i (\int_{\R^3}x^i f\dx)
\end{align*}
(I thank Piotr Chru\'sciel for discussions related to this operator.) 

\item 
Let $\GDD:\vect\to\matst$ be the map in Theorem \ref{thm:DIV}, 
restricted to degree 3. 
Fix $\bump_4\in\func$ that satisfies
$\int_{\R^3}\bump_4=1$, 
$\int_{\R^3}x^i\bump_4=0$,
$\int_{\R^3}x^ix^j\bump_4=\delta_{ij}$,
and set
$\bump_8 = x^1\bump_4$,
$\bump_9 = x^2\bump_4$,
$\bump_{10} = x^3\bump_4$.
Then $\GGRk:\vect \to \matsym$ given by 
\begin{align*}
\GGRk(v) 
&= 
\GDD\left(v + \tfrac13\grad(\tau(v))
\right) - \tfrac13 \tau(v)\one_3\\
\tau(v)
&=\textstyle
\sum_{\ell=4,8,9,10} \bump_\ell \int_{\R^3} (K_\ell\cdot v) \dx
\end{align*}
is a right inverse of $\PGRk$ up to 6 necessary integrability conditions:
Let $\nu_\ell$ be as in \eqref{eq:PIDDformula},
then for all $v\in\vect$,
\begin{align*}
\PGRk \GGRk(v)
&=
\textstyle
v -
\sum_{\ell=1,2,3,5,6,7} \nu_\ell \int_{\R^3}(K_\ell\cdot v)\dx
\end{align*}
\end{itemize}
The analytic properties in Theorem \ref{thm:Mainx}, \ref{thm:DIV} imply analogous properties for $\GGRg$, $\GGRk$.
\end{remark}


\bigskip
  \footnotesize

  \textsc{Department of Mathematics, Stanford University, Stanford CA 94305, USA}\par\nopagebreak
  \textit{Email address:} \texttt{anuetzi@stanford.edu}

\end{document}